\documentclass[nonblindrev]{informs1} % for review, not blinded
\usepackage[utf8]{inputenc} % allow utf-8 input
\usepackage[T1]{fontenc}    % use 8-bit T1 fonts
\usepackage{url}            % simple URL typesetting
\usepackage{booktabs}       % professional-quality tables
\usepackage{nicefrac}       % compact symbols for 1/2, etc.
\usepackage{microtype}      % microtypography
\usepackage{lipsum}
\usepackage{graphicx}
\usepackage{lipsum}
\usepackage{amsmath}
\usepackage{xcolor}
\usepackage{tikz}
\usepackage{graphicx}
\usepackage{wrapfig,lipsum,booktabs}
\usepackage{float}
\graphicspath{ {./images/} }

\usepackage{thmtools}
\usepackage{enumitem}

\makeatletter
\def\thickhline{%
  \noalign{\ifnum0=`}\fi\hrule \@height \thickarrayrulewidth \futurelet
   \reserved@a\@xthickhline}
\def\@xthickhline{\ifx\reserved@a\thickhline
               \vskip\doublerulesep
               \vskip-\thickarrayrulewidth
             \fi
      \ifnum0=`{\fi}}
\makeatother

\newlength{\thickarrayrulewidth}
\newcolumntype{?}{!{\vrule width 1pt}}

\newlist{thmlist}{enumerate}{1}
\setlist[thmlist]{label=(\roman{thmlisti}), ref=\thethm.(\roman{thmlisti}),noitemsep}

\newcommand{\q}{\mathbf{q}}
\newcommand{\ar}{\mathbf{a}}
\newcommand{\s}{\mathbf{s}}
\newcommand{\un}{\mathbf{u}}

\newcommand{\nsys}{$\mathcal{N}$-system }
\newcommand{\EE}{\mathbb E}
\newcommand{\EEpe}{\mathbb E_{\pi^\epsilon}}
\newcommand{\PPpe}{\mathbb P_{\pi^\epsilon}}

\newcommand{\TT}{\boldsymbol \theta}
\newcommand{\capTT}{\boldsymbol \Theta}
\newcommand{\tilTT}{\Tilde{\boldsymbol \theta}}
\newcommand{\eplim}{\lim_{\epsilon \rightarrow 0}}

\usepackage{tikz}
\usetikzlibrary{shapes.geometric, arrows}

\tikzstyle{startstop} = [rectangle, rounded corners, 
minimum width=3cm, 
minimum height=1cm,
text centered, 
draw=orange!200, 
fill=orange!20]

\tikzstyle{io} = [rectangle, 
minimum width=3cm, 
minimum height=1cm, 
text centered, 
draw=blue!200, 
fill=blue!20]

\tikzstyle{process} = [rectangle, 
minimum width=3cm, 
minimum height=1cm, 
text centered, 
draw=red!200, 
fill=red!30]

\tikzstyle{decision} = [rectangle, 
minimum width=3cm, 
minimum height=0.5cm, 
text centered, 
draw=green!200, 
fill=green!30]
\tikzstyle{arrow} = [thick,->,>=stealth]

\usepackage{natbib}
 \NatBibNumeric
 \def\bibfont{\small}%
 \bibpunct[, ]{[}{]}{,}{n}{}{,}%

\usepackage[colorlinks=true,breaklinks=true,bookmarks=true,urlcolor=blue,
     citecolor=blue,linkcolor=blue,bookmarksopen=false,draft=false]{hyperref}

\def\EMAIL#1{\href{mailto:#1}{#1}}% When hyperref is used, otherwise outcomment 
\def\URL#1{\href{#1}{#1}}         % When hyperref is used, otherwise outcomment 

\TheoremsNumberedThrough     % Preferred (Theorem 1, Lemma 1, Theorem 2)
\EquationsNumberedThrough    % Default: (1), (2), ...
\begin{document}
%%%%%%%%%%%%%%%%

%\setcounter{page}{1}
%\DOI{10.1287/xxxx.0000.0000}%

\RUNAUTHOR{Jhunjhunwala and Maguluri}

\RUNTITLE{Heavy Traffic Joint Queue Length Distribution without Resource Pooling}

\TITLE{Heavy Traffic Joint Queue Length Distribution without Resource Pooling}
%\TITLE{}

\ARTICLEAUTHORS{%
\AUTHOR{Prakirt Raj Jhunjhunwala}
\AFF{Columbia University, \EMAIL{prj2122@columbia.edu}, \URL{}}
\AUTHOR{Siva Theja Maguluri}
\AFF{Georgia Institute of Technology, \EMAIL{siva.theja@gatech.edu}, \URL{}}
} % end of the block

\ABSTRACT{
This paper studies the heavy-traffic joint distribution of queue lengths in two stochastic processing networks (SPN), viz., an input-queued switch operating under the MaxWeight scheduling policy and a two-server parallel server system called the $\mathcal{N}$-system. These two systems serve as representatives of SPNs that do not satisfy the so-called Complete Resource Pooling (CRP) condition, and consequently exhibit a multidimensional State Space Collapse (SSC). Except in special cases, only mean queue lengths of such non-CRP systems is known in the literature. In this paper, we develop the Transform method to study the joint distribution of queue lengths in non-CRP systems. The key challenge is in solving an implicit functional equation involving the Laplace transform of the heavy-traffic limiting distribution. For the \nsys and a special case of an input-queued switch involving only three queues, we obtain the exact limiting heavy-traffic joint distribution in terms of a linear combination of two iid exponentials. For the general $n\times n$ input-queued switch that has $n^2$ queues, under a conjecture on uniqueness of the solution of the functional equation, we obtain an exact joint distribution of the heavy-traffic limiting queue-lengths in terms of a non-linear combination of $2n$ iid exponentials. 
}

%\KEYWORDS{deterministic inventory theory; infinite linear programming duality; 
%  existence of optimal policies; semi-Markov decision process; cyclic schedule}
%\MSCCLASS{Primary: 90B05; secondary: 90C40, 90C90}
%\ORMSCLASS{Primary: Inventory/production: deterministic multi-item;
%  secondary: dynamic programming/optimal control: deterministic 
%  semi-Markov; programming: infinite dimensional}
%\HISTORY{Received November 20, 2003; revised March 8, 2004, and March 26, 2004.}

% \MSCCLASS{}
% \ORMSCLASS{Primary: ; secondary: }
% \HISTORY{}

\maketitle
%%%%%%%%%%%%%%%%%%%%%%%%%%%%%%%%%%%%%%%%%%%%%%%%%%%%%%%%%%%%%%%%%%%%%%

\section{Introduction}
Stochastic Processing Networks (SPNs) \cite{williams_survey_SPN} are ubiquitous in engineering with applications in manufacturing, telecommunications, transportation, computer systems, etc. A general stochastic processing network consists of jobs or packets that compete for  limited resources. SPNs in general are modeled using a set of interacting queues. A key performance metric of interest in such systems is delay and queue length. In general, it is not possible to exactly characterize the steady-state queue length behavior in such SPNs. Therefore, SPNs are studied in various asymptotic regimes. In this paper, we consider the heavy-traffic regime where the system is loaded close to its capacity. The queue length in this case, usually blows up to infinity, at a rate of $1/\epsilon$, where $\epsilon$ is the heavy-traffic parameter that denotes the gap between the arrival rate  and the system capacity. Therefore, the objective of interest is typically the asymptotic behavior of the queue length, scaled by $\epsilon$.

Heavy-traffic analysis took root in the work of Kingman \cite{kingman1962_brownian}, who showed that the scaled queue length of a single server queue converges to an exponential random variable in heavy traffic. This was done using diffusion limit approximation and studying the limiting reflected Brownian motion process. 
Since then, a variety of SPNs has been studied in heavy traffic. A key phenomenon in the heavy-traffic regime is that the multi-dimensional queue-length vector typically collapses to a lower-dimensional subspace. This is called the \textit{State Space Collapse} (SSC), and simplifies the analysis of an SPN. When the so-called Complete Resource Pooling (CRP) condition is satisfied, various SPNs exhibit an SSC to a one-dimension subspace, i.e., a line. In this case, the SPN behaves like a single server queue in heavy traffic, and the limiting distribution of scaled queue lengths converges to an exponential random variable. CRP intuitively means that there is a single bottleneck in the system leading to heavy traffic. A popular example of such a system is the load-balancing system under an algorithm such as join-the-shortest queue \cite{foschini1978basic}. 

However, several SPNs that arise in practice do not satisfy the CRP condition, and the SSC occurs to a multi-dimensional subspace. Despite special efforts, except in special cases, the classical diffusion limit approach failed to characterize the heavy-traffic steady state queue length behaviour.
% Except in special cases, the classical diffusion limit approach fails in this setting. 
Recent work \cite{maguluri2016heavy, hurtado2022heavy} developed the drift method and used it to characterize the mean of the (weighted) sum of the queue lengths in such systems under great generality. However, it was shown in \cite{hurtado2022heavy} that the drift method is insufficient to even obtain the individual mean queue lengths. Going beyond the mean queue lengths, the key question we focus on in this paper is: \textit{What is the heavy traffic \textbf{joint distribution} of queue lengths in an SPN when the CRP condition is not satisfied?} 
% We answer this question in this paper by studying two systems that have served as representatives of non-CRP systems in the literature.

In this work, we consider two well-studied stochastic processing networks (SPN), viz., an input-queued switch policy and a two-server parallel server system called the $\mathcal{N}$-system. For both the system, we characterize the heavy traffic joint distribution by establishing an implicit functional equation, and also provide the solution to the functional equation under certain condition on system parameters. Our main contribution in this work are listed below. 

% , and
% To overcome the limitations of the drift method, a novel transform method based on Laplace or Fourier transforms was developed in \cite{hurtado2020transform}. However, the analysis and the results in \cite{hurtado2020transform} were limited to CRP systems. In this paper, we develop the transform method for non-CRP systems, and 
% use it to study two systems that have served as representatives of non-CRP systems in the literature.

\subsection{Main Contribution}

The main contribution of this paper is in finding the joint distribution of the Input-queued switch. Historically, developments on input-queued switch have served as guide posts to study more general SPNs \cite{shah_switch_open, hurtado2022heavy}. For the case of Input-Queued Switch, that models a data center network, finding `the complete joint distribution of queue length vector in heavy traffic' was posed as an open problem in 
\cite{shah_switch_open}. As a result, there has been extensive research focused on characterizing the behavior of queue lengths under heavy traffic conditions in an Input-queued switch.

Noteworthy contributions include \cite{maguluri2016heavy}, where the authors obtained the mean delay of an Input-queued switch operating under MaxWeight scheduling in heavy traffic.  In \cite{kang2012diffusion}, the authors characterized the diffusion approximation for the Input-queued switch under MaxWeight scheduling. Despite these efforts, a closed-form expression for the joint distribution in heavy traffic remains open.

In uniform traffic conditions, and under a certain conjecture, we solve this open problem mentioned in \cite{shah_switch_open} by providing a closed form expression for the heavy traffic joint distribution in terms of non-linear combination of i.i.d. exponential random variables.

\subsubsection{Input-queued switch} 
Input-queued switch that also models a data center networks is a discrete-time queueing systems that has served as a representative of non-CRP systems in the literature. In Section \ref{sec: switch}, we consider the input-queued switch with $n$ ports and $n^2$ queues operating under a class of scheduling policies that satisfies SSC (e.g. MaxWeight scheduling or other algorithms studied in \cite{jhunjhunwala2021low}). 

Before presenting the results for a general input-queued switch, in Section  \ref{sec: 3q}, we consider a special case consisting of just three queues, which we call the Three-queue system. The dynamics of the Three-queue system is similar to that of an Input-queued switch, although it has only three-queues. For Three-queue system, the three dimensional queue vector collapses to a two-dimensional subspace in heavy traffic (see Definition \ref{def: 3q_ssc}). We establish the functional equation for the Three-queue system (in Theorem \ref{thm: 3q_functional_eq}) and show that the functional equation has a unique solution (see Lemma \ref{lem: 3q_uniqueness}). 
Using this, we solve the functional equation for the Three-queue system under a special condition on variance of the arrival process, and characterize the heavy-traffic queue-length vector in terms of the linear combinations of two independent exponential random variables, as presented in Theorem \ref{thm: 3q_dist}. 

% For the Three-queue system, we prove that the uniqueness conjecture holds, i.e., the functional equation has a unique solution.

The results for Input-queued switch is presented in a similar manner as that for the Three-queue system. In Section \ref{sec: switch_ssc}, we present definition of state space collapse in case of Input-queued switch. For Input-queued switch, $n^2$-dimensional queue length vector collapses to a $(2n-1)$-dimensional subspace in heavy traffic (see Definition \ref{def: switch_ssc}). Afterwards, in Section \ref{sec: switch_results}, we obtain the implicit functional equation (see Theorem \ref{thm: switch_functional_eq}) for the transform of the limiting queue-length vector. Solving this functional equation is a major challenge.  In particular, the key difficulty is in establishing uniqueness of its solution. In contrast to the Three-queue system, for the Input-queued switch, we \textit{conjecture} (see Conjecture \ref{lem: switch_uniqueness}) that the functional equation has a unique solution. In case of a uniform traffic, we identify one solution of this functional equation and conjecture that the proposed solution is unique. 
Our solution, as presented in Theorem \ref{thm: switch_dist}, for the heavy-traffic joint distribution of the queue lengths in a switch involves a non-linear combination of $2n$ iid exponential random variables. Mathematically, we obtain that the heavy traffic steady state queue length vector is given by
  \begin{equation*}
     \epsilon q_{i+n(j-1)} \stackrel{d}{\rightarrow}  \Upsilon_i + \Upsilon_{n+j} - 2 \Tilde{\Upsilon}, \ \ \forall i,j \in \{1,\dots,n\},
 \end{equation*}
where $\{\Upsilon_1,\dots,\Upsilon_{2n}\}$ are independent exponential random variable and $\Tilde{\Upsilon}  =\displaystyle \min_{1\leq k\leq 2n}  \Upsilon_k$. An implication of our result is that, under uniform traffic,
\begin{align*}
        \epsilon q_{i+n(j-1)} \stackrel{d}{\rightarrow} \Upsilon \sim \begin{cases}
            \text{Exponential} & w.p. \ \ \frac{1}{n}\\
            \text{Erlang-2} & w.p. \ \ 1-\frac{1}{n},
        \end{cases}
    \end{align*}
where $q_{i+n(j-1)}$ is the queue length corresponding to the $i^{th}$ input and $j^{th}$ output.
It is important to highlight that the Input-queued switch system does not fulfill the CRP condition, even in the case of uniform traffic conditions. As such, for analyzing an SPN under uniform traffic conditions, it is necessary to study the system when the CRP condition is not satisfied.

% The mean of the sum of the queue-lengths under the proposed joint distribution matches with the known solution in the literature \cite{maguluri2016heavy}.

\subsubsection{\nsys}
\nsys is a two-server parallel server system operating in continuous time under Poisson arrivals and exponential service times. 
It is one of the simplest parallel server system that preserves much of the complexity of more general models, and so has been extensively studied, albeit only under CRP. In this work, we study the \nsys operating under the MaxWeight policy, when the CRP condition is not satisfied. In this case, the two dimensional state of the system collapses to a two-dimensional cone, and thus, there is no dimensionality reduction. However, the area of the cone is half of the area of the actual state space, and so, we still have a state space collapse (see Proposition \ref{prop: n_sys_ssc}). Finally, in Section \ref{sec: n_sys_dist}, we present the functional equation for \nsys in Theorem \ref{thm: n_sys_mgf_eq}. Afterwards, in Theorem \ref{thm: n_sys_distribution}, we present the heavy traffic joint distribution of the steady-state scaled queue length vector of the \nsys in terms of two independent and exponentially distributed random variables. The details of our results for \nsys are presented in Section \ref{sec: n_sys}. This result illustrates the use of Transform method in both the discrete and continuous time Markovian systems.

A brief summary of our results and differences among the systems considered in this paper is in presented in  Table \ref{table: results}.

\newcolumntype{P}[1]{>{\centering\arraybackslash}p{#1}}
\newcolumntype{N}{@{}m{0pt}@{}}
\begin{table}[h]
\begin{center}
\begin{tabular}{?P{3.6cm}?P{3.9cm} ?P{4.1cm}?P{3.4cm}?}
\thickhline
\rule{0pt}{12pt}
\textbf{System}  & \textbf{Three-queue system}  &  \textbf{ Input-queued switch} &  \textbf{\nsys}
\\ 
\thickhline
\rule{0pt}{12pt}
State space collapse   & $3$-dim $\rightarrow$ $2$-dim & $n^2$-dim $\rightarrow$ $(2n -1)$-dim           & $2$-dim $\rightarrow$ $2$-dim   \\ \hline
Lower dimensional representation &\vspace{-3pt}  $2$-dim and unique  &$2n$-dim with $1$ degree of freedom      &\vspace{-3pt} $2$-dim and unique              \\ \hline
 \vspace{-3pt} Variables in func. eq. &   3, with only $2$ independent ones &      $n^2$, with only $(2n-1)$ independent ones                              &  \vspace{-3pt} $2$  \\ \hline
Uniqueness of solution of func. eq.  & \vspace{-3pt} Yes   & \vspace{-3pt}\textit{ Open}      & \vspace{-3pt} Yes 
\\ \hline
\vspace{-3pt} Limiting joint distribution & Linear combination of two independent exponential r.v.              & \textit{non-linear} combination of $2n$ independent exponential r.v.    & Linear combination of two independent exponential r.v. \\ \thickhline
\end{tabular}
\centering
\caption{Summary of the results presented in this paper. 
}
\label{table: results}
\end{center}
% \vspace{-23pt}
\end{table}

\subsubsection{Methodological contribution}
A major methodological contribution of this paper is to extend the transform method for non-CRP systems. Transform method was first developed in \cite{hurtado2020transform} to study the CRP system.
The key idea in the transform method is to work with exponential Lyapunov functions, which enables one to work with Laplace or Fourier transforms. However, \cite{hurtado2020transform} was limited to CRP systems. 

Building upon the transform method for CRP systems, we use complex exponential as the test function for non-CRP systems. For CRP systems, when the drift of this test function is set to zero in steady-state,  one obtains an exact expression for the Fourier transform of the limiting distribution, as SSC occurs to a line.  Based on this limiting transform, one immediately concludes convergence in distribution to an exponential random variable. For non-CRP systems, when the complex exponential is used as test function, after using the multidimensional SSC, we obtain an \textit{implicit functional equation} in the Laplace transform of the limiting distribution. A major challenge in non-CRP systems is in solving this implicit functional equation. When SSC is into two dimensional subspace, such functional equations are solved in the literature \cite{franceschi2019integral}, using Carleman boundary value problem \cite{litvinchuk1970generalized}. We adopt these results to obtain the limiting distribution under two dimensional SSC.

%Our third major contribution is the mathematical technique by using the complex exponential as the test function in the Lyapunov analysis to obtain the Laplace transform of the heavy traffic distribution of non-CRP systems. Previously, this method can only be applied to CRP systems as seen in \cite{hurtado2020transform}, and now we have a similar framework for analyzing a much larger class of queueing systems. To provide more insights into the mathematical analysis behind the results presented in this paper, we have presented an outline of the proofs for the Three-queue system in Section \ref{sec: 3q_outline}. We also provide a small discussion and some future works in Section \ref{sec: discussion}.

\subsection{Outline of our method}

The analysis presented in this paper is based on the transform method for heavy-traffic analysis that was first developed in \cite{hurtado2020transform}. It is a variant of the drift method, where a complex exponential is chosen as a Lyapunov test function, and its drift is set to zero in steady-state. This leads to working with the Laplace or Fourier transform of the stationary-distribution in the heavy-traffic limit. When the CRP condition is satisfied, one first establishes a one-dimensional SSC. Using this SSC result, setting the drift of the test function to zero, one obtains  an exact expression for the transform of the limiting stationary distribution (i.e., the moment-generating function or the characteristic function). By identifying the limiting MGF with that of an exponential random variable, one concludes convergence in distribution to the exponential. In this paper, we extend this framework to non-CRP systems. 

After first establishing SSC, our framework is then in two steps. The first step is to use the complex exponential as the test function and equate its expected drift to zero in steady-state.
Then, we use the second-order approximation of the complex exponential in terms of the heavy traffic parameter to get the functional equation that characterizes the heavy traffic distribution of the scaled queue length vector. Here we make use of the SSC result. 
To be more specific, due to the SSC, the number of variables in the functional equation matches with the dimension of subspace onto which SSC occurs. 

The second step is to solve the functional equation to get the Laplace transform of the heavy traffic distribution of the steady-state scaled queue length vector. Solving the functional equation, in general, is not easy.  Under some specific conditions on the parameter involved in the functional equation, one could guess the solution and check whether it satisfies the functional equation. If it does, then the solution gives the Laplace transform of the heavy traffic distribution. A crucial step to solve the functional equation is to show that it has a unique solution. This ensures that the guessed solution is the only solution for the functional equation. In this paper, we use the results presented in \cite{franceschi2019integral} to show that if the queueing system has a functional equation in two variables (for example, \nsys and Three-queue system), then there is a unique solution to the functional equation. 
More generally, in the case of $n\times n$ switch, the functional equation has more than two variables, and in this case, we conjecture that the functional equation has a uniques solution.

% The analysis presented in this paper follows three main steps. First, we do the drift analysis of the system using a complex exponential as the Lyapunov function and  Second step is to prove that there is unique distribution that satisfies this functional equation. Then, the third step is just to find a joint distribution that satisfies the functional equation. 
% Guessing the solution to the functional equation under general variance 

\subsection{Related Work}

Using diffusion limit to study the behaviour of a queueing system in heavy traffic was first introduced by Kingman \cite{kingman1962_brownian}, where he studied a single server queue. The phenomenon of state space collapse was used to study the heavy traffic optimality in \cite{foschini1978basic}, where the authors studied the performance of the Join-the-shortest queue policy in a multi-server system. This method was successfully applied to several queueing systems that satisfy the CRP condition \cite{harrison1998heavy, harrison1987brownian, Williams_CRP, stolyar2004maxweight, gamarnik2006validity}. The idea has also been used to study some non-CRP systems, eg., bandwidth sharing network \cite{Weina_bandwidth, kang2009state, zwart_bandwidth_diffusion, yeyaobandwidth2012}. A major drawback of the diffusion limit method is that it involves a certain interchange of limits which is hard to establish. 

The idea behind diffusion limits \cite{gurvich2014diffusion,  Williams_state_space, rei_state_space} is to show a \textit{process level convergence} of the scaled queue length vector to a Reflected Brownian Motion (RBM) \cite{harrison_2013_book, morters2010brownian, uhlenbeck1930theory}. Due to state space collapse, the corresponding RBM lives in a lower dimensional subspace compared to the original state space of the queueing system. 
Next step is to study  the stationary distribution of the obtained RBM process. 
The stationary distribution of an RBM motion can be characterized the Basic Adjoint Relationship (BAR) \cite{dai2011nonnegativity}. Solving the BAR to obtain the stationary distribution is hard in general. But under the skew-symmetry condition \cite{harrison1987multidimensional, williams1987reflected, harrison1987brownian}, one can solve the BAR to show that the stationary distribution of the RBM is given by product-form exponential. The authors in \cite{franceschi2019integral, harrison1978diffusion} attempt to solve the BAR even when the skew-symmetry condition is not satisfied, while others \cite{dai2011reflecting, Franceschi2017asymptotic} use the BAR to study the tail behaviour of the stationary distribution of the RBM. Numerical methods to solve the BAR and obtain the stationary distribution is presented in \cite{dai1991steady, dai1992reflected}.

In addition to diffusion limits method, three different \textit{direct methods} to study the heavy traffic behaviour of a queueing systems have been developed in recent years. A major advantage of these direct methods over the diffusion limit method is that these methods directly work with the stationary distribution of the pre-limit system, and so, do not require the interchange of limits. 
The first direct method, named as the \textit{drift method}, uses a test function and equates its expected drift to zero in steady-state. 
Drift method (introduced in \cite{kingman}) was used in \cite{atilla} to study the moments of weighted queue lengths of a multi server system. A common choice for test functions in drift method is polynomial test functions, which can be used to obtain bounds on the moments of queue length. However, for non-CRP systems, the drift method with polynomial test functions is not enough to obtain bounds on the higher moments of queue lengths \cite{Hurtado-gen-switch-SIGMETRICS}. Transform method \cite{kingman1961_charfunction, hurtado2020transform} is an extension of the drift methods where an exponential test function is used. Second is \textit{BAR method} \cite{braverman_BAR} which studies a continuous time system under general arrivals and services by using carefully constructed exponential functions to handle the jumps. The third method is \textit{Stein's method} \cite{gurvich2014diffusion, braverman2017stein}, which focuses on studying the rate of convergence to the diffusion limit. Among the direct methods, so far, the BAR method and the Stein's method were only used to study the systems that satisfy the CRP condition, while only drift method is used to study the non-CRP system. In this paper, we extend the transform methods by using complex exponential as the test function to study two well-known non-CRP system, i.e., \nsys and Input-queued switch. 

A general model for a parallel server system (including the \nsys) is provided in \cite{rubino2009dynamic}. The Brownian control problem for parallel server systems is presented in \cite{harlop_state_space}, where a linear program in terms of arrival rates and mean service times was presented to define the heavy traffic regime for this system and articulate the condition for complete resource pooling. In \cite{belwil_state_space}, the authors studied a Brownian control problem for \nsys under the CRP condition. They proposed a threshold control policy which is asymptotically optimal in the heavy traffic limit. \nsys with reneging were studied in \cite{tezcan2010dynamic} which shows that under certain conditions on the service speed, a $c\mu$-type greedy policy is asymptotically optimal in the heavy traffic. The focus of most of the existing literature on \nsys is minimizing the cost under CRP condition. More recently, the mean delay of parallel server systems are studied under non-CRP condition \cite{Hurtado-gen-switch-SIGMETRICS} with MaxWeight as the scheduling algorithm. To best of our knowledge, ours is the first work that studies the heavy traffic distribution of \nsys under non-CRP condition.

% In the service times for any job class in \nsys is only server dependent and each job is of equal cost, then  $c\mu$-type greedy policy reduces to a MaxWeight policy where the server chooses the longest queue that it can serve. 

Input-queued switch is one of the most popular queueing system that does not satisfy the CRP condition and as mentioned in \cite{shah2012optimal, williams_survey_SPN}, Input-queued switch serves a guiding principle for design and analysis of scheduling algorithms in general SPNs. The performance and throughput optimality of different scheduling algorithms (including MaxWeight) for Input-queued switch was  studied in \cite{mckeown1995scheduling, mckeown96walrand, 665071}. The holding cost for a generalized switch under CRP condition with MaxWeight as scheduling algorithm was studied in \cite{stolyar2004maxweight}. While the mean delay of Input-queued switch operating under MaxWeight scheduling in heavy traffic was studied using the drift method in the paper \cite{maguluri2016heavy} with some extensions provided in \cite{QUESTA_switch,Hurtado-gen-switch-SIGMETRICS,jhunjhunwala2021low}. The diffusion approximation for Input-queued switch of size $n$ under MaxWeight scheduling was presented in \cite{kang2012diffusion}, where the authors showed the process level convergence of a $(2n-1)$-dimensional workload process to a Semimartingale-RBM.

% In \cite{whitt1971weak}, Whitt observed that a class of priority queues satisfy state-space collapse. 

% In \cite{franceschi2019integral}, the authors present a method where they map the functional equation to a Carleman-Boundary value problem \cite{litvinchuk1970generalized} in two dimension and use the result from Carleman-BVP theory to provide the solution of the functional equation in the form of Cauchy integrals.

\subsection{Basic Notations}

We use $\mathbb R$ to denote the set of real numbers and $\mathbb{C}$ to denote the set of complex numbers. Also, we use $\mathbb R_+$ to denote the set of non-negative real numbers. Similarly, $\mathbb R^d$ and $\mathbb C^d$ denote the set of $d$-dimensional real and complex vectors, respectively. For any complex vector $x\in \mathbb C^d$, $Re(x)$ and $Im(x)$ denote the real part and imaginary part of $x$, respectively.  For any vector $\mathbf{x}$, we use $x_i$ to denote the $i^{th}$ element of $\mathbf{x}$.  The inner product of two vectors $\mathbf{x}$ and $\mathbf{y}$ is defined as $\langle \mathbf{x},\mathbf{y}\rangle = \mathbf{x}^T \Bar{\mathbf{y}}$, where  $\Bar{\mathbf{y}}$ is the complex conjugate of $\mathbf{y}$. If the vectors $\mathbf{x}$ and $\mathbf{y}$ are both real vectors, then $\langle \mathbf{x},\mathbf{y}\rangle$ just represents the dot product of two vectors. The function  $|\mathbf{x}| = \sqrt{\langle \mathbf{x} , \mathbf{x}\rangle }$ denotes the absolute value of $\mathbf{x}$. Further, $\|\cdot\|$ denotes the $\ell_2$-norm of real vector in $\mathbb R^d$.  For any set $A$, $\mathbf{1}_A$ denotes the indicator random variable for set $A$. For any positive natural number $d$, $\mathbf{1}_d$ and $\mathbf{0}_d$ denotes the vector of all ones and vector of all zeros of size $d$ respectively, and $\mathbf{I}_{d}$ denotes the identity matrix of size $d$. 

For any queueing system, we use $\pi_\epsilon$ to denote the heavy traffic distribution of the queue length vector, where $\epsilon$ is the heavy traffic parameter that captures the distance of arrival rate vector from the boundary of capacity region. Suppose $\q$ is the queue length vector that follows the steady-state distribution $\pi_\epsilon$, then $\epsilon \q$ is called the steady state scaled queue length vector.  Note that the steady-state distribution itself depends on $\epsilon$, and a more suitable notation would be $\q^\epsilon$, however, we drop the superscript for convenience. We use the term \textit{heavy traffic distribution} to denote the limiting distribution $\epsilon \q$. For any given system, unless otherwise specified, $\mathbb{E}_{\pi_{\epsilon}}[\cdot]$ denotes the expectation under the steady state distribution of the corresponding system. Under the condition that the Laplace transform of the heavy traffic distribution exists for a given $\boldsymbol \theta$, it is given by $\lim_{\epsilon \rightarrow 0} \EEpe [e^{\epsilon \langle \boldsymbol \theta , \q\rangle}]$.

\section{Input-queued switch}
\label{sec: switch}
In this section, we provide our results on the heavy traffic distribution of the Input-queued switch. 
% An Input-queued switch has a similar model as the three-queue system and the results or analysis presented in this section are also similar to that of the Three-queue system. 
In Section \ref{sec: switch_model}, we present the model for an Input-queued switch and in Section \ref{sec: switch_ssc}, we provide the SSC result for the Input-queued switch. The results regarding the heavy traffic distribution of the Input-queued switch is presented in Section \ref{sec: switch_results}. The results for the Input-queued switch holds under the assumption that a certain conjecture holds (see Conjecture \ref{lem: switch_uniqueness}). In Section \ref{sec: 3q}, we present a simpler system, which we call the Three-queue system, for which the conjecture holds. 

% Please note that in this section, we redefine the notations that we used in the previous sections. 

% In this section, we present the functional equation that describes the joint distribution of heavy traffic steady state scaled queue length of the Input queued switch. 

% \subsection{Preliminaries for Input-queued switch}
% In this section, we present the model (in Section \ref{sec: switch_model}) and the SSC result (in Section \ref{sec: switch_ssc}) for the switch system.

\subsection{Input-queued switch model}
\label{sec: switch_model}
% An Input-queued switch (or just switch) is a device that exchanges data from one channel to another in a data center.
A switch of size $n$ consists of $n$ input ports and $n$ output ports. The message packets flow from input ports to output ports in a time-slotted manner. 
For any time slot $t$, we denote $a_{i+n(j-1)}(t)$ as the number of packets that arrive at input $i$, and to be sent to output port $j$. As there are $n^2$ such input-output pairs, the arrival in any time slot can be represented by an $n^2$ vector  $\ar(t)$. The architecture of the device doesn't allow all the packets to be transferred in one go, which leads to a queue build up on the inputs. We use  $q_{i+n(j-1)}(t)$ (or $\q(t)$ in vector notation) to denote the backlog (or queue) of packets that need to be transferred to the output $j$ from input $i$. We assume that the arrivals are i.i.d. with respect to time $t$ and across the input-output pair $(i,j)$. Also, the arrivals are uniformly bounded by a constant, i.e., there exists $a_{\max}$ such that for $(i,j)$ and $t$, $a_{i+n(j-1)}(t)\leq a_{\max}$. The mean arrival rate vector is given by $\mathbb E[\ar(t)] = \boldsymbol \lambda$ and let $\boldsymbol \sigma^2$ be the co-variance matrix of the arrival vector $\ar(t)$. The independence of the arrivals across the input-output pair gives us that the co-variance matrix $\boldsymbol\sigma^2$ is a diagonal matrix. We say that the arrivals satisfy the \textit{symmetric variance condition} if all the variances are equal, and then $\boldsymbol \sigma^2 =\sigma^2 \mathbf{I}_{n^2} $. Note that the symmetric variance condition is satisfied when the system is in uniform traffic.

The bottlenecks in the system don't allow the transfer of all the packets in the queue simultaneously. Every port can send or receive at most one packet in any time slot, i.e., any input port can send at most one packet in a given time slot. Similarly, any output port can receive at most one packet in a given time slot. The packet transfer can happen only among the connected input-output pairs in that time slot. Therefore, the switch system can be modeled analogously as a complete bipartite graph with $2n $ nodes where $q_{i+n(j-1) }(t)$ denotes the weight of the edge $(i,j)$.
A \textit{schedule} denoted by $\s(t) \in \{0,1\}^{n^2}$ gives the set of input-output pairs that are connected in time slot $t$. The element $s_{i+n(j-1)}(t) =1$ if and only if the pair $(i,j)$ is connected in time slot $t$.  In a complete bipartite graph between the input and output ports, a schedule corresponds to a perfect matching.
It follows that the set of possible schedules $\mathcal{X}$ is given by
\begin{equation*}
    \mathcal{X} = \left\{ \s \in \{0,1\}^{n^2} : \sum_{i=1}^n s_{i+n(j-1)} = 1 \ \forall j, \sum_{j=1}^n s_{i+n(j-1)} = 1 \ \forall i \right\}
\end{equation*}
 The queue length evolution process is given by 
\begin{align*}
    \q(t+1) &= [\q(t) + \ar(t) - \s(t)]^+= \q(t) + \ar(t) - \s(t) + \un(t),
\end{align*}
where operation $[\cdot]^+ = \max(0,\cdot)$ in the above equation is used because the queue length can't be negative. Note that the set $\mathcal{ X}$ corresponds to the set of perfect matchings, and thus we assumed that a schedule is always a complete matching. However, if a queue ($q_{i+n(j-1)}(t)$) is empty, no packets can be transferred even if there is connection between the corresponding input-output pair (i.e., $s_{i+n(j-1)}(t) =1$). Thus, there might be  \textit{unused service}, denoted by $\un(t)$, which arises because it might happen that there is a connection between a input-output pair but there is no packet available to be transferred. For any $i,j \in \{1,2,\dots n\}$, $u_{i+n(j-1)}(t) =1$ if and only if $s_{i+n(j-1)}(t) =1, a_{i+n(j-1)}(t) =0$ and $q_{i+n(j-1)}(t) =0$. This gives us that $q_{i+n(j-1)}(t+1)u_{i+n(j-1)}(t) = 0 $ for all $(i,j)$, which in vector notation is given by $\langle \q(t+1),\un(t) \rangle = 0$.

A scheduling algorithm is then the policy that chooses the schedule in each time slot.
We define the \textit{weight} of the schedule as the sum of the queue lengths that are being served in any time slot. A popular scheduling algorithm for switch system is \textit{MaxWeight} scheduling which chooses the schedule with maximum weight, i.e.,
\begin{equation*}
    \s(t) = \arg\max_{\s \in \mathcal{X}} \langle \q(t) , \s \rangle = \arg\max_{\s \in \mathcal{X}}  \sum_{i=1}^n\sum_{j=1}^n s_{i+n(j-1)} \times q_{i+n(j-1)}(t),
\end{equation*}
where ties are broken arbitrarily as long as the corresponding queue length process   $\{\q(t)\}_{t=0}^\infty$ is Markovian. For example, uniformly at random or according to a static priority policy between the schedules with same weight. The stability of the system in this scenario means that the queue lengths are not going to infinity. More mathematically, we define a system to be stable if the Markov chain $\{\q(t)\}_{t=0}^\infty$ is positive recurrent. The term \textit{capacity region} is used to denote the set of arrival rate vectors for which there exists a scheduling policy for which the system is stable.
The capacity region of the switch system is given by 
\begin{equation*}
	\mathcal C = \Big \{ \mathbf{\boldsymbol \lambda} \in \mathbb{R}_+^{n^2 } : \sum_{i=1}^n \lambda_{i+n(j-1)} < 1 \ \forall j, \ \sum_{j=1}^n \lambda_{i+n(j-1)} < 1 \ \forall i \Big \}.
\end{equation*}
It has been proved in prior literature that MaxWeight scheduling is \textit{throughput optimal} \cite{665071}, i.e., the corresponding Markov chain is stable for any arrival rate vector in $\mathcal{C}$. 
To prove that the Markov chain is stable, one can use the Foster-Lyapunov Theorem by showing that the expected drift of a suitably chosen Lyapunov function is negative as shown in \cite{maguluri2016heavy, 665071}. The two requirements for using the Foster-Lyapunov Theorem, i.e., irreducibility and aperiodicity of the Markov chain can be obtained by using the arguments presented in \cite[Exercise 4.2]{srikant2014communication}. In this paper, we only consider the scheduling algorithms for which the process $\{\q(t)\}_{t=0}^\infty$ forms an irreducible and aperiodic Markov chain.

Let $\mathcal{F}$ denote the part of boundary of the capacity region given by the convex hull of $\mathcal X$, i.e.,
\begin{equation*}
    	\mathcal F = \Big \{ \boldsymbol \nu \in \mathbb{R}_+^{n^2 } : \sum_{i=1}^n \nu_{i+n(j-1)} = 1 \ \forall j, \sum_{j=1}^n \nu_{i+n(j-1)} =1 \ \forall i \Big \}.
\end{equation*}
A switch system is in \textit{heavy traffic} when the arrival rate vector $\boldsymbol \lambda$ approaches the boundary $\mathcal{F}$. For simplicity, we assume that arrival rate vector approaches the boundary along a straight line, i.e., there exists a vector $\boldsymbol \nu \in \mathcal{F}$ and the heavy traffic parameter $\epsilon \in (0,1)$ such that $\boldsymbol\lambda = (1-\epsilon)\boldsymbol \nu$. Further, we assume that none of the arrival rates are zeros, so that $\nu_{\min} = \min_{ij}\nu_{i+n(j-1)} > 0$. 

When the system is stable, that is, the underlying Markov chain is positive recurrent, we use $\pi_\epsilon$ to denote the steady-state distribution for a given $\epsilon$. For convenience, we drop the symbol $t$ to denote the variables in steady state, that is, $\q$ follows the steady state distribution of the queue length process $\{\q(t)\}_{t=0}^\infty$. Further, $\q^+$ denote the state that comes after $\q$, i.e., $\q^+ =[\q+\mathbf a -\mathbf s]^+ = \q + \mathbf a -\mathbf s +{\bf u}$, where $\mathbf a$ follows the same distribution as $\mathbf a(t)$, ${\bf s}$ is the schedule corresponding to the state $\q$, and $\mathbf u$ is the unused service. The steady-state distribution of the random variables $(\q^+,\q,{\bf a},{\bf s},{\bf u})$ depends on the parameter $\epsilon$, but, for the ease of notations, we have avoided using a subscript to denote the steady-state variables. Before presenting the results for Input-queued switch, we present the results for a simpler systems, namely, Three-queue system.

% This is also called the \textit{Completely Saturated Case} \cite{maguluri2016heavy}. 
% Next, we look at the SSC result for the switch system.

\subsection{Three-queue system}
\label{sec: 3q}
% \color{red}
This section presents the heavy traffic distribution for a simpler system, named as the Three-queue system. 
The dynamics of the Three-queue system are similar to that of the Input-queued switch.
We provide the mathematical model for the Three-queue system in Section \ref{sec: 3q_model}.  We also present the state space collapse result of the Three-queue system onto a two-dimensional subspace in the same section. Finally, the results related to the Three-queue system are presented in \ref{sec: 3q_results}.
% A three-queue system can be obtained by taking the arrival rate of one of the queues in a $2 \times 2 $ input-queued switch to be zero. Thus, a three-queue system can be considered to be a simplification of a switch system. More details on a switch system is provided in Section \ref{sec: switch}.
%A Three-queue system can be considered to be a simplification of a $2 \times 2 $ input-queued switch. In a $2\times 2$ switch, there are a total of $4$ queues, one queue for each input-output pair. We obtain the 3-queue system by taking the arrival rate of one of the queue in $2 \times 2 $ queue to be zero.  
% \color{black}

\subsubsection{Model of the Three-queue system}
\label{sec: 3q_model}
We consider a simplification of $2\times 2$ Input-queued switch (consisting of four queues) by picking the arrival rate for the fourth queue to be zero, i.e., $\lambda_4 = 0$. The dynamics for the Three-queue system are similar to that of the IQ-switch in Section \ref{sec: switch_model} with the modification that $q_4(t)=0$ for all values of $t$. For convenience, we redefine the notations used in Section \ref{sec: switch_model} for the Three-queue system. The queue length vector is given by $\q(t) = \big(q_1(t), q_2(t), q_3(t) \big)$ and the arrival vector is given by $\ar(t) = \big( a_1(t),a_2(t),a_3(t)\big) $ with $\mathbb{E}[\ar(t)] = \boldsymbol{\lambda}$ and Var$(\ar(t)) = \boldsymbol{\sigma}^2$, where $\boldsymbol \sigma^2$ is a $3\times 3$ diagonal matrix. The two possible schedules for the Three-queue system are  $(1,0,0)$ and $(0,1,1)$. Further, without loss of generality, we assume that the schedule $(1,0,0)$ is chosen only if $q_1(t)>0$, as otherwise, choosing the schedule $(1,0,0)$ does not provide any service.
% Under this assumption, the schedule $(1,0,0)$ is not chosen unless there are jobs waiting in the queue $q_1(t)$. 
Under this assumption, the unused service for the first queue is always zero, i.e., $u_1(t)=0$ for all $t\geq 0$. 

The capacity region and the corresponding boundary for Three-queue system is given by 
\begin{align*}
    \mathcal{C} &= \Big\{ \boldsymbol \lambda \in \mathbb{R}^3_+ : \lambda_1+ \lambda_2 <1, \lambda_1 + \lambda_3 <1\Big\}, \\
    \mathcal{F} &= \Big\{ \boldsymbol \nu \in \mathbb{R}^3_+ : \nu_1+ \nu_2 =1, \nu_1+ \nu_3 =1\Big\}.
\end{align*}
Note that we have considered $\mathcal{F}$ to be the part of the boundary of the capacity region $\mathcal C$ for which the system does not satisfy complete resource pooling.
We assume that for any $i$, $\lambda_i >0$, otherwise the system can be further simplified. As such, there exists a $\boldsymbol \nu \in \mathcal{F}$ and the \textit{heavy traffic parameter} $\epsilon \in (0,1)$ such that $\mathbf{\boldsymbol \lambda} = (1-\epsilon)\mathbf{\boldsymbol \nu}$. The parameter $\epsilon$ is a measure of the distance of the arrival rate vector from the boundary $\mathcal{F}$. The system approaches heavy traffic as the heavy traffic parameter $\epsilon$ goes to $0$. 
And as $\lambda_i > 0$ for all $i$, $\nu_{\min} \triangleq \min_{i} \nu_{i}  >0.$ 

Similar to that for the Input-queued switch, for the Three-queue system also, we use $\pi_\epsilon$ to denote the steady-state distribution for a given $\epsilon$, under the condition that the system is stable, that is, the underlying Markov chain is positive recurrent. The random variable $\q$ follows the steady state distribution $\pi_\epsilon$, and $\q^+$ denote the state that comes after $\q$, i.e., $\q^+ =[\q+\mathbf a -\mathbf s]^+ = \q + \mathbf a -\mathbf s +{\bf u}$, where $\mathbf a$ follows the same distribution as $\mathbf a(t)$, ${\bf s}$ is the schedule corresponding to the state $\q$, and $\mathbf u$ is the unused service. 
% The steady-state distribution of the random variables $(\q^+,\q,{\bf a},{\bf s},{\bf u})$ depends on the parameter $\epsilon$, but, for the ease of notations, we have avoided using a subscript to denote the steady-state variables. 

% \subsubsection{State Space Collapse for Three-queue system}
% \label{sec: 3q_ssc}
% In this section, we present the definition  state space collapse of the Three-queue system under MaxWeight scheduling in heavy traffic. 
% The SSC for the Three-queue system states that the 3-dimensional state vector can be closely approximated by a two dimensional workload process in heavy traffic.
% We define some matrices as follows:
%     \begin{align*}
%         \mathbf B = \begin{bmatrix}
%                         1 & 1\\
%                         1 &  0\\
%                         0 & 1
%                 \end{bmatrix} && 
%         \mathbf D = \mathbf B^T\mathbf B = \begin{bmatrix}
%                         2 & 1\\
%                         1 &  2
%                 \end{bmatrix}
%     \end{align*}
\subsubsection{State Space Collapse in the Three-queue system}
Consider the subspace $\mathcal{S} \subseteq \mathbb{C}^3$ given by,
\begin{equation*}
    \mathcal{S} = \Big\{ \mathbf{y} \in \mathbb{C}^3 : y_1 = y_2 + y_3 \Big\} = \Big\{ \mathbf{y} \in \mathbb{C}^{3} : \exists \mathbf{r} \in \mathbb{C}^{2} \ s.t. \ \mathbf{y} =\mathbf B \mathbf r \Big\},
\end{equation*}
where \[\mathbf B = \begin{bmatrix}
                        1 & 1 \\
                        1 &  0\\
                        0 & 1
\end{bmatrix}^T.\] Thus, $\mathcal{ S}$ is the space spanned by the columns of $\mathbf B$. We use $\mathcal{S}^\perp$ to denote the space orthogonal to $\mathcal{S}$.
For any vector $\mathbf{x} \in \mathbb{C}^3$, we define $\mathbf{x}_{\|}$ as the projection onto the subspace $\mathcal{S}$ and $\mathbf x_{\perp} = \mathbf{x} - \mathbf{x}_{\|}$. 
% For Three-queue system, the lower dimensional representation of the $\mathbf x_{\|}$ is unique, i.e., there is a unique $\mathbf r$ such that $\mathbf x_{\|} = \mathbf B \mathbf r $. By simple calculation, we can show that $\mathbf r = (\mathbf B^T \mathbf B)^{-1}\mathbf B \mathbf x$. 
For the queue length vector $\mathbf q$, we use $\mathbf w$ to denote $\mathbf w = (\mathbf B^T \mathbf B)^{-1}\mathbf B^T \mathbf q$. Then, $\mathbf q_{\|} = \mathbf B \mathbf w$, and we call $\mathbf w$ the lower dimensional representation of $\q_\|$ as $\mathbf w$ is in $\mathbb C^2$.
 
\begin{definition}[State Space Collapse]
\label{def: 3q_ssc}
For the Three-queue system as defined in Section \ref{sec: 3q_model}
operating under a given scheduling algorithm, we say that the algorithm achieves \textit{State Space Collapse} (SSC) if the underlying Markov chain is positive recurrent and there exists $\theta_0$ and $\epsilon_0$ such that for every $\theta <\theta_0$ and $\epsilon<\epsilon_0$, 
% there exists $\epsilon( \theta) >0$ such that for every $0< \epsilon \leq \epsilon( \theta)$, 
the steady state queue length vector satisfies,
    \begin{equation*}
        \EEpe[e^{\theta \|  \q_{\perp} \|}] < C^\star< \infty.
    \end{equation*}
    % where $\q_{\perp} = \q - \q_{\|} $ and  the expectation is taken under the steady-state distribution.
As a conclusion, for any scheduling policy that achieves SSC, we have that for every $ \theta \in \mathbb R$, $\lim_{\epsilon \rightarrow 0} \EEpe[e^{\epsilon \theta \|  \q_{\perp} \|}]  < \infty.$ Furthermore, for each $r\in \mathbb Z_+$, there exists a $C_r$ independent of $\epsilon$ such that,
    \begin{equation}
    \label{eq: 3q_bound}
        \EEpe \Big [\|\q_{\perp }\|^r \Big] \leq C_r \quad \forall r \geq 1.
    \end{equation}
\end{definition}

For Three-queue system, MaxWeight scheduling chooses the schedule $(1,0,0)$ if $q_1(t)> q_2(t)+q_3(t)$, otherwise it chooses the schedule $(0,1,1)$. For Three-queue system, MaxWeight scheduling achieves SSC according to Definition \ref{def: 3q_ssc}. The proof follows on similar lines as that in \cite[Proposition 2]{maguluri2016heavy}.

\begin{remark}
\label{rem: 3q_qperp}
  For the Three-queue system, when a scheduling algorithm achieves SSC according to Definition \ref{def: 3q_ssc}, the scaled perpendicular component converges to $0$ is distribution, that is, $\epsilon \q_{\perp} \stackrel{d}{\rightarrow} 0$. As such, the distribution of the scaled queue length vector $\epsilon \q$ matches the distribution of the corresponding parallel component $\epsilon \q_{\|}$ as $\epsilon \rightarrow 0$. And so, the limiting distribution of $\epsilon \q_{\|}$ is sufficient to characterize the limiting distribution of $\epsilon \q$.
\end{remark}

We define two sets given by
\begin{align*}
    \boldsymbol \Theta & = \{\boldsymbol \theta \in \mathbb{C}^{3} : \boldsymbol \theta \in \mathcal{S}, \ Re(\mathbf B^T \boldsymbol \theta) \leq  \mathbf 0_{2}\}\\
    &= \{\boldsymbol \theta \in \mathbb{C}^{3} : \theta_1 = \theta_2+\theta_3, Re(2\theta_2+\theta_3)\leq 0, Re(\theta_2+2\theta_3)\leq 0\},\\
    \tilde{ \boldsymbol \Theta } & = \{\tilTT \in \mathbb{C}^{3}: Re(\tilTT) \leq \mathbf 0_{3}\},
\end{align*}
and we use $\TT$ and $\tilTT$ to denote an element in $\capTT$ and $\tilde{ \boldsymbol \Theta }$, respectively. 
For any $\tilTT \in \tilde{ \boldsymbol \Theta }$, we can write $\tilTT = \TT_\| + \TT_\perp$, where $\TT_\| \in \mathcal{S}$ and $\TT_\perp \in \mathcal{S}^\perp$. 
Then, as $\mathcal{S}$ is the span of columns of $\mathbf B$, we have $\mathbf B^T \tilTT = \mathbf B^T \TT_\| $. 
Further, as $Re(\tilTT) \leq \mathbf 0_3$, we have $Re(\mathbf B^T \TT_\|) = Re(\mathbf B^T  \tilTT) \leq \mathbf 0_3$, implying $\TT_\| \in \capTT$. 
As such, $ \tilde{ \boldsymbol \Theta } \subseteq \capTT \oplus \mathcal S^\perp$, and for any $\tilTT \in \tilde{ \boldsymbol \Theta }$, we can write $\tilTT = \TT +\TT_\perp$ such that $\TT\in \capTT$ and $\TT_\perp \in \mathcal{ S}^\perp$. 

\begin{proposition}
\label{prop: 3q_mgf_convergence}
   Consider the Three-queue system as defined in Section \ref{sec: 3q_model} operating under a scheduling algorithm that achieves SSC according to Definition \ref{def: 3q_ssc}. Suppose $\tilTT \in \tilde{ \boldsymbol \Theta } \subseteq \capTT \oplus \mathcal S^\perp$ such that $\tilTT = \TT +\TT_\perp$, where $\TT\in \capTT$ and $\TT_\perp \in \mathcal{ S}^\perp$. Then, 
   \begin{align}
   \label{eq: 3q_q_qpara_mgf}
       \eplim \EEpe \left[ e^{\epsilon \langle \tilTT,\q\rangle}\right] = \eplim \EEpe \left[ e^{\epsilon \langle \TT,\q\rangle}\right].
   \end{align}
   Further, if $\eplim \EEpe \left[ e^{\epsilon \langle \TT,\q\rangle}\right] = \EE\left[ e^{ \langle \mathbf B^T \TT,\mathbf X\rangle}\right]$ for all $\TT\in \capTT$, then 
   % $\eplim \EEpe \left[ e^{\epsilon \langle \tilTT,\q\rangle}\right] = \EE\left[ e^{ \langle \tilTT,\mathbf B \mathbf X\rangle}\right]$ for all $\tilTT \in \tilde{ \boldsymbol \Theta }$ and so,
   \begin{align*}
       \epsilon \q \stackrel{d}{\rightarrow} \mathbf B \mathbf X.
   \end{align*}
\end{proposition}

% \begin{remark}
% According to Proposition \ref{prop: 3q_mgf_convergence}, in order to characterize the limiting distribution of the heavy traffic queue length, it is enough to solve for $\eplim \EEpe \left[ e^{\epsilon \langle \TT,\q\rangle}\right]$ for $\TT\in \capTT$.
% \end{remark}

From Eq. \eqref{eq: 3q_q_qpara_mgf} of Proposition \ref{prop: 3q_mgf_convergence}, we have that, if we can solve for $\eplim \EEpe \left[ e^{\epsilon \langle \TT,\q\rangle}\right]$ for all $\TT \in \capTT$, we can extend the result to solve for $\eplim \EEpe \left[ e^{\epsilon \langle \tilTT,\q\rangle}\right]$ for all $\tilTT \in \tilde{\boldsymbol \Theta}$. And from Lemma \ref{lem: laplace_convergence} (presented later), we know that characterizing $\eplim \EEpe \left[ e^{\epsilon \langle \tilTT,\q\rangle}\right]$ for all $\tilTT \in \tilde{\boldsymbol \Theta}$ is enough to provide the limiting distribution of the scaled queue length vector $\epsilon \q$. Another argument is that, for any $\TT \in \capTT \subseteq \mathcal{S}$, \[\langle \TT, \q \rangle = \langle \TT, \q_{\|} \rangle = \langle \tilTT, \q_{\|} \rangle,\] as $ \langle \TT, \q_{\perp} \rangle =0, \ \forall \TT \in \capTT \subseteq \mathcal{S}$ and $\langle \TT_{\perp},  \q_{\|} \rangle = 0$, and so, \[\eplim \EEpe \left[ e^{\epsilon \langle \TT,\q\rangle}\right] = \eplim \EEpe \left[ e^{\epsilon \langle \tilTT,\q_{\|}\rangle}\right]\] 
gives the limiting distribution of the scaled perpendicular component $\epsilon \q_{\|}$.  Thus, Eq. \eqref{eq: 3q_q_qpara_mgf} implies that limiting distribution of $\epsilon \q$ and $\epsilon \q_{\|}$ matches as $\epsilon \rightarrow 0$, and this formalizes the statement provided in Remark \ref{rem: 3q_qperp}.

One can show that \[\{\TT \in \mathbb C^3 : \TT \in \mathcal{S}, Re(\TT) \leq \mathbf 0_3\} \oplus \mathcal{S}^\perp \subsetneq \tilde{\boldsymbol \Theta}.\]
This means that solving for $\eplim \EEpe \left[ e^{\epsilon \langle \TT,\q\rangle}\right]$ for $\TT \in \{\TT \in \mathbb C^3 : \TT \in \mathcal{S}, Re(\TT) \leq \mathbf 0_3\}$ is not enough to extend the result to solve for $\eplim \EEpe \left[ e^{\epsilon \langle \tilTT,\q\rangle}\right]$ for all $\tilTT \in \tilde{\boldsymbol \Theta}$. Essentially, $\capTT$ is a large enough set such that solving for $\eplim \EEpe \left[ e^{\epsilon \langle \TT,\q\rangle}\right]$ for $\TT\in\capTT$ is enough to characterize the limiting distribution, and considering a set smaller than $\capTT$ might not be sufficient. Further, as $\capTT \subset \mathcal{S}$, it has some useful structural properties that can be exploited to make the analysis simpler. As such, in Theorem \ref{thm: 3q_functional_eq}, we establish the function equation only for $\TT\in \capTT$.

\proof{\textit{Proof outline of Proposition \ref{prop: 3q_mgf_convergence}}.} For $\TT\in \capTT$ and $\TT_\perp \in \mathcal{ S}^\perp$, with $\TT\in \capTT$ and $\TT_\perp \in \mathcal{ S}^\perp$, we have 
\begin{align*}
    \epsilon\langle \tilTT,\q\rangle = \epsilon\langle \TT,\q\rangle + \epsilon \langle \TT_\perp,\q_\perp\rangle.
\end{align*}
From SSC, we know that $\epsilon \q_\perp \approx 0$ for $\epsilon$ small, and so, $\epsilon\langle \tilTT,\q\rangle$ and $\epsilon\langle \TT,\q\rangle$ follows similar distribution (or have similar Laplace transform) for $\epsilon$ small.  The complete proof for Eq. \eqref{eq: 3q_q_qpara_mgf} is provided in Lemma \ref{lem: 3q_mgf_equivalence_b}. Next, as $\mathcal{S}$ is the span of columns of $\mathbf B$, we have $\mathbf B^T \boldsymbol\theta_\perp = \mathbf 0$ for all $\boldsymbol \theta_\perp \in \mathcal S^\perp$. Thus, $\mathbf B^T \tilTT = \mathbf B^T \TT$, and so,
\begin{align*}
    \eplim \EEpe \left[ e^{\epsilon \langle \tilTT,\q\rangle}\right] = \eplim \EEpe \left[ e^{\epsilon \langle \TT,\q\rangle}\right] = \EE\left[ e^{ \langle \mathbf B^T \TT,\mathbf X\rangle} \right] = \EE\left[ e^{ \langle \mathbf B^T \tilTT,\mathbf X\rangle} \right] = \EE\left[ e^{ \langle \tilTT, \mathbf B \mathbf X\rangle} \right].
\end{align*}
Afterwards, the conclusion $\epsilon \q \stackrel{d}{\rightarrow} \mathbf B \mathbf X$ follows using Lemma \ref{lem: laplace_convergence}. \hfill $\blacksquare$
\endproof

\subsubsection{Results for Three-queue system}
\label{sec: 3q_results}
In this section, we present the results related to the heavy traffic distribution of the Three-queue system. Theorem \ref{thm: 3q_functional_eq} presents the functional equation for the heavy traffic distribution of the Three-queue system. Theorem \ref{thm: 3q_dist} provides the heavy traffic distribution for the Three-queue system under a certain condition on the variance of the arrival process. 

\begin{theorem}[Functional equation for Three-queue system]
\label{thm: 3q_functional_eq}
 Consider the Three-queue system as defined in Section \ref{sec: 3q_model} operating under a scheduling algorithm that achieves SSC according to Definition \ref{def: 3q_ssc}. Let 
 \begin{align*}
     \boldsymbol \Theta &= \{\boldsymbol \theta \in \mathbb{C}^{3} : \boldsymbol \theta \in \mathcal{S}, \ Re(\mathbf B^T \boldsymbol \theta) \leq  \mathbf 0_{2}\} \\
     &= \{\boldsymbol \theta \in \mathbb{C}^{3} : \theta_1 = \theta_2+\theta_3, Re(2\theta_2+\theta_3)\leq 0, Re(\theta_2+2\theta_3)\leq 0\}.
 \end{align*}
%  and take $\boldsymbol\theta \in \boldsymbol \Theta$. 
 Then, for all $\boldsymbol \theta \in \boldsymbol \Theta$, the limiting scaled queue length satisfies,
 \begin{equation}
 \label{eq: 3q_functional_eq}
     \mathcal{P}(\TT):= \left( -\frac{1}{2}\langle \boldsymbol{\theta}, \mathbf{1}_{3} \rangle + \frac{1}{2} \langle \boldsymbol{\theta} ,  \boldsymbol\sigma^2 \boldsymbol{\theta} \rangle \right) L(\boldsymbol{\theta}) + \theta_2 M_2(\boldsymbol\theta)+\theta_3 M_3(\boldsymbol\theta) = 0.
 \end{equation}
 where 
 \begin{align*}
    L(\boldsymbol \theta) = \lim_{\epsilon \rightarrow 0} \EEpe[e^{\epsilon \langle \boldsymbol \theta, \q \rangle}], && M_2(\boldsymbol \theta) = \lim_{\epsilon \rightarrow 0} \frac{1}{\epsilon} \EEpe [u_2e^{\epsilon \langle \boldsymbol \theta, \q \rangle}], && M_3(\boldsymbol \theta) = \lim_{\epsilon \rightarrow 0} \frac{1}{\epsilon} \EEpe [u_3e^{\epsilon \langle \boldsymbol \theta, \q \rangle}].
 \end{align*}
 % where the expectation is taken under steady state distribution. 
\end{theorem}

Theorem \ref{thm: 3q_functional_eq} provides the functional equation for the Three-queue system. As mentioned in Proposition \ref{prop: 3q_mgf_convergence}, it is enough to solve for $L(\TT)$ for $\TT\in\capTT$ to characterize the limiting distribution of $\epsilon \q$. Further,  the structure of the set $\boldsymbol \Theta$ is appropriately chosen such that the limiting quantities $L(\boldsymbol \theta), M_2(\boldsymbol \theta)$ and $M_3(\boldsymbol \theta)$ are well-defined. Note that we do not have a term $M_1(\boldsymbol \theta) =0$ in the functional equation given in Eq. \eqref{eq: 3q_functional_eq} as $u_1 =0$ by our convention. 
% According to Lemma \ref{lem: 3q_projection}, for any $\boldsymbol \theta \in \boldsymbol \Theta$, $\langle \boldsymbol \theta , \q \rangle = \langle \boldsymbol \theta , \q_{\|} \rangle$ and so, the functional equation gives a characterization of the limiting distribution of the scaled parallel component $\epsilon\q_{\|}$. 

% The solution of functional equation gives us the Laplace transform of the limiting distribution, that is, $L(\TT)$ for $\TT\in \capTT$, which in turn gives us the limiting distribution by using Proposition \ref{prop: 3q_mgf_convergence}.

% we get the limiting distribution of only the scaled perpendicular component, that is, $\epsilon\q_{\|}$.  However, from SSC, we know that $\epsilon \q \approx \epsilon \q_{\|}$ (as $\epsilon \q_{\perp} \approx 0$), which implies that, as $\epsilon \rightarrow 0$, the limiting distribution of $\epsilon\q_{\|}$ matches with  the limiting distribution of  $\epsilon\q$.

The functional equation presented in Theorem \ref{thm: 3q_functional_eq} is derived by performing the Lyapunov drift analysis on the complex exponential function $e^{\epsilon\langle \TT,\q\rangle}$ and then using a second-order Taylor approximation in terms of the heavy traffic parameter $\epsilon$. The structure of the set $\capTT$ plays a key role in much of the mathematical simplification while doing the second-order approximation. We provide a brief outline of the proof in Section \ref{sec: 3q_outline}, and the complete proof is given in Appendix \ref{app: 3q_functional_eq}.

After establishing the functional equation, the next step is to solve the derived functional equation to get the Laplace transform of the heavy traffic distribution of the steady-state scaled queue length vector. Solving the functional equation, in general, is not easy.  One way to solve the functional equation is to
guess the solution and check whether it satisfies the functional equation or not. If it does, then the solution gives one possible candidate for the Laplace transform of the heavy traffic distribution. Next crucial step is to show that the functional equation has a unique solution. This ensures that the guessed solution is the only solution for the functional equation. For Three-queue system, the functional equation $\mathcal{P}(\TT) =0, \ \forall \TT \in \capTT$ provided in Eq. \eqref{eq: 3q_functional_eq} has a unique solution. 

Before proving the uniqueness result, we provide a technical detail regarding the functions $L(\boldsymbol \theta), M_2(\boldsymbol \theta)$ and $ M_3(\boldsymbol \theta)$ defined in Theorem \ref{thm: 3q_functional_eq}. Essentially, $L(\boldsymbol \theta), M_2(\boldsymbol \theta)$ and $ M_3(\boldsymbol \theta)$ are well-defined functions, and they satisfy some crucial properties that are required to show the uniqueness of solution of the functional equation. 
Suppose $\mathcal{L}(\boldsymbol \Theta)$ is a class of functions such that 
\begin{align}
\label{eq: 3q_analytic_functions}
 \mathcal{L}(\boldsymbol \Theta) =  \{f: \boldsymbol \Theta \rightarrow \mathbb C \ : f \text{ is non-zero, holomorphic, continuous and bounded over }   \boldsymbol \Theta \}.
\end{align}

We claim that the functions $L(\boldsymbol \theta), M_2(\boldsymbol \theta)$ and $ M_3(\boldsymbol \theta)$ lie in the set $\mathcal{L}(\boldsymbol \Theta)$. The statement is as follows.

\begin{lemma}
\label{lem: 3q_analytic}
    Consider the Three-queue system as defined in Section \ref{sec: 3q_model} operating under a scheduling algorithm that achieves SSC according to Definition \ref{def: 3q_ssc}. Then, the functions $(L(\boldsymbol \theta), M_2(\boldsymbol \theta), M_3(\boldsymbol \theta))$ as defined in Theorem \ref{thm: 3q_functional_eq} lie in the set $\mathcal{L}(\boldsymbol \Theta)$.
\end{lemma}

 The argument for Lemma \ref{lem: 3q_analytic} follows by using the properties of the complex exponential function along with Lemma \ref{app: 3q_mgf_equivalence} presented in Appendix \ref{app: 3q}. Complete proof is provided in Appendix \ref{app: 3q_uniqueness}.
 % from . 
 Next, we present the uniqueness result for the functional equation for the Three-queue system, i.e., there is a unique set of function $(L(\boldsymbol \theta), M_2(\boldsymbol \theta), M_3(\boldsymbol \theta))$ that lie in the class $\mathcal L(\TT)$
and solves the functional equation $\mathcal{P}(\TT) =0, \ \forall \TT \in \capTT$ provided in Eq. \eqref{eq: 3q_functional_eq}.

\begin{lemma}
\label{lem: 3q_uniqueness}
Consider the Three-queue system as defined in Section \ref{sec: 3q_model} operating under a scheduling algorithm that achieves SSC according to Definition \ref{def: 3q_ssc}. 
Then, there is a unique set of functions $(L(\boldsymbol \theta), M_2(\boldsymbol \theta), M_3(\boldsymbol \theta))$ such that $L(\boldsymbol \theta), M_2(\boldsymbol \theta)$ and $ M_3(\boldsymbol \theta)$ lie in $\mathcal{L}(\boldsymbol \Theta)$ and satisfies the functional given in Eq. \eqref{eq: 3q_functional_eq} for all values of $\boldsymbol \theta \in \boldsymbol \Theta$.
% that is a valid Laplace transform of a joint distribution of two variables. 
\end{lemma}

Lemma \ref{lem: 3q_uniqueness}  says that the functional equation for the Three-queue system has a unique solution, i.e., there is a unique $L(\boldsymbol\theta)$ defined over the set $\boldsymbol\Theta$ that satisfies Eq. \eqref{eq: 3q_functional_eq} for all $\TT\in \capTT$. A crucial element in order to prove Lemma \ref{lem: 3q_uniqueness} is to show that the limiting quantities $M_2(\TT)$ and $M_3(\TT)$ depend only on a single variable. 
And to show this, we consider a linear transformation of $\TT$, that is, we consider $\boldsymbol \psi = \mathbf B^T \TT$. Then, we show that $M_2(\TT)$ is a function of only $\psi_2$ and $M_3(\TT)$ is a function of only $\psi_1$. This allows us to rewrite the functional equation in Eq. \eqref{eq: 3q_functional_eq} in a form that is consistent with the class of functional equations in Lemma \ref{lem: functional_uniqueness} (see Section \ref{sec: mgfplusuniqueness}). A key element here is that the set $\capTT$ is a subset $\mathcal{S}$, and we exploit this to do the linear transformation functional equation as mentioned. Afterwards, by using Lemma \ref{lem: 3q_analytic} and Lemma \ref{lem: functional_uniqueness}, we imply that the functional equation in Eq. \eqref{eq: 3q_functional_eq} has a unique solution, that is, there is a unique $L(\TT)$, which is a valid Laplace transform and also satisfy the functional equation for all values of $\TT\in\capTT$.

Under the linear transformation $\boldsymbol \psi = \mathbf B^T \TT$, the set $\capTT$ maps to $\boldsymbol \Psi$, where
\[\boldsymbol\Psi =\mathbf B^T \capTT = \{ \boldsymbol \psi \in \mathbb C^2: Re(\boldsymbol \psi ) \leq \mathbf 0_2\}. \]
As the functional equation in Eq. \eqref{eq: 3q_functional_eq} holds for any $\TT\in \capTT$, the functional equation obtained after the linear transformation $\boldsymbol \psi = \mathbf B^T \TT$ holds for any $\boldsymbol \psi \in \boldsymbol \Psi$. This satisfies the essential conditions needed to use Lemma \ref{lem: functional_uniqueness}. 
% This is another reason for choosing $\capTT$ in such a form as given in Theorem \ref{thm: 3q_functional_eq}.
The complete proof of Lemma \ref{lem: 3q_uniqueness} is provided in Appendix \ref{app: 3q_uniqueness}.
% A brief outline of the proof for Lemma \ref{lem: 3q_uniqueness} is provided in Section \ref{sec: 3q_outline}. 
% This is in contrast with the Conjecture \ref{lem: switch_uniqueness}, as even though the uniqueness of the solution of the functional equation for switch is not known, for simpler systems like Three-queue system, we are able to prove the uniqueness of the solution of the functional equation.

\begin{theorem}[Limiting distribution for Three-queue system]
\label{thm: 3q_dist}
 Consider the Three-queue system as defined in Section \ref{sec: 3q_model} operating under a scheduling algorithm that achieves SSC according to Definition \ref{def: 3q_ssc}. Suppose the variance vector $\boldsymbol \sigma^2$ satisfy the condition $2\sigma_1^2 = \sigma_2^2 + \sigma_3^2 $. Then, the limiting distribution is given by 
  \begin{equation*}
      \epsilon \q  \stackrel{d}{\rightarrow} (\Upsilon_1+\Upsilon_2,\Upsilon_1,\Upsilon_2) = \mathbf B \boldsymbol \Upsilon,
 \end{equation*}
 where $\Upsilon_1$ and $\Upsilon_2$ are independent exponential random variables with mean $\frac{3\sigma_2^2 + \sigma_3^2}{4}$ and  $\frac{\sigma_2^2 + 3\sigma_3^2}{4}$ respectively.
\end{theorem}

Theorem \ref{thm: 3q_dist} provides the limiting distribution of the Three-queue system as a linear combination of two independent and exponentially distributed random variables when the variance of the arrival process satisfies the condition  $2\sigma_1^2 = \sigma_2^2 + \sigma_3^2 $. According to Theorem \ref{thm: 3q_dist}, under the condition $2\sigma_1^2 = \sigma_2^2 + \sigma_3^2 $, the limiting distribution is the linear span of the columns of matrix $\mathbf B$ with the exponentially distributed coefficients. Mathematically, $\epsilon \q \stackrel{d }{\rightarrow} \mathbf B \boldsymbol \Upsilon $, where $\boldsymbol \Upsilon = ( \Upsilon_1,  \Upsilon_2)$. 
The proof of Theorem \ref{thm: 3q_dist} uses the results provided in Theorem \ref{thm: 3q_functional_eq} and Lemma \ref{lem: 3q_uniqueness}. The mathematical details regarding the proof of Theorem \ref{thm: 3q_dist} are provided in Appendix \ref{app: 3q_dist}. 
% A flow diagram for deriving the result in Theorem \ref{thm: 3q_dist} is provided in Figure \ref{fig: 3q_flowchart}.

\begin{remark}
Theorem \ref{thm: 3q_dist} also implies that the lower dimensional representation $\mathbf w = (\mathbf B^T \mathbf B)^{-1}\mathbf B^T \mathbf q$ satisfies $\epsilon \mathbf w \stackrel{d }{\rightarrow} \boldsymbol \Upsilon $. Such an implication does not hold for the Input-queued switch as the corresponding $\mathbf B^T \mathbf B$ matrix is not invertible. 
\end{remark}

\begin{corollary}
    \label{cor: 3q_max}
     For the Three-queue system as defined in Section \ref{sec: 3q_model}, MaxWeight scheduling satisfies the functional equation given in Theorem \ref{thm: 3q_functional_eq} and the heavy traffic distribution in Theorem \ref{thm: 3q_dist}.
\end{corollary}

As mentioned in Section \ref{sec: 3q_model}, MaxWeight scheduling achieves SSC according to Definition \ref{def: 3q_ssc}. Now, Corollary \ref{cor: 3q_max} is a direct application  of  Theorem \ref{thm: 3q_functional_eq} and Theorem \ref{thm: 3q_dist}.\\

% all the $n^2$ queues are heavily loaded simultaneously, i.e., there exists 

\subsection{State space collapse in an Input-queued switch}
\label{sec: switch_ssc}

In this section, we present some existing results that are necessary for the analysis of the heavy traffic distribution of the switch. Before presenting the definition of SSC for switch, we present some required geometry.
% Similar to the Three-queue system, a switch system operating under MaxWeight scheduling undergoes a state-space collapse. 
Let $\mathbf B \in \{0,1\}^{n^2\times2n}$ is such that for any $1\leq i,j\leq n$,  
    \begin{equation*}
        B_{i+n(j-1),i} =B_{i+n(j-1),n+j} = 1,
    \end{equation*}
and all other elements are $0$. The matrix $\mathbf B$ is constructed in a way that it takes in a vector $\mathbf r\in \mathbb C^{2n}$ and gives a vector $\mathbf y = \mathbf B \mathbf r \in \mathbb C^{n^2}$ such that $y_{i+n(j-1)} = r_i +r_{n+j}$. Similarly, $\mathbf B^T$ takes in a vector $\mathbf y \in \mathbb C^{n^2}$ and gives out a vector $\mathbf z =\mathbf B^T \mathbf y \in \mathbb C^{2n}$ that consists of the `row sums' and `column sums' of $\mathbf y$, that is, $z_i = \sum_{j=1}^n y_{i+n(j-1)}$ and $z_{n+j} = \sum_{i=1}^n y_{i+n(j-1)}$ for all $1\leq i,j\leq n$.

Consider the subspace $\mathcal{S} \subseteq \mathbb{C}^{n^2}$ to be the space spanned by the columns of $\mathbf B$,
\begin{equation*}
    \mathcal{S} = \Big\{ \mathbf{y} \in \mathbb{C}^{n^2} : \exists \mathbf{r} \in \mathbb{C}^{2n} \ s.t. \ y_{i+n(j-1)} =  r_{i} + r_{n+j} \Big\} = \Big\{ \mathbf{y} \in \mathbb{C}^{n^2} : \exists \mathbf{r} \in \mathbb{C}^{2n} \ s.t. \ \mathbf{y} =\mathbf B \mathbf r \Big\}. 
\end{equation*}
% We define the cone $\mathcal{K}$ to be the intersection of the subspace $\mathcal{S}$ and the positive orthant, i.e., $\mathcal{K} = \mathcal{S} \cap \mathbb{R}^3_+$.
% With a slight abuse of notation, we say that a vector $\mathbf x \in \mathbb{C}^{n^2}$ lies in the space $\mathcal{S}$ if $Re(\mathbf{x}) \in \mathcal{S}$ and $Im(\mathbf{x}) \in \mathcal{S}$.  
Suppose for any vector $\mathbf x \in \mathbb{C}^{n^2}$, $\mathbf x_{\|}$ denotes the projection of $\mathbf{x}$ onto the space $\mathcal{S}$ and $\mathbf x_{\perp} = \mathbf x - \mathbf x_{\|}$. As $\mathbf x_{\|} \in \mathcal{S}$, $\exists \mathbf{r} \in \mathbb{C}^{2n}$ such that $\mathbf x_{\|} =\mathbf B \mathbf r$.  In this case, we call $\mathbf r$ the lower dimensional (or $2n$-dimensional) representation of $\mathbf x_{\|}$.
\begin{remark}
The columns of the matrix $\mathbf B$ are not linearly independent, and so the vector $\mathbf r$, such that $\mathbf x_{\|} =\mathbf B \mathbf r$, need not be unique. If we are given a candidate $\mathbf r$ such that $\mathbf{x}_{\|} = \mathbf{B}\mathbf{r}$, we can create another candidate by considering $\mathbf{r}' = \mathbf{r} - r\begin{bmatrix} \mathbf{1}_n \\-\mathbf{1}_n \end{bmatrix}$ for any $r\in \mathbb C$. Then, $\mathbf r'$ satisfies 
$\mathbf{x}_{\|} = \mathbf{B}\mathbf{r}'$ as $\mathbf{B}\begin{bmatrix} \mathbf{1}_n \\-\mathbf{1}_n \end{bmatrix} = \mathbf{0}_{n^2}$ by structure of matrix $\mathbf B$. This is because even though we are using $2n$ elements (as $\mathbf r \in \mathbb C^{2n}$) to represent a vector in $\mathcal{S}$, the affine dimension of $\mathcal{S}$ is $(2n-1)$ as we can fix one of the elements of $\mathbf r$ to $0$. 
This is the same as saying that the column rank of the matrix $\mathbf B$ is $(2n-1)$.
\end{remark}

We also define the cone $\mathcal K \subset \mathcal S$ given by,
\begin{align*}
    \mathcal{K} &= \Big\{ \mathbf{y} \in \mathbb{R}^{n^2} : \exists \mathbf{w} \in \mathbb{R}^{2n}_+ \ s.t.  \ y_{i+n(j-1)} =  w_{i} + w_{n+j} \Big\} = \Big\{ \mathbf{y} \in \mathbb{R}^{n^2} : \exists \mathbf{w} \in \mathbb{R}^{2n}_+ \ s.t. \ \mathbf{y} =\mathbf B \mathbf r \Big\}. 
\end{align*}
For any vector $\mathbf x\in \mathbb R^{n^2}$, we use $\mathbf x_{\| \mathcal K}$ as the projection to the cone $\mathcal{K}$, and $\mathbf x_{\perp \mathcal K} = \mathbf x - \mathbf x_{\| \mathcal K}$. Note that as $\mathcal{K} \subset \mathcal{S}$, we get that $\|\mathbf x_{\perp}\| \leq \|\mathbf x_{\perp \mathcal K}\|$. Similar to that for the projection onto the space $\mathcal{S}$, the lower dimensional representation to the projection onto the cone $\mathcal{K}$ is also not unique, i.e., there can be multiple $\mathbf w$ such that $\mathbf x_{\| \mathcal K} = \mathbf B \mathbf w$. One way to ensure that the lower dimensional representation is unique is to enforce an extra condition that the smallest element in the vector $\mathbf w$ is zero, i.e., $\min_{1\leq i\leq 2n}w_i =0$.

A difference between the projection the space $\mathcal{ S}$ and the projection to the cone $\mathcal{K}$ is that the projection to $\mathcal{S}$ follows a distributive property, i.e., for $\mathbf x,\mathbf y \in \mathbb R^{n^2}$, $(\mathbf x + \mathbf y)_{\|} = \mathbf x_{\|} + \mathbf y_{\|}$, but the projection to the cone $\mathcal{K}$ need follow the distributive property, i.e., it might happen that $(\mathbf x + \mathbf y)_{\|\mathcal{K}} \neq \mathbf x_{\|\mathcal{K}} + \mathbf y_{\|\mathcal{K}}$. An important consequence of this is that we can write the recursive equation for queue length evolution process for the parallel component as 
\begin{equation*}
    \q_{\|}(t+1) = \q_{\|}(t) +\mathbf a_{\|}(t) - \mathbf s_{\|}(t) +\mathbf u_{\|}(t),
\end{equation*}
while a similar recursive equation cannot be written for the projection of the queue length vector onto the cone $\mathcal{K}$. Thus, while doing the Lyapunov drift analysis, we consider the projection of the queue length vector to the subspace $\mathcal{S}$. Next, we present the definition of the SSC in an Input-queued switch.

\begin{definition}[State Space Collapse]
\label{def: switch_ssc}
For the Input-queued switch as defined in Section \ref{sec: switch_model}
operating under a given scheduling algorithm, we say that the algorithm achieves \textit{State Space Collapse} (SSC), if there exists $\epsilon_0,\theta_0 >0$ such that for every $\theta < \theta_0$ and $\epsilon<\epsilon_0$, the steady state queue length vector satisfies,
    \begin{equation*}
        \EEpe[e^{ \theta \| \q_{\perp}\|}] \leq  \EEpe[e^{ \theta \| \q_{\perp \mathcal K}\|}] < C^\star< \infty,
    \end{equation*}
    where $C^\star$ is a constant independent of $\epsilon$.
    % where $\q_{\|}$ and $\q_{\| \mathcal{K}}$ are the projection of $\q$ onto the subspace $\mathcal S$ and cone $\mathcal K$, respectively. Also, $\q_{\perp} = \q - \q_{\|} $ and $\q_{\perp \mathcal{K}} = \q - \q_{\| \mathcal{K}}$. And the expectation $\mathbb E[\cdot]$ is taken under the steady-state distribution.
As a consequence, for any scheduling policy that achieves state space collapse, for every $\theta <\theta_0$, \[\eplim \EEpe[e^{\epsilon \theta \| \q_{\perp}\|}] \leq \eplim  \EEpe[e^{\epsilon \theta \| \q_{\perp \mathcal K}\|}]  < \infty.\] Furthermore, there exists a $C_r$ independent of $\epsilon$ such that,
    \begin{equation}
    \label{eq: switch_qperp_bound}
        \EEpe \Big [\|\q_{\perp }\|^r \Big] \leq \EEpe \Big [\|\q_{\perp \mathcal K }\|^r \Big] \leq C_r \quad \forall r \geq 1.
    \end{equation}
\end{definition}
According to Definition \ref{def: switch_ssc}, for any scheduling algorithm that achieves SSC, the moments of $\q_{\perp}$ are bounded irrespective of the heavy traffic parameter $\epsilon$. We know from \cite[Proposition 1]{maguluri2016heavy}, that in heavy traffic, queue length scales at least at the rate of $\Omega(1/\epsilon)$.
This shows that, in heavy traffic, $\q_{\perp}$ is insignificant compared to $\q$ and so, $\epsilon\q$ is nearly equal to its projection $\epsilon\q_{\|}$. 
For the Input-queued switch, MaxWeight scheduling achieves SSC according to the Definition \ref{def: switch_ssc}, and the proof is provided in \cite[Proposition 2]{maguluri2016heavy}.
% , which in turn uses the result given in \cite{hajek1982hitting}.
In \cite{jhunjhunwala2021low}, the authors prove that there is a large class of \textit{MaxWeight-Like} algorithms that also achieve SSC according to Definition \ref{def: switch_ssc}.

With slight abuse of notation, we redefine the set $\capTT$ and $\tilde{\boldsymbol \Theta}$ as 
\begin{align*}
    \boldsymbol \Theta & = \{\boldsymbol \theta \in \mathbb{C}^{n^2} : \boldsymbol \theta \in \mathcal{S}, \ Re(\mathbf B^T \boldsymbol \theta) \leq  \mathbf 0_{2n}\}\\
    \tilde{ \boldsymbol \Theta } & = \{\tilTT \in \mathbb{C}^{n^2}: Re(\tilTT) \leq \mathbf 0_{n^2}\},
\end{align*}
and we use $\TT$ and $\tilTT$ to denote an element in $\capTT$ and $\tilde{ \boldsymbol \Theta }$, respectively.

\begin{proposition}
\label{prop: switch_mgf_convergence}
   Consider the Input-queued switch as defined in Section \ref{sec: switch_model} operating under a scheduling algorithm that achieves SSC according to Definition \ref{def: switch_ssc}. Suppose $\tilTT \in \tilde{ \boldsymbol \Theta } \subseteq \capTT \oplus \mathcal S^\perp$ such that $\tilTT = \TT +\TT_\perp$, where $\TT\in \capTT$ and $\TT_\perp \in \mathcal{ S}^\perp$. Then, 
   \begin{align}
   \label{eq: switch_q_qpara_mgf}
       \eplim \EEpe \left[ e^{\epsilon \langle \tilTT,\q\rangle}\right] = \eplim \EEpe \left[ e^{\epsilon \langle \TT,\q\rangle}\right].
   \end{align}
   Further, if $\eplim \EEpe \left[ e^{\epsilon \langle \TT,\q\rangle}\right] = \EE\left[ e^{ \langle \mathbf B^T \TT,\mathbf X\rangle}\right]$ for all $\TT\in \capTT$, then 
   % $\eplim \EEpe \left[ e^{\epsilon \langle \tilTT,\q\rangle}\right] = \EE\left[ e^{ \langle \tilTT,\mathbf B \mathbf X\rangle}\right]$ for all $\tilTT \in \tilde{ \boldsymbol \Theta }$ and so,
   \begin{align*}
       \epsilon \q \stackrel{d}{\rightarrow} \mathbf B \mathbf X.
   \end{align*}
\end{proposition}

Proposition \ref{prop: switch_mgf_convergence} has same implications for the Input-queued system as that of Proposition \ref{prop: 3q_mgf_convergence} for Three-queue system. As such, we only need to solve for $\eplim \EEpe \left[ e^{\epsilon \langle \TT,\q\rangle}\right] $ for $\TT\in \capTT$ to get the limiting distribution of the scaled queue length process.

\subsection{Results for Input-queued switch}
\label{sec: switch_results}
In this section, we present our results for the Input-queued switch. Theorem \ref{thm: switch_functional_eq} provides a functional equation that characterizes the heavy traffic distribution of the scaled queue length vector for Input-queued switch. Under the assumption that Conjecture \ref{lem: switch_uniqueness} holds true, Theorem \ref{thm: switch_dist} provides the solution of the functional equation given in Theorem \ref{thm: switch_functional_eq} under the symmetric variance condition. 

\begin{theorem}[Functional equation for Input-queued switch]
\label{thm: switch_functional_eq}
Consider the Input-queued switch as defined in Section \ref{sec: switch_model}
operating under a given scheduling algorithm that achieves state space collapse according to Definition \ref{def: switch_ssc}. Let $\boldsymbol \Theta = \{\boldsymbol \theta \in \mathbb{C}^{n^2} : \boldsymbol \theta \in \mathcal{S}, \ Re(\mathbf B^T \boldsymbol \theta) \leq  \mathbf 0_{2n}\}$.
% , where $\mathbf{d}_i$'s are the columns of $\mathbf B^T \mathbf B$. 
% Suppose $ \boldsymbol \phi \in \mathbb C^{2n}$ such that $\boldsymbol \phi \in \boldsymbol \Phi$ and take $\boldsymbol \theta = \mathbf B \boldsymbol \phi$. 
Then, for any $\boldsymbol\theta \in \boldsymbol \Theta$, the heavy traffic scaled queue length vector satisfies,
 \begin{equation}
 \label{eq: switch_functional_eq}
   \mathcal{P}(\TT) = \left( 2\langle \boldsymbol{\theta}, \mathbf{1}_{n^2} \rangle - n \langle \boldsymbol{\theta} ,  \boldsymbol\sigma^2 \boldsymbol{\theta} \rangle \right) L(\boldsymbol{\theta}) - 2n \langle \boldsymbol{\theta}, \mathbf{M}(\boldsymbol{\theta}) \rangle =0,
 \end{equation}
  where 
 \begin{align*}
    L(\boldsymbol \theta) = \lim_{\epsilon \rightarrow 0} \EEpe[e^{\epsilon \langle \boldsymbol \theta, \q \rangle}], && M_{k}(\boldsymbol \theta) = \lim_{\epsilon \rightarrow 0} \frac{1}{\epsilon} \EEpe [u_{ k}e^{\epsilon \langle \boldsymbol \theta, \q \rangle }], \ \ \forall k \in \{1,2,\dots,n^2\}.
 \end{align*}
%  and $\boldsymbol \Gamma = \mathbf B^T \boldsymbol \sigma^2 \mathbf B$.
\end{theorem}

%  The functional equation presented in Theorem \ref{thm: switch_functional_eq} is derived by performing the Lyapunov drift analysis on the complex exponential function and then doing the second-order approximation in terms of the heavy traffic parameter $\epsilon$. 
%  Next, we conjecture that the solution for the functional equation presented in Eq. \ref{eq: switch_functional_eq} has a unique solution.

% Theorem \ref{thm: switch_functional_eq} gives a characterization of the functional equation for the Input-queued switch.
The functional equation (Eq. \eqref{eq: switch_functional_eq} in Theorem \ref{thm: switch_functional_eq}) is mathematical relationship between the term $L(\boldsymbol \theta)$, which is the Laplace transform of the limiting heavy traffic distribution and the terms $M_k(\boldsymbol{\theta})$, which intuitively denote the Laplace transform under the condition $q_k=0$ (since $u_k =1$ implies $q_k =0$). Further, the set $\boldsymbol \Theta$ is appropriately chosen such that the quantities $L(\boldsymbol \theta)$ and $M_k(\boldsymbol{\theta})$'s are well defined. 

The steps to establish the functional equation for Input-queued switch is very similar to that of the Three-queue system. 
% In order to establish the functional equation presented in Theorem \ref{thm: switch_functional_eq}, 
The first step is to use the complex exponential as the Lyapunov function and equate its expected drift to zero in steady-state.
After that, we use the second-order approximation of the complex exponential in terms of the heavy traffic parameter $\epsilon$ and eliminate the higher order terms to get the functional equation. Here, SSC plays a key role in the mathematical analysis. To be more specific, due to the SSC, 
we only have to consider $\q_{\|}$ (since $\langle \boldsymbol \theta, \q \rangle = \langle \boldsymbol \theta, \q_{\|} \rangle$ for any $\TT\in \mathcal S$), which leads to a lot of technical simplicity.
This allows us to perform the analysis to characterize the heavy traffic distribution using the functional equation.  The complete proof for Theorem \ref{thm: switch_functional_eq} is presented in Appendix \ref{app: switch_functional_eq}.
% , and so, we can pick $\boldsymbol \theta \in \mathbb C^{n^2}$ and take $\TT = \mathbf B \boldsymbol \theta$ in Theorem \ref{thm: switch_functional_eq}. 
% Also, the condition $Re(\mathbf B^T \boldsymbol \theta) \leq  \mathbf 0_{2n}$ is chosen to ensure that the limiting Laplace transforms $L(\boldsymbol \theta)$ and $\mathbf M(\boldsymbol \theta)$ exists. 

Similar to that for the Three-queue system, we define $\mathcal{L}(\boldsymbol \Theta)$ to be class of functions such that 
\begin{align}
\label{eq: switch_analytic_functions}
 \mathcal{L}(\boldsymbol \Theta) =  \{f: \boldsymbol \Theta \rightarrow \mathbb C \ : f \text{ is non-zero, holomorphic, continuous and bounded over }   \boldsymbol \Theta \}.
\end{align}

Then, we have the following lemma.

\begin{lemma}
\label{lem: switch_analytic}
    Consider the Input-queued switch as defined in Section \ref{sec: switch_model} operating under a given scheduling algorithm that achieves state space collapse according to Definition \ref{def: switch_ssc}. Then, the functions $(L(\boldsymbol \theta), \mathbf M(\boldsymbol \theta))$ as defined in Theorem \ref{thm: switch_functional_eq} lie in the set $\mathcal{L}(\boldsymbol \Theta)$.
\end{lemma}

Lemma \ref{lem: switch_analytic} is same as Lemma \ref{lem: 3q_analytic} for the case of Three-queue system and uses similar arguments as in the proof of Lemma \ref{lem: 3q_analytic}. The implication of Lemma \ref{lem: switch_analytic} for Input-queued switch is same as that in the case of Lemma \ref{lem: 3q_analytic} for Three-queue system.
For Input-queued switch, we conjecture that the functional equation has a unique solution as given below.

\begin{conjecture}
\label{lem: switch_uniqueness}
Consider the Input-queued switch as defined in Section \ref{sec: switch_model}
operating under a given scheduling algorithm that achieves state space collapse according to Definition \ref{def: switch_ssc}. 
% Suppose $(L(\boldsymbol \theta),\mathbf M(\boldsymbol\theta))$ are as defined in Theorem \ref{thm: switch_functional_eq}. 
Then, there is a unique set of functions $(L(\boldsymbol \theta), \mathbf M(\boldsymbol \theta))$ such that $L(\boldsymbol \theta)$ and $, M_k(\boldsymbol \theta)$'s lie in $\mathcal{L}(\boldsymbol \Theta)$ and satisfies the functional equation $\mathcal{P}(\TT) =0$ given in Eq. \eqref{eq: switch_functional_eq} for all $\TT\in \boldsymbol \Theta$. 
\end{conjecture}

A major challenge in solving the implicit functional equation given in Eq. \eqref{eq: switch_functional_eq} is proving that the functional equation has a unique solution. For simpler systems such as the Three-queue system or \nsys (see Section \ref{sec: n_sys}) where the SSC happens to a two-dimensional subspace, one can prove that the corresponding functional equation has a unique solution using the theory of Carleman boundary value problem \cite{litvinchuk1970generalized}. Extending that result to a functional equation with more than two variables, such as the case for the Input-queued switch, is open. 
% We are currently unaware of any method to prove the uniqueness of the functional equation for the Input-queued switch. 
Next, we present Theorem \ref{thm: switch_dist}, which assumes that the Conjecture \ref{lem: switch_uniqueness} holds and provide a solution to the functional equation for Input-queued switch under symmetric variance condition.

% \color{red}
% The proof for the uniqueness of the solution for the functional equation for the Three-queue system follows from the analysis presented in \cite{franceschi2019integral}, which in turn is heavily dependent on the number of variables in the functional equation. The arguments presented in \cite{franceschi2019integral} hold only if the functional equation involves two variables. Extending that result to a functional equation with more than two variables turns out to be highly non-trivial. For \nsys and Three-queue system, we know that the functional equation has a unique solution, so we believe that the same holds for Input-queued switch. 
% \color{black}
% Next, we provide the solution for the Eq. \eqref{eq: switch_functional_eq} in Theorem \ref{thm: switch_functional_eq}.
% Theorem \ref{thm: switch_dist} provides the heavy traffic distribution of the switch system under the symmetric variance condition. 

\begin{theorem}
\label{thm: switch_dist}
Consider the Input-queued switch as defined in Section \ref{sec: switch_model}
operating under a given scheduling algorithm that achieves state space collapse according to Definition \ref{def: switch_ssc}. Assume Conjecture \ref{lem: switch_uniqueness} holds. Suppose the variance vector $\boldsymbol \sigma^2$ is symmetric, i.e., $\boldsymbol \sigma^2 =\sigma^2 \mathbf{I}_{n^2} $. Then, the heavy traffic steady state queue length vector is given by
  \begin{equation*}
     \epsilon q_{i+n(j-1)} \stackrel{d}{\rightarrow}  \Upsilon_i + \Upsilon_{n+j} - 2 \Tilde{\Upsilon}, \ \ \forall i,j \in \{1,\dots,n\},
 \end{equation*}
where $\{\Upsilon_1,\dots,\Upsilon_{2n}\}$ are independent exponential random variable with mean $\frac{\sigma^2}{2}$ and $\Tilde{\Upsilon}  =\displaystyle \min_{1\leq k\leq 2n}  \Upsilon_k$. In vector notations, we have
 \begin{equation*}
     \epsilon \q \stackrel{d}{\rightarrow} \mathbf  B ( \boldsymbol \Upsilon - \Tilde{\Upsilon}  \mathbf 1_{2n}),
 \end{equation*}
 where $\boldsymbol \Upsilon $ is the vector $ \boldsymbol \Upsilon = (\Upsilon_1,\dots,\Upsilon_{2n})$.
\end{theorem}

Theorem \ref{thm: switch_dist} completely characterizes the heavy traffic distribution of the Input-queued switch under the symmetric variance condition, where the heavy traffic distribution is represented using $2n$ independent exponential random variable. 
% This suggests that the heavy traffic distribution of the switch is light-tailed even when the symmetric variance condition is not satisfied.
The key idea behind the proof is to show that the Laplace transform of the limiting distribution provided in the Theorem \ref{thm: switch_dist} is a solution of a functional equation (Eq. \eqref{eq: switch_functional_eq}) given in Theorem \ref{thm: switch_functional_eq} when the variances of the arrival process are symmetric. And under the assumption that the functional equation has a unique solution claimed by Conjecture \ref{lem: switch_uniqueness}, the solution provided in Theorem \ref{thm: switch_dist} is the unique solution for the heavy traffic distribution for Input-queued switch under symmetric variance condition. The mathematical proof of Theorem \ref{thm: switch_dist} is provided in Appendix \ref{app: switch_dist}. 

\begin{remark}
\label{rem: switch_memoryless}
As $\{\Upsilon_1,\dots,\Upsilon_{2n}\}$ are identically distributed, we have that $\mathbb P(\Tilde{\Upsilon} = \Upsilon_k) = \frac{1}{2n}$, for all $k\in \{1,2,\dots,2n\}$. Further, by memoryless property of exponential random variables, $[\Upsilon_k-\Tilde{\Upsilon}|\Upsilon_k>\Tilde{\Upsilon}]$ are independent and exponentially distributed with mean $\frac{\sigma^2}{2}$ (or rate $\frac{2}{\sigma^2}$). This implies that in the random vector $\boldsymbol \Upsilon - \Tilde{\Upsilon}  \mathbf 1_{2n}$, exactly one element is zero, and all others are independent and exponentially distributed with rate $\frac{2}{\sigma^2}$.
    For a given queue, say $q_{i+n(j-1)}$, if $\Tilde{\Upsilon} = \Upsilon_i$, we have $\epsilon q_{i+n(j-1)} \stackrel{d}{\rightarrow} [\Upsilon_{n+j}-\Tilde{\Upsilon}|\Upsilon_{n+j}>\Tilde{\Upsilon}] $, which is exponentially distributed with rate $\frac{2}{\sigma^2}$. Same happens when $\Tilde{\Upsilon} = \Upsilon_{n+j}$. And if neither happens, that is, $\Upsilon_{i}>\Tilde{\Upsilon}$ and $\Upsilon_{n+j}>\Tilde{\Upsilon}$, then by the memoryless property, we also get that  $[\Upsilon_{i}-\Tilde{\Upsilon}|\Upsilon_{i}>\Tilde{\Upsilon}] $ and $ [\Upsilon_{n+j}-\Tilde{\Upsilon}|\Upsilon_{n+j}>\Tilde{\Upsilon}]$ are independent of each other in addition to being exponentially distributed. This can be summarized as, 
    \begin{align*}
        \epsilon q_{i+n(j-1)} \stackrel{d}{\rightarrow} \Upsilon_{i}+\Upsilon_{n+j}-2\Tilde{\Upsilon} \sim \begin{cases}
            \text{Exponential with rate $\frac{2}{\sigma^2}$} & w.p. \ \ \frac{1}{n}\\
            \text{Erlang-2 with rate $\frac{2}{\sigma^2}$} & w.p. \ \ 1-\frac{1}{n}.
        \end{cases}
    \end{align*}
\end{remark}

The arguments presented in Remark \ref{rem: switch_memoryless} are used in the proof of Theorem \ref{thm: switch_dist} to compute the Laplace transform of $\mathbf  B ( \boldsymbol \Upsilon - \Tilde{\Upsilon}  \mathbf 1_{2n})$.

\begin{remark}
\label{rem: switch_minterm}
Suppose $\mathbf w$ is the lower dimensional representation of the $\q_{\| \mathcal K}$ (the projection onto the cone $\mathcal{ K}$). For $\epsilon$ small, we have $\epsilon \q \approx \epsilon \q_{\| \mathcal K} = \epsilon \mathbf B \mathbf w$. As mentioned before,  we can ensure that the lower dimensional representation $\mathbf r$ is unique by adding an extra constraint that $\min_{1\leq i\leq 2n} w_i =0$. Such a condition is satisfied by the vector $\boldsymbol \Upsilon - \Tilde{\Upsilon}  \mathbf 1_{2n}$. This provides intuitive reasoning behind the appearance of the term $\Tilde{\Upsilon} $. However, unlike that for the Three-queue system, we cannot claim that the scaled lower dimensional representation $\epsilon\mathbf w \stackrel{d}{\rightarrow} \boldsymbol \Upsilon - \Tilde{\Upsilon}  \mathbf 1_{2n}$, even though $\epsilon \q \stackrel{d}{\rightarrow} \mathbf  B ( \boldsymbol \Upsilon - \Tilde{\Upsilon}  \mathbf 1_{2n})$ as $\mathbf B^T \mathbf B$ is not an invertible matrix. Essentially, there can be other random vectors $\boldsymbol\Upsilon'$ such that $\mathbf B \boldsymbol\Upsilon' = \mathbf  B ( \boldsymbol \Upsilon - \Tilde{\Upsilon}  \mathbf 1_{2n})$. As such, our result in Theorem \ref{thm: switch_dist} does not provide the limiting distribution of the lower dimensional representation, but directly computes the limiting distribution of the projected component.
    % adding the extra constraint that $\min_{1\leq i\leq 2n} r_i =0$ is not the only possible way of ensuring that the lower dimensional representation $\mathbf r$ is unique. 
\end{remark}

\begin{corollary}
    \label{cor: switch_max}
     For the Input-queued switch as defined in Section \ref{sec: switch_model}, MaxWeight scheduling satisfies the functional equation given in Theorem \ref{thm: switch_functional_eq} and the heavy traffic distribution in Theorem \ref{thm: switch_dist}.
\end{corollary}
As mentioned in Section \ref{sec: switch_ssc}, MaxWeight scheduling achieves SSC according to Definition \ref{def: switch_ssc}. Now, Corollary \ref{cor: switch_max} is a direct application  of  Theorem \ref{thm: switch_functional_eq} and Theorem \ref{thm: switch_dist}. Similarly, the class of algorithms presented in \cite{jhunjhunwala2021low} also satisfies the result presented in Theorem \ref{thm: switch_functional_eq} and Theorem \ref{thm: switch_dist}. 

Next, we reiterate the results mentioned in Theorem \ref{thm: switch_dist} for a $2\times 2$ switch to provide more clarity. With slight abuse of notation, we use the matrix notation $q_{ij} = q_{i+n(j-1)}$. Using this notation, the queue length matrix and its corresponding projection is given by,
\begin{align*}
\q = \begin{bmatrix}
        q_{11} & q_{12}\\
        q_{21} &  q_{22} 
    \end{bmatrix},
&& \q_{\|\mathcal K} = 
    \begin{bmatrix}
        w_1 & w_1\\
        w_2 &w_2
    \end{bmatrix}  
    + \begin{bmatrix}
        w_3 & w_4\\
        w_3 &w_4
    \end{bmatrix}.
\end{align*}
From the above representation, we note that $\q_{\|\mathcal K}$ is represented as the sum of two matrices. For the first matrix, row elements are common and for the second matrix, column elements are common. Even though the representation of $\q_{\| \mathcal K}$ involves four elements, $\q_{\| \mathcal K}$ lies in a three-dimensional subspace, and so, the elements $\{w_1,w_2,w_3,w_4\}$ are not uniquely determined for a given vector $\q$. Now, according to Theorem \ref{thm: switch_dist}, in heavy traffic and under symmetric variance condition,
\begin{equation*}
    \epsilon\q \stackrel{d}{\rightarrow} \begin{bmatrix}
                        \Upsilon_1 & \Upsilon_1\\
                        \Upsilon_2 &\Upsilon_2
                \end{bmatrix}  + \begin{bmatrix}
                        \Upsilon_3 & \Upsilon_4\\
                        \Upsilon_3 &\Upsilon_4
                \end{bmatrix} - 2\begin{bmatrix}
                        \Tilde{\Upsilon} & \Tilde{\Upsilon}\\
                        \Tilde{\Upsilon} &\Tilde{\Upsilon}
                \end{bmatrix},
\end{equation*}
where $\Upsilon_1,\Upsilon_2,\Upsilon_3$ and $\Upsilon_4$ are exponential random variables with mean $\sigma^2/2$ and $\Tilde{\Upsilon} = \min\{\Upsilon_1,\Upsilon_2,\Upsilon_3,\Upsilon_4\}$. 
% By SSC, we know that $\epsilon\q_{\perp} \approx 0$ and so $\epsilon \q \approx \epsilon \q_{\|}$ in heavy traffic. This means that we can think of $\epsilon w_{i}$ converging to $\Upsilon_i - \Tilde{\Upsilon}$ in distribution. Next, we look at a simpler system, which we call Three-queue system.

\subsubsection{Comparison of Input-queued switch and Three-queue system}
Three-queue system is a simpler queueing system compared to Input-queued switch, but they are analogous in basic structure. However, there are some distinctions in terms of the heavy traffic behavior of these two systems. The first difference between these two systems is the state space collapse result. For the Three-queue system, the state space collapse occurs to a two-dimensional subspace, and the two dimensional representation of the projection of the queue length vector to the corresponding subspace is unique. In contrast to that, for the Input-queued switch of size $n$, the state space collapse occurs to a subspace of dimension $2n-1$, and the projection of the queue length vector to its corresponding subspace does not have a unique lower dimensional representation. As mentioned in Remark \ref{rem: switch_minterm}, this `non-unique representation' in the case of Input-queued switch is, intuitively, the reason behind an additional $\Tilde{\Upsilon}$ term for Input-queued switch as seen in Theorem \ref{thm: switch_dist}.

Another difference is in terms of the uniqueness of the functional equation of the two systems. The functional equation for the Three-queue system involves only two variables, and so the functional equation for the Three-queue system has a unique solution as shown by Lemma \ref{lem: 3q_uniqueness}. As opposed to that, the functional equation for the Input-queued switch has more than two variables. Currently, we do not have a proof for the uniqueness of the solution for the functional equation for the Input-queued switch. Finally, the third difference between these two systems is in terms of the condition on the variances of the arrival process under which the heavy traffic distribution can be represented by using independent exponential random variables. For Three-queue system, the variance condition in Theorem \ref{thm: 3q_dist} (i.e., $2\sigma_1^2 = \sigma_2^2 + \sigma_3^2$) is more general than the symmetric variance condition (i.e.,  $\boldsymbol \sigma^2 = \sigma^2 \mathbf I_{n^2}$) for Input-queued switch as considered in Theorem \ref{thm: switch_dist}. 

% The system is said to be in \textit{heavy traffic} if the arrival rate vector $\boldsymbol\lambda \in \mathcal{C}$ is close to the boundary of the capacity region. Note that if the arrival rate of any of the queues is zero, we can simplify the model. 
% So to preserve the complexity of the system, 

% \subsubsection{Input-queued switch of size 2}
% In this section, we consider the example of an Input-queued switch with $n=2$.

% Next, we present the technical details required to prove the result.

% Next, we briefly present the mathematical analysis behind the results presented in this paper for the Three-queue system. 

% \color{red}
% \subsubsection{Three-queue system vs Input-queued switch} 
% \color{black}

\section{$\mathcal{N}$-System}
\label{sec: n_sys}
We study a parallel server system consisting of two queues (each representing a different job class) and two servers. The first server can only serve jobs in the first job class, while the second server can serve jobs in both job classes. This system is commonly known as $\mathcal{N}$-system. We assume that the service time of any job is server-dependent and not class-dependent, i.e., the mean service time depends only on the server that serves the job. 
% This is similar to the system considered in \textcolor{red}{ \cite{} --Jim dai paper--.}
$\mathcal{N}$-system has been heavily studied under the CRP condition, where only one of the two queues are loaded close to the capacity. This assumption leads to a comparatively simple mathematical analysis. In our paper, we study the $\mathcal{N}$-system under a regime in which CRP condition is not satisfied, i.e., both the queues are simultaneously working close to the capacity. Also, for the sake of completeness, we provide the limiting queue length distribution for the $\mathcal{N}$-system under both CRP and non-CRP condition.
% \nsys is one of the simplest systems for which CRP condition is not satisfied.

In Section \ref{sec: n_sys_model}, we provide the model for the $\mathcal{N}$-system operating in a continuous-time fashion. Section \ref{sec: nsys_ssc} gives the state space collapse result for the $\mathcal{N}$-system. Finally, the results, including the functional equation and limiting distribution for $\mathcal{N}$-system is provided in Section \ref{sec: n_sys_dist}.
% Section \ref{sec: nsys_ssc} gives the heavy traffic distribution of steady-state scaled queue length vector under the condition that service rates of the two servers are symmetric. 
% In Section \ref{sec: n_sys_dist}, we provide the functional equation that models the heavy traffic steady-state distribution of the \nsys under general service rates, along with the claim that the functional equation for \nsys has a unique solution.

\subsection{Model for \nsys}
\label{sec: n_sys_model}

We consider a continuous time $\mathcal{N}$-system with two queues given by $Q_1$ and $Q_2$, each denoting a different job class. The corresponding queue lengths, at time $t$, are denoted by $q_1(t)$ and $q_2(t)$. The arrival process of jobs for $Q_1$ and $Q_2$ are independent Poisson processes with rates $\lambda_1$ and $\lambda_2$, respectively. The two servers in the system are denoted by $S_1$ and $S_2$. 
% The service time of the server $S_1$ and $S_2$ are independent and  exponentially distributed with rate  respectively. 
The jobs in $Q_1$ can be served only by $S_1$, while the jobs in the $Q_2$ are more flexible and can be served by either $S_1$ or $S_2$. Thus, server $S_1$ can serve jobs from either of the queues, but it cannot serve both queues simultaneously. As a result, there are two possible service configurations. The first one is that $S_1$ serves $Q_1$ and $S_2$ serves $Q_2$, and the second configuration being $S_1$ and $S_2$ both serve $Q_2$. The processing times of jobs on $S_1$ and $S_2$ are exponentially distributed with parameters $\mu_1$ and $\mu_2$, respectively, and are independent of the job beign served. The model for \nsys is pictorially depicted in Fig. \ref{n_system}.

\begin{figure}[h!]
\centering
\begin{tikzpicture}[>=latex]
% the rectangle with vertical rules
\draw (0,0) -- ++(2cm,0) -- ++(0,-0.5cm) -- ++(-2cm,0);
\draw (0,1) -- ++(2cm,0) -- ++(0,-0.5cm) -- ++(-2cm,0);
% \foreach \i in {1,...,4}
%   \draw (2cm-\i*10pt,0) -- +(0,-1.5cm);
% the circle
\draw (4,-0.25cm) circle [radius=0.3cm];
\draw (4,0.75cm) circle [radius=0.3cm];
% the arrows and labels
\draw[->] (2.25,-0.25) -- +(35pt,0);
\draw[->] (2.25,-0.15) -- +(35pt,0.8);
\draw[<-] (0,-0.25) -- +(-20pt,0) node[left] {$\lambda_2 $};
\node at (4,-0.25cm) {$\mu_2$};
\draw[->] (2.25,0.75) -- +(35pt,0);
\draw[<-] (0,0.75) -- +(-20pt,0) node[left] {$\lambda_1 $};
\node at (4,0.75cm) {$\mu_1$};
% \node[align=center] at (1cm,-2cm) {Waiting \\ Area};
% \node[align=center] at (3cm,-2cm) {Service \\ Node};
\end{tikzpicture}
\caption{The model for $\mathcal{N}$-system}
\label{n_system}
\end{figure}
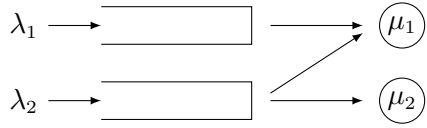
% Due to the constraints of the service process, $S_1$ can serve either $Q_1$ or $Q_2$, not both simultaneously. Thus, there are two possible service configurations.

A scheduling policy is an underlying rule by which the system chooses the service configuration to use in any time slot. In this paper, we only consider the scheduling algorithms for which the process $\{\q(t)\}_{t=0}^\infty$ forms an irreducible and aperiodic Markov chain, where $\q(t) = (q_1(t),q_2(t))$. The scheduling policy is said to be \textit{stable} if the Markov chain $\{\q(t)\}_{t=0}^\infty$ is positive recurrent. 
The \textit{capacity region} $\mathcal{C}$ is the set of arrival rate vector $\boldsymbol \lambda = (\lambda_1,\lambda_2)$ for which there exists a scheduling policy such the system is stable.  The capacity region of $\mathcal{N}$-system is given by 
\begin{equation*}
    \mathcal{C} = \{\boldsymbol \lambda \in \mathbb{R}^2_+ : \lambda_1 + \lambda_2 < \mu_1 + \mu_2, \lambda_1< \mu_1 \}.
\end{equation*}
We define $\mathcal F$ to be a part of the boundary of the capacity region $\mathcal C$, given by $\mathcal F = \mathcal{F}_1 \cup \mathcal{F}_2 \cup \mathcal{F}_3 $, where
\begin{align*}
    \mathcal{F}_1 &= \{\boldsymbol \nu \in \mathbb{R}^2_+ : \nu_1 + \nu_2 = \mu_1 + \mu_2, \nu_1< \mu_1 \},\\
    \mathcal{F}_2 &= \{\boldsymbol \nu \in \mathbb{R}^2_+ : \nu_1 + \nu_2 < \mu_1 + \mu_2, \nu_1= \mu_1 \}, \\
    \mathcal{F}_3 &= \{\boldsymbol \nu \in \mathbb{R}^2_+ : \nu_1 = \mu_1 , \nu_2= \mu_2 \}.
\end{align*}
The system operates in heavy traffic if the arrival rate $\boldsymbol \lambda$ is very close to a vector $\boldsymbol \nu \in \mathcal{F}$ i.e., there exists   $\boldsymbol \nu \in \mathcal{F}$, such that
\begin{align}
\label{eq: n_sys_arrival_vector}
    \lambda_1 = (1-\epsilon)\nu_1, && \lambda_1 + \lambda_2 =(1-\gamma \epsilon)( \nu_1 + \nu_2),
\end{align}
where $\epsilon$ is the heavy traffic parameter.
The parameter $\gamma>0$ defines the direction of approach of the arrival rate vector $\boldsymbol \lambda$ to the point $(\nu_1,\nu_2)$. This gives us that $\lambda_2 = (1-\gamma\epsilon) \nu_2 + \epsilon \nu_1 (1-\gamma)$. 
In this paper, we assume that the direction parameter $\gamma$ is independent of the heavy traffic parameter $\epsilon$, which means that $\boldsymbol\lambda $ approaches the vector $\boldsymbol\nu \in \mathcal{F}$ along a straight line, the slope of which depends on the parameter $\gamma$. 
\begin{remark}
    For $\boldsymbol \nu \in \mathcal{F}_1$, we have $\lambda_1+ \lambda_2 \rightarrow \mu_1+\mu_2$ and $\lambda_1 \rightarrow \nu_1 <\mu_1$ as $\epsilon \rightarrow 0$. Similarly, for $\boldsymbol \nu \in \mathcal{F}_2$, $\lambda_1 \rightarrow \mu_1$ and $\lambda_1+ \lambda_2 \rightarrow \nu_1+\nu_2 < \mu_1+\mu_2$ as $\epsilon \rightarrow 0$. In both the cases, the corresponding \nsys satisfies the CRP condition. However, when $\boldsymbol \lambda$ approaches a vector in  $\mathcal{F}_3$, that is,  $\boldsymbol \lambda \rightarrow (\mu_1,\mu_2)$, the CRP condition is not satisfied. 
\end{remark}

For $\mathcal{N}$-system, a scheduling algorithm corresponds to a policy that server $S_1$ uses to choose whether to serve the queue $Q_1$ or $Q_2$. In this work, for $\mathcal{N}$-system, we only consider the MaxWeight scheduling policy.
Under the MaxWeight scheduling policy, $S_1$ chooses the queue with higher queue length (with priority to $Q_2$ when tied), that is, if $q_2(t) \geq q_1(t)$, then both $S_1 $ and $S_2$ serve the jobs in the queue $Q_2$,  and if $   q_2(t) < q_1(t)$, $S_1$ serves $Q_1$ and $S_2$ serves $Q_2$. We assume that the jobs are served preemptively. Under MaxWeight, the queue length process $\{\mathbf{q}(t)\}_{t\geq0}$ forms a continuous time Markov chain, where $\mathbf{q}(t) = (q_1(t),q_2(t))$.  We use $G(\cdot, \cdot)$ to denote the generator matrix of the corresponding CTMC. We use $\pi_\epsilon$ to denote the steady-state distribution of the queue length vector. Also, we use $\q =(q_1,q_2)$ to denote the queue length vector that follows the steady-state distribution $\pi_\epsilon$.

\subsection{State-space collapse in the $\mathcal{N}$-system}
\label{sec: nsys_ssc}
In this section, we provide the definition of SSC for the $\mathcal{N}$-system under the MaxWeight scheduling policy. Consider the cone $\mathcal{K}_1$, $\mathcal{K}_2$ and $\mathcal{K}_3$ to be 
\begin{align*}
    \mathcal{K}_1 = \left\{ \mathbf{x}\in \mathbb{R}^2_+ : x_2 = x_1 \right\},&&\mathcal{K}_2 = \left\{ \mathbf{x}\in \mathbb{R}^2_+ : x_2 = 0 \right\}, &&\mathcal{K}_3 = \left\{ \mathbf{x}\in \mathbb{R}^2_+ : x_2 \leq x_1 \right\}.
\end{align*}
Note that the affine dimension of the cone $\mathcal{K}_1$ and $\mathcal{K}_2$ is just, while for the cone $\mathcal{K}_3$, the affine dimension is two. For any vector $\mathbf{y} \in \mathbb{R}^2_+$, let $\mathbf y_{\| \mathcal{K}_i}$ denotes the projection of vector $\mathbf y$ onto the cone $\mathcal K_i$ for $i \in \{1,2,3\}$. Then, 
\begin{align*}
    \mathbf y_{\| \mathcal{K}_1} = \frac{y_1+y_2}{2} \begin{bmatrix} 1\\1 \end{bmatrix}, && \mathbf y_{\| \mathcal{K}_2} = \begin{bmatrix} y_1\\0 \end{bmatrix}, &&  \mathbf{y}_{\| \mathcal{K}_3} =  \mathbf{y} \mathbf{1}_{\{y_2 \leq y_1\}} + \frac{y_1+y_2}{2} \begin{bmatrix} 1\\1 \end{bmatrix} \mathbf{1}_{\{y_2 > y_1\}}.
\end{align*}
Now, the perpendicular component $\mathbf{y}_{\perp \mathcal{K}_i} = \mathbf y - \mathbf y_{\| \mathcal{K}_i} $ for $i\in \{1,2,3\}$ is given by
\begin{align*}
    \mathbf y_{\perp \mathcal{K}_1} = \frac{y_2-y_1}{2} \begin{bmatrix} -1\\1 \end{bmatrix}, && \mathbf y_{\perp \mathcal{K}_2} = \begin{bmatrix} 0\\y_2 \end{bmatrix}, &&  \mathbf{y}_{\perp \mathcal{K}_3} =  \frac{y_2-y_1}{2} \begin{bmatrix} -1\\1 \end{bmatrix} \mathbf{1}_{\{y_2 > y_1\}}.
\end{align*}
As the focus of this paper is to characterize the limiting distribution when the system does not satisfy the CRP condition, in Appendix \ref{app: n_sys}, we have provided mathematical details mostly for the non-CRP case, i.e., when $\boldsymbol \lambda \rightarrow \boldsymbol \nu$ where $\boldsymbol\nu\in \mathcal{F}_3$, and so, in Appendix \ref{app: n_sys}, for the sake of convenience, we use the notation $\mathbf y_{\|} = \mathbf y_{\| \mathcal{K}_3} $ and $\mathbf y_{\perp} = \mathbf y_{\perp \mathcal{K}_3}$ for any vector $\mathbf y \in \mathbb R^2_+$.
% Then, for any vector $\mathbf{y} \in \mathbb{R}^2_+$, the projection of $\mathbf{y} $ onto the cone $\mathcal{K}$ is given by
% \begin{equation*}
%   \mathbf{y}_{\|} =  \mathbf{y} \mathbf{1}_{\{y_2 \leq y_1\}} + \frac{y_1+y_2}{2} \begin{bmatrix} 1\\1 \end{bmatrix} \mathbf{1}_{\{y_2 > y_1\}} \ \ \text{ and } \ \ \mathbf{y}_{\perp} = \mathbf y - \mathbf y_{\|}= \frac{y_2-y_1}{2} \begin{bmatrix} -1\\1 \end{bmatrix} \mathbf{1}_{\{y_2 > y_1\}}
% \end{equation*}

\begin{proposition}
    \label{prop: n_sys_ssc}
    Consider the \nsys as defined in Section \ref{sec: n_sys_model} operating under the MaxWeight scheduling policy. Pick any $i\in \{1,2,3\}$ and suppose $\boldsymbol \lambda \rightarrow \boldsymbol \nu$ as $\epsilon \rightarrow 0$ according to Eq. \eqref{eq: n_sys_arrival_vector}, where $\boldsymbol \nu \in \mathcal{F}_i$. 
    The MaxWeight scheduling policy achieves \textit{State Space Collapse} (SSC) onto the cone $\mathcal{K}_i$, that is,  there exists $\epsilon_0,\theta_0 >0$, such that for every $\theta <\theta_0$ and $\epsilon<\epsilon_0$, the steady-state queue length vector satisfies,
    \begin{equation*}
    \mathbb{E}_{\pi_\epsilon}[e^{ \theta  \|\q_{\perp \mathcal{K}_i} \| }] < C^\star < \infty,
    \end{equation*}
    where $\q_{\perp \mathcal{K}_i} = \q -\q_{\| \mathcal{K}_i} $. 
    As a consequence, for every $\theta > 0$, $\lim_{\epsilon \rightarrow 0} \mathbb{E}[e^{\epsilon \theta  \|\q_{\perp \mathcal{K}_i} \| }]  < \infty.$ Further, for every non-negative integer $r$, there  exists a $C_r$ independent of $\epsilon$ such that \[\mathbb{E} \big [\|\q_{\perp \mathcal{K}_i} \|^r \big] \leq C_r.\]
\end{proposition}

From Proposition \ref{prop: n_sys_ssc}, if $\boldsymbol \lambda $ approaches a vector $ \boldsymbol \nu \in \mathcal{F}_1$ or $ \boldsymbol \nu \in \mathcal{F}_2$  as $\epsilon \rightarrow 0$, then the SSC happens to a one-dimensional cone, as $\mathcal{K}_1$ and $\mathcal{K}_2$ are one-dimensional. This means that the system satisfies the CRP condition. If $\boldsymbol \lambda \rightarrow \boldsymbol \nu$ as $\epsilon \rightarrow 0$, where $\boldsymbol \nu \in \mathcal{F}_3$, the state space of the queue length vector does not collapse to a lower dimensional subspace, as $\mathcal{ K}_3$ is two-dimensional. However, the size of the state space reduces as $\mathcal{K}_3 \subsetneq \mathbb R^2_+$. 
The proof of SSC for MaxWeight when $\boldsymbol \lambda $ approaches a vector $ \boldsymbol \nu \in \mathcal{F}_3$ is provided in Appendix \ref{app: n_sys_ssc}.
% , which follows by using the result presented in \cite[Lemma 10]{Weina_bandwidth}. 
% by taking $V(\q) = (q_2-q_1)\mathbf{1}_{\{q_2>q_1\}}$ as the Lyapunov function.

% Also, using the arguments presented in \cite[Exercise 4.2]{srikant2014communication}, it can be shown that, the Markov chain $\{\q(t)\}_{t=0}^\infty$ is irreducible and aperiodic. Thus, we can use the Foster-Lyapunov Theorem \cite{meyn1993stability} and analyse the drift of an appropriately chosen Lyapunov function, which leads us to the conclusion that MaxWeight is throughput optimal, i.e., for any $\boldsymbol \lambda \in \mathcal{C}$, the corresponding Markov chain $\{\q(t)\}_{t=0}^\infty$ is stable. 

% Note that $\lambda_2> \mu_2$ when $\gamma < \frac{\mu_1}{\mu_1+\mu_2}$. 
% Next, we are going to model the steady state behaviour of the system when the heavy traffic parameter $\epsilon$ approaches zero.

\subsection{Results for $\mathcal{N}$-System}
\label{sec: n_sys_dist}
In this section, we present the functional equation (in Theorem \ref{thm: n_sys_mgf_eq}) that the limiting distribution for the $\mathcal{N}$-system satisfies under the non-CRP condition, that is, $\boldsymbol \lambda \rightarrow \boldsymbol (\mu_1,\mu_2)$. In Lemma \ref{lem: n_sys_uniqueness}, we prove that the solution to the functional equation is unique. Finally, 
We provide the heavy traffic distribution for $\mathcal{N}$-system under MaxWeight scheduling in Theorem \ref{thm: n_sys_distribution}, under  the condition that service rates are symmetric. 
\begin{theorem}
\label{thm: n_sys_mgf_eq}
 Consider the $\mathcal{N}$-system as defined in Section \ref{sec: n_sys_model}, operating under the MaxWeight scheduling policy. Suppose $\boldsymbol \lambda \rightarrow (\mu_1,\mu_2)$.  Let
 $\boldsymbol \Theta = \{\boldsymbol \theta \in\mathbb{C}^2: Re(\theta_1) \leq 0, Re(\theta_1 + \theta_2 )\leq  0\}$.
 Then, for all $\boldsymbol \theta \in \boldsymbol \Theta$, 
 \begin{align}
 \label{eq: n_sys_mgf_eq}
  \mathcal{P}(\TT):=  L(\boldsymbol \theta) \big[ \mu_2 (-\gamma\theta_2  +\theta_2^2)+\theta_2\mu_1(1-\gamma)+ \mu_1 (-\theta_1+\theta_1^2)\big] +\mu_1(\theta_1-\theta_2) M_1(\boldsymbol \theta)  +\theta_2 M_2(\boldsymbol \theta)   =0,
\end{align}
where $L(\boldsymbol \theta) =  \eplim\EEpe \left[e^{\epsilon(\theta_1 q_1 + \theta_2 q_2)}\right]$, 
% \begin{align*}
%    L(\boldsymbol \theta) =  \eplim\EEpe \left[e^{\epsilon(\theta_1 q_1 + \theta_2 q_2)}\right],
%    % =  \eplim \EEpe \left[e^{\epsilon(\theta_1 (q_1-q_2)\mathbf{1}_{\{q_1\geq q_2\}} + (\theta_1 + \theta_2)q_2 )}\mathbf{1}_{\mathcal A} \right],
% \end{align*}
and 
\begin{align*}
   M_1(\boldsymbol \theta) &= \eplim \frac{1}{\epsilon} \EEpe \left[e^{\epsilon(\theta_1 q_1 + \theta_2 q_2) }\mathbf{1}_{\{q_1\leq q_2\}}\right], \\
   M_2(\boldsymbol \theta) &= \eplim \frac{1}{\epsilon}  \EEpe[e^{\epsilon(\theta_1 q_1 + \theta_2 q_2) }\left(\mu_2\mathbf 1_{\{q_2 = 0\}}+\mu_1\mathbf 1_{\{q_2 =q_1=0\}} \right) ].
\end{align*}
% \begin{align*}
%    M_1(\boldsymbol \theta) = \eplim \frac{1}{\epsilon} \EEpe \left[e^{\epsilon (\theta_1 + \theta_2)q_2 }\mathbf{1}_{\{q_1\leq q_2\}}\right], && M_2(\boldsymbol \theta) = \eplim \frac{1}{\epsilon}  \EEpe[e^{\epsilon \theta_1 q_1 }\left(\mu_2\mathbf 1_{\{q_2 = 0\}}+\mu_1\mathbf 1_{\{q_2 =q_1=0\}} \right) ],
% \end{align*}
\end{theorem}

Theorem \ref{thm: n_sys_mgf_eq} provides the functional equation for the heavy traffic distribution of \nsys under the non-CRP condition, i.e., when the arrival rate vector $\boldsymbol \lambda$ approaches the point $(\mu_1,\mu_2)$ according to straight line with direction parameter $\gamma$. 
% Note that Theorem \ref{thm: n_sys_mgf_eq} holds for any value of $\mu_1$ and $\mu_2$, while in Theorem \ref{thm: n_sys_distribution} case (c), we took $\mu_1 = \mu_2$. This is only because we can solve Eq. \eqref{eq: n_sys_mgf_eq} under the condition $\mu_1 = \mu_2$. Finding the analytic solution Eq. \eqref{eq: n_sys_mgf_eq} when $\mu_1\neq \mu_2$ is in general quite hard. One approach is presented in \cite{franceschi2019integral}, where the solution can be represented as a complicated Cauchy integral by solving a properly defined boundary value problem. 

% As the Markov chain corresponding to the \nsys operating under MaxWeight scheduling is stable, 

The proof of Theorem \ref{thm: n_sys_mgf_eq} follows by equating the drift of an appropriate Lyapunov function to zero in steady state.
% is that for any stable scheduling policy, the expected drift of a well-defined  is zero in steady-state. 
For our analysis, we choose complex exponential as the Lyapunov function. As shown in the proof of Theorem \ref{thm: n_sys_mgf_eq}, $\boldsymbol \theta$ is chosen in such a way that the exponential Lyapunov function for \nsys is well-defined. By equating the drift of complex exponential function to zero in steady state and then by using a second-order approximation in terms of the heavy traffic parameter $\epsilon$, we obtain the functional equation given in Eq. \eqref{eq: n_sys_mgf_eq} for \nsys as $\epsilon \rightarrow 0$. The complete proof for Theorem \ref{thm: n_sys_mgf_eq} is provided in Appendix \ref{app: n_sys_mgf_eq}. 
\begin{remark}
In order to derive the limiting distribution of $\epsilon \q$ for the $\mathcal{N}$-system, it is sufficient to solve for $L(\TT)$ (in Theorem \ref{thm: n_sys_mgf_eq}) for $\TT\in \capTT$, as for $\mathcal{N}$-system, $\{\boldsymbol \theta \in\mathbb{C}^2: Re(\TT) \leq \mathbf 0_2\} \subseteq \capTT.$ As such, for $\mathcal{N}$-system, we do not need an argument as in Proposition \ref{prop: 3q_mgf_convergence} or Proposition \ref{prop: switch_mgf_convergence}.
\end{remark}

Similar to that for the Input-queued switch, we define $\mathcal{L}(\boldsymbol \Theta)$ to be class of functions such that 
\begin{align}
\label{eq: nsys_analytic_functions}
 \mathcal{L}(\boldsymbol \Theta) =  \{f: \boldsymbol \Theta \rightarrow \mathbb C \ : f \text{ is non-zero, holomorphic, continuous and bounded over }   \boldsymbol \Theta \}.
\end{align}

\begin{lemma}
\label{lem: nsys_analytic}
    Consider the $\mathcal{N}$-system as defined in Section \ref{sec: n_sys_model}, operating under the MaxWeight scheduling policy. Then, the functions $(L(\boldsymbol \theta), M_1(\boldsymbol \theta), M_2(\boldsymbol \theta))$ as defined in Theorem \ref{thm: n_sys_mgf_eq} lie in the set $\mathcal{L}(\boldsymbol \Theta)$.
\end{lemma}

Lemma \ref{lem: nsys_analytic} is same as Lemma \ref{lem: 3q_analytic} for the case of Three-queue system. By the definition of the function $(L(\boldsymbol \theta), M_1(\boldsymbol \theta), M_2(\boldsymbol \theta))$, they satisfy some specific properties, i.e. they belong to a certain class of functions given by $\mathcal{L}(\capTT)$. Our next lemma implies that the function equation in Theorem \ref{thm: n_sys_mgf_eq} has a unique solution in the class of functions given by $\mathcal{L}(\capTT)$. This would imply that there is a unique function that describes the Laplace transform of the heavy traffic queue length and also solves the functional equation. The argument for Lemma \ref{lem: nsys_analytic} follows from Lemma \ref{lem: nsys_mgf_equivalence} presented in Appendix \ref{app: n_sys}. 
Next, we claim that there is a unique solution to the functional equation.

\begin{lemma}
\label{lem: n_sys_uniqueness}
Consider the $\mathcal{N}$-system as defined in Section \ref{sec: n_sys_model}, operating under the MaxWeight scheduling policy. Then, there is a unique set of functions $(L(\boldsymbol \theta), M_1(\boldsymbol \theta), M_2(\boldsymbol \theta))$ such that $L(\boldsymbol \theta), M_1(\boldsymbol \theta)$ and $ M_2(\boldsymbol \theta)$ lie in $\mathcal{L}(\boldsymbol \Theta)$ and satisfies the functional equation $\mathcal{P}(\TT) =0, \forall \TT\in \capTT$, given in Eq. \eqref{eq: n_sys_mgf_eq} for all values of $\boldsymbol \theta \in \boldsymbol \Theta$. 
\end{lemma}
A crucial step to solve the functional equation given in  Theorem \ref{thm: n_sys_mgf_eq} is to ensure that it has a unique solution. Then, one can just guess the solution and check that it satisfies the functional equation. 
% And then, as the functional equation has a unique solution, we get that the guessed solution is the only solution. 
We exploit this argument to solve the functional equation under the condition $\mu_1=\mu_2$, and derive the result presented in Theorem \ref{thm: n_sys_distribution}.

The proof of Lemma \ref{lem: n_sys_uniqueness} is similar to that of Lemma \ref{lem: 3q_uniqueness}.
We consider a linear transformation of $\TT$ given by $\boldsymbol \psi = (\psi_1,\psi_2) =(\theta_1,\theta_1+\theta_2)$. Then, we show that $M_1(\TT)$ is a function of only $\psi_2$ and $M_2(\TT)$ is a function of only $\psi_1$. This allows us to rewrite the functional equation $\mathcal{P}(\TT) =0, \forall \TT\in \capTT$ in Eq. \eqref{eq: n_sys_mgf_eq} in a form that is consistent with the class of functional equations in Lemma \ref{lem: functional_uniqueness}, that is,  $\mathcal{P}(\TT) = \tilde{ \mathcal{P}}(\boldsymbol \psi) =0, \ \forall \boldsymbol \psi \in \boldsymbol \Psi$, where
\[\boldsymbol\Psi = \{ \boldsymbol \psi \in \mathbb C^2: Re(\boldsymbol \psi ) \leq \mathbf 0_2\}. \]
% A key element here is that the set $\capTT$ is a subset $\mathcal{S}$, and we exploit this to do the linear transformation functional equation as mentioned. 
Afterwards, by using Lemma \ref{lem: functional_uniqueness}, we imply that the functional equation $ \tilde{ \mathcal{P}}(\boldsymbol \psi) =0, \ \forall \boldsymbol \psi \in \boldsymbol \Psi$. This implies that there is a unique $L(\TT)$ that satisfies the functional equation $\mathcal{P}(\TT) =0, \forall \TT\in \capTT$ in Eq. \eqref{eq: 3q_functional_eq} for all values of $\TT\in\capTT$.
The complete proof of Lemma \ref{lem: n_sys_uniqueness} is provided in Appendix \ref{app: n_sys_uniqueness}. 

\begin{theorem}
\label{thm: n_sys_distribution}
Consider the $\mathcal{N}$-system defined in Section \ref{sec: n_sys_model}, operating under the MaxWeight scheduling policy. Suppose $\Upsilon_1$ and $\Upsilon_2$ are independent and exponentially distributed random variables with mean $1$ and  $\frac{1}{2\gamma}$ respectively.  
\begin{enumerate}[label=(\alph*), ref=\ref{thm: n_sys_distribution}.\alph*]
    \item \label{thm: n_sys_distribution_a} If $\boldsymbol \lambda \rightarrow \boldsymbol\nu$ for  $\boldsymbol \nu \in \mathcal{F}_1$ according to Eq. \eqref{eq: n_sys_arrival_vector}, the limiting distribution is given by,
    \begin{align*}
         \epsilon (q_1,q_2) \stackrel{d}{\rightarrow} (\Upsilon_2,\Upsilon_2).
    \end{align*}
    \item \label{thm: n_sys_distribution_b} If $\boldsymbol \lambda \rightarrow \boldsymbol\nu$ for  $\boldsymbol \nu \in \mathcal{F}_2$ according to Eq. \eqref{eq: n_sys_arrival_vector}, the limiting distribution is given by,
    \begin{align*}
        \epsilon (q_1,q_2) \stackrel{d}{\rightarrow}  (\Upsilon_1,0).
    \end{align*}
    \item \label{thm: n_sys_distribution_c} If $\boldsymbol \lambda \rightarrow \boldsymbol\nu $ for $ \boldsymbol \nu \in \mathcal{F}_3$ according to Eq. \eqref{eq: n_sys_arrival_vector}, and suppose $\mu_1 = \mu_2$, the limiting distribution is given by,
    \begin{align*}
     \epsilon (q_1,q_2) \stackrel{d}{\rightarrow}  (\Upsilon_1+\Upsilon_2,\Upsilon_2).
    \end{align*}
\end{enumerate}
\end{theorem}
Theorem \ref{thm: n_sys_distribution} provides the limiting distribution for $\mathcal N$-system in three different cases. When the system satisfies the CRP condition (i.e., SSC happens to a one dimensional cone), as in Theorem \ref{thm: n_sys_distribution_a} and Theorem \ref{thm: n_sys_distribution_b}, the limiting distribution is represented by a single exponential random variable. However, in Theorem \ref{thm: n_sys_distribution_c}, the system does not satisfy the CRP condition, that is, SSC happens to a two dimensional cone $\mathcal{K}_3$. 
If $\mu_1=\mu_2$, we observe that the limiting distribution is given by a linear combination of two independent and exponentially distributed random variables. Note that the functional equation in Theorem \ref{thm: n_sys_mgf_eq} holds for any value of $\mu_1$ and $\mu_2$, while in Theorem \ref{thm: n_sys_distribution_c}, we impose the condition $\mu_1 = \mu_2$. This is because we can easily solve Eq. \eqref{eq: n_sys_mgf_eq} under the condition $\mu_1 = \mu_2$. Finding the analytic solution Eq. \eqref{eq: n_sys_mgf_eq} when $\mu_1\neq \mu_2$ is, in general, quite hard. One approach is presented in \cite{franceschi2019integral}, where the solution can be represented as a Cauchy integral by solving a properly defined boundary value problem.

% In part (a) and (b), the \nsys satisfies the CRP condition and the analysis for these cases is very similar to the analysis presented in \cite{hurtado2020transform}. Our major focus in Theorem \ref{thm: n_sys_distribution} is part (c), which does not satisfy the CRP condition as both the queues are in heavy traffic. Part (c) of Theorem \ref{thm: n_sys_distribution} says that, if both the queues are in heavy traffic and the service rates are symmetric, the heavy traffic distribution of \nsys is represented using two independent exponential random variables. 
\begin{remark}
    The boundary sets $\mathcal{F}_1$ and $\mathcal{F}_2$ are line segments and the singleton set $\mathcal{F}_3$ contains the common end point of the line segments $\mathcal{F}_1$ and $\mathcal{F}_2$. Thus, intuitively, if the system approaches the boundary $\mathcal{F}_3$ in heavy traffic, it displays a behaviour which is a combination of the behaviour it displays on the other two boundaries $\mathcal{F}_1$ and $\mathcal{F}_2$. 
\end{remark}
% \textcolor{red}{Should we talk about CRP condition being satisfied on a facet and non-CRP on a smaller face, and on the smaller face, the system behaves like a combination of facets it comprises of.}
% \textcolor{red}{write something representing the relation between CRP and non-CRP}
% For $\mathcal{N}$-system, when the CRP condition holds, the heavy traffic steady-state distribution is exponentially distributed. Our result is consistent with the known results for \nsys under non-CRP conditions. 

In Theorem \ref{thm: n_sys_distribution_a} and Theorem \ref{thm: n_sys_distribution_b}, that SSC happens to $\mathcal{K}_1$ and $\mathcal{K}_2$, which are one-dimensional cone. Thus, the systems satisfies the CRP condition, and so the proof follows from the arguments provided in \cite{hurtado2020transform}. For the non-CRP case, that is, in Theorem \ref{thm: n_sys_distribution_c},
the proof uses the functional equation provided in Theorem \ref{thm: n_sys_mgf_eq}. Complete proof of Theorem \ref{thm: n_sys_distribution_c} is given in Appendix \ref{app: n_sys_distribution}.

\section{Proof outline for Three-queue system}
\label{sec: 3q_outline}

In this section, we present a brief outline of the proof of the results presented in  Section \ref{sec: 3q_results}. The complete proof for all the results in provided in Appendix \ref{app: 3q}. Note that, for this section, we are using the notations presented in Section \ref{sec: 3q_model}. 
The arguments provided in this section present much of the mathematical complexity required to prove the results mentioned in previous sections while omitting most of the technical details details. 

\subsection{Essential Lemmas}
\label{sec: mgfplusuniqueness}

\begin{lemma}
    \label{lem: laplace_convergence}
    Suppose $\{\mathbf X_n\}_{n\in \mathbb N}  $ is a sequence of random vector in $\mathbb R^d_+$ such that $\mathbf X_n$ follows the distribution $\pi_n$. Let $\tilTT \in \tilde{\boldsymbol \Theta}$, where $\tilde{\boldsymbol \Theta} = \{\tilTT \in \mathbb C^d : Re(\tilTT) \leq \mathbf 0_d\}.$
    If the sequence of Laplace transform $L_n(\tilTT) = \EE_{\pi_n} \big[ e^{\langle \tilTT, \mathbf X_n \rangle}\big]$ converges pointwise to $L(\tilTT)$ for all $\tilTT\in \tilde{\boldsymbol \Theta}$, where $L(\tilTT)$ is the Laplace transform of a random vector $\mathbf X$, then, $\mathbf X_n$ converges weakly to random variable $\mathbf X$, that is
    \[\mathbf X_n \stackrel{d}{\rightarrow} \mathbf X.\]
\end{lemma}

Lemma \ref{lem: laplace_convergence} follows simply by using \cite[Theorem 26.3]{Bill86}.
Next, we present the uniqueness result for a functional equation with two variables.
\begin{lemma}
\label{lem: functional_uniqueness}
Consider a functional equation of the form 
\begin{equation}
    \label{eq: functional}
    \gamma(\boldsymbol \psi) \Phi (\boldsymbol \psi) + \gamma_1(\boldsymbol \psi) \Phi_1 (\psi_2) +\gamma_2(\boldsymbol \psi) \Phi_2 (\psi_1) = 0,
\end{equation}
for all $\boldsymbol \psi \in \boldsymbol\Psi = \{\mathbf x \in \mathbb{C}^2:  Re( \mathbf x ) \leq \mathbf 0_2\}$, where $\Phi (\boldsymbol \psi), \Phi_1 (\psi_2)$ and $\Phi_2 (\psi_1)$ are analytic, continuous and bounded functions over the domain $\boldsymbol\Psi$, and $\gamma(\boldsymbol \psi), \gamma_1(\boldsymbol \psi)$ and $\gamma_2(\boldsymbol \psi)$ are given by
\begin{align*}
    \gamma(\boldsymbol \psi) &= \alpha_1 \psi_1 + \alpha_2 \psi_2 + \frac{1}{2} (\sigma_{11} \psi_1^2 + 2\sigma_{12} \psi_1\psi_2 + \sigma_{22}\psi_2^2),\\
    \gamma_1 (\boldsymbol \psi) &= r_{11} \psi_1 + r_{21}\psi_2,\\
     \gamma_2 (\boldsymbol \psi)  &= r_{12} \psi_1 + r_{22}\psi_2.
\end{align*}
Suppose the following conditions are satisfied,
\begin{align}
\label{eq: funceqconditions}
    r_{11} > 0, && r_{22} >0, && r_{11}r_{22}  -r_{12}r_{21} >0 && r_{22}\alpha_1  -r_{12}\alpha_{2} <0, && r_{11}\alpha_2  -r_{21}\alpha_{1} <0. 
\end{align}
Also, suppose $\Phi(\mathbf 0_2), \Phi_1 (0)$ and $\Phi_2 (0)$ are given. Then, there is a unique solution $(\Phi (\boldsymbol \psi), \Phi_1 (\psi_2),\Phi_2 (\psi_1))$ that satisfies the functional equation given in Eq. \eqref{eq: functional} for all $\boldsymbol \psi \in \boldsymbol \Psi$.
\end{lemma}

The arguments behind the proof of Lemma \ref{lem: functional_uniqueness} is provided in Appendix \ref{app: functional_uniqueness}.

% \endproof

% \color{red}
% \proof{\textit{Proof of Lemma \ref{lem: functional_uniqueness}}}
% The proof of Lemma \ref{lem: functional_uniqueness} uses the result provided in \cite[Theorem 1]{franceschi2019integral}. Consider a SRBM \cite{williams1995semimartingale}\cite[Equation 1]{franceschi2019integral} with drift in the quarter plane $\mathbb{R}_+^2$. Suppose following are the parameters of the SRBM: $\Sigma$ is the covariance matrix of the Brownian motion, $\boldsymbol \alpha$ denotes the interior drift and $\mathbf R$ be the reflection matrix given by,
% \begin{align*}
%     \Sigma = \begin{bmatrix} \sigma_{11} & \sigma_{12}\\ \sigma_{21} & \sigma_{22} \end{bmatrix}, && \boldsymbol \alpha = \begin{bmatrix} \alpha_1\\ \alpha_2 \end{bmatrix}, && \mathbf R = \begin{bmatrix} r_{11} & r_{12}\\ r_{21} & r_{22} \end{bmatrix}.
% \end{align*}
% Then, for this SRBM, the functional equation, as presented in  \cite{franceschi2019integral} matches with the functional equation as given in Eq. \eqref{eq: functional}. Then, from the result provided in \cite[Theorem 1]{franceschi2019integral}, we know that there is a unique solution to the functional equation given in Eq. \eqref{eq: functional}. The result in \cite[Theorem 1]{franceschi2019integral} provide the unique solution to the functional equation in terms of Cauchy integrals. For the proof of Lemma \ref{lem: functional_uniqueness}, we only use the fact that the solution is unique. \hfill $\blacksquare$
% \endproof
% \color{black}

\subsection{Proof outline for Theorem \ref{thm: 3q_functional_eq} and Theorem \ref{thm: 3q_dist}}

\proof{\textit{Proof outline for Theorem \ref{thm: 3q_functional_eq}}.}
Suppose $\capTT^\epsilon \subseteq \capTT$ is the set of $\TT$ such that $\left|\EEpe \Big[ e^{ \epsilon \langle \boldsymbol{\theta}, \mathbf{q} \rangle } \Big]\right| < \infty $ for any $\TT\in \capTT^\epsilon$ (see Lemma \ref{lem: 3q_mgf_equivalence} for more details on $\capTT^\epsilon$). 

\textbf{Step 1:} [Exponential Lyapunov Function] Recall that $\q$ and $\q^+$ satisfies the equation $\q^+ = \q + \mathbf a - \mathbf s + \mathbf u$. In the first step, we convert this into an exponential form as follows.
For any $\TT\in \capTT^\epsilon$,
\begin{align*}
  \EEpe \Big[  e^{ \epsilon \langle \boldsymbol{\theta}, \mathbf{q}^+ \rangle } \Big( e^{- \epsilon \langle \boldsymbol{\theta}, \mathbf{u} \rangle}  -1 \Big) \Big] 
  & = \EEpe \Big[ e^{ \epsilon \langle \boldsymbol{\theta}, \mathbf{q} \rangle } e^{ \epsilon \langle \boldsymbol{\theta}, \mathbf{a }- \mathbf{s } \rangle } \Big] - \EEpe \Big[ e^{ \epsilon \langle \boldsymbol{\theta}, \mathbf{q}^+ \rangle}\Big] \allowdisplaybreaks \nonumber \\
  & \stackrel{(a)}{=} \EEpe \Big[ e^{ \epsilon \langle \boldsymbol{\theta}, \mathbf{q} \rangle } \Big] \EEpe \Big[ e^{ \epsilon \langle \boldsymbol{\theta}, \mathbf{a }- \mathbf{s } \rangle } \Big] - \EEpe \Big[ e^{ \epsilon \langle \boldsymbol{\theta}, \mathbf{q}^+ \rangle}\Big]. \allowdisplaybreaks \nonumber 
  % \\
  % & \stackrel{(b)}{=}  \EEpe \Big[ e^{ \epsilon \langle \boldsymbol{\theta}, \mathbf{q} \rangle } \Big] \Bigg( \EEpe \Big[ e^{ \epsilon \langle \boldsymbol{\theta}, \mathbf{a }- \mathbf{s } \rangle } \Big] - 1\Bigg),
\end{align*}
where (a) follows as the arrivals are independent of the queue length vector and $\langle \TT,\s\rangle = \frac{1}{2}\langle \TT, \mathbf 1_3 \rangle $ for all $\TT \in \capTT$, so $\langle \TT,\s\rangle$ is independent of $\q$.

\textbf{Step 2:} [Zero drift] As $\q$ follows the steady state distribution $\pi_\epsilon$,  $\q^{+}$ also follows the steady state distribution $\pi_\epsilon$. Thus, $\EEpe \Big[ e^{ \epsilon \langle \boldsymbol{\theta}, \mathbf{q} \rangle } \Big] = \EEpe \Big[ e^{ \epsilon \langle \boldsymbol{\theta}, \mathbf{q}^+ \rangle } \Big]$. This gives us, 
\begin{align*}
  \EEpe \Big[  e^{ \epsilon \langle \boldsymbol{\theta}, \mathbf{q}^+ \rangle } \Big( e^{- \epsilon \langle \boldsymbol{\theta}, \mathbf{u} \rangle}  -1 \Big) \Big] 
  & = \EEpe \Big[ e^{ \epsilon \langle \boldsymbol{\theta}, \mathbf{q} \rangle } \Big] \Bigg( \EEpe \Big[ e^{ \epsilon \langle \boldsymbol{\theta}, \mathbf{a }- \mathbf{s } \rangle } \Big] - 1\Bigg).
\end{align*}
Also, we have $\eplim \capTT^\epsilon = \capTT$ (again, see Lemma \ref{lem: 3q_mgf_equivalence}), and thus, by taking $\epsilon \rightarrow 0$, for all $\TT\in \capTT$, we have
\begin{align*}
    \eplim \frac{1}{\epsilon^2} \EEpe \Big[  e^{ \epsilon \langle \boldsymbol{\theta}, \mathbf{q}^+ \rangle } \Big( e^{- \epsilon \langle \boldsymbol{\theta}, \mathbf{u} \rangle}  -1 \Big) \Big] = \eplim \frac{1}{\epsilon^2} \EEpe \Big[ e^{ \epsilon \langle \boldsymbol{\theta}, \mathbf{q} \rangle } \Big] \Bigg( \EEpe \Big[ e^{ \epsilon \langle \boldsymbol{\theta}, \mathbf{a }- \mathbf{s } \rangle } \Big] - 1\Bigg).
\end{align*}

\textbf{Step 3:} [Second order approximation] From Lemma \ref{app: 3q_2ndorderapprox}, we have
\begin{align*}
    \lim_{\epsilon \rightarrow 0} \frac{1}{\epsilon^2} \EEpe \Big[e^{ \epsilon \langle \boldsymbol{\theta}, \mathbf{q}^+ \rangle } \Big( e^{- \epsilon \langle \boldsymbol{\theta}, \mathbf{u} \rangle}  -1 \Big) \Big] &=  - \Bigg \langle  \boldsymbol{\theta}, \lim_{\epsilon \rightarrow 0} \frac{1}{\epsilon} \EEpe \Big[  \mathbf{u} e^{ \epsilon \langle \boldsymbol{\theta}, \mathbf{q} \rangle} \Big] \Bigg \rangle =  - \langle \boldsymbol{\theta}, \mathbf{M}(\boldsymbol{\theta}) \rangle,\\
    \eplim \frac{1}{\epsilon^2} \Bigg( \EEpe \Big[ e^{ \epsilon \langle \boldsymbol{\theta}, \mathbf{a }- \mathbf{s } \rangle } \Big] - 1\Bigg) & =  - \frac{1}{2} \langle \boldsymbol \theta, \mathbf 1_3 \rangle + \frac{1}{2} \langle \boldsymbol \theta, \boldsymbol \sigma^2 \boldsymbol \theta \rangle .
\end{align*}
Using this, we arrive at the functional equation
\begin{align*}
    - \langle \boldsymbol{\theta}, \mathbf{M}(\boldsymbol{\theta}) \rangle
   & = \left( - \frac{1}{2} \langle \boldsymbol \theta, \mathbf 1_3 \rangle + \frac{1}{2} \langle \boldsymbol \theta, \boldsymbol \sigma^2 \boldsymbol \theta \rangle \right) L(\boldsymbol{\theta}), \ \ \forall \TT\in \capTT,
\end{align*}
where, 
\begin{align*}
    L(\boldsymbol \theta) = \lim_{\epsilon \rightarrow 0} \EEpe[e^{\epsilon \langle \boldsymbol \theta, \q \rangle}], && M_2(\boldsymbol \theta)= \lim_{\epsilon \rightarrow 0} \frac{1}{\epsilon} \EEpe [u_2e^{\epsilon \langle \boldsymbol \theta, \q \rangle}], && M_3(\boldsymbol \theta)= \lim_{\epsilon \rightarrow 0} \frac{1}{\epsilon} \EEpe [u_3e^{\epsilon \langle \boldsymbol \theta, \q \rangle}].
\end{align*}
Note that $M_1(\boldsymbol\theta)=0$ because $u_1=0$ by the definition of the service process.
 \hfill $\blacksquare$
\endproof

\proof{\textit{Proof outline for Theorem \ref{thm: 3q_dist}}.}
The Laplace transform of $\mathbf B \boldsymbol \Upsilon$, for any $\TT \in \boldsymbol \Theta$, is given by 
\begin{align*}
    \mathbb{E}[e^{\theta_1 (\Upsilon_1+\Upsilon_2)+\theta_2 \Upsilon_1 + \theta_3 \Upsilon_2 }]  &= \mathbb{E}[e^{(2\theta_2 + \theta_3) \Upsilon_1+(\theta_2 + 2\theta_3)\Upsilon_2}] \\
    &= \frac{1}{\bigg( 1- (2\theta_2 + \theta_3) \frac{3\sigma_2^2 + \sigma_3^2}{8}\bigg)\bigg( 1- (\theta_2 + 2\theta_3)\frac{\sigma_2^2 + 3\sigma_3^2}{8}\bigg)}.
\end{align*}
Now pick
\begin{align} 
    L(\boldsymbol \theta) = \frac{1}{\bigg( 1- (2\theta_2 + \theta_3) \frac{3\sigma_2^2 + \sigma_3^2}{8}\bigg)\bigg( 1- (\theta_2 + 2\theta_3)\frac{\sigma_2^2 + 3\sigma_3^2}{8}\bigg)},\nonumber\\ M_2 (\boldsymbol \theta)  = \frac{1}{1- (2\theta_2 + \theta_3)\frac{\sigma_2^2 + 3\sigma_3^2}{8}}, && M_3 (\boldsymbol \theta) = \frac{1}{1- (\theta_2 + 2\theta_3)\frac{3\sigma_2^2 + \sigma_3^2}{8}}.
\end{align}
We can show that $(L(\TT),M_2(\TT),M_3(\TT))$ given above satisfies the functional equation  $\mathcal{P}(\TT) =0, \ \forall\TT\in \capTT$ given in Eq. \eqref{eq: 3q_functional_eq} under the condition $2\sigma^2_1 = \sigma_2^2 +\sigma_3^2$. And from Lemma \ref{lem: 3q_uniqueness}, as the solution to the functional equation is unique, this has to be the only solution of the functional equation. 
Further, $L(\TT)$ matches with the Laplace transform of $\mathbf B \boldsymbol \Upsilon$ for all $\TT \in \boldsymbol \Theta$, that is 
\begin{align*}
    L(\TT) = \EE \big[e^{\langle \TT,\mathbf B \boldsymbol \Upsilon\rangle}\big], \ \forall \TT\in\capTT.
\end{align*}
Now, the result follows from Proposition \ref{prop: 3q_mgf_convergence}. \hfill $\blacksquare$
\endproof

\section{Discussion \& Future work}
\label{sec: discussion}

In this paper, we looked at queueing systems that do not satisfy the CRP condition and developed a technique using the Laplace transform to specify the heavy traffic distribution. The idea is to use complex exponential as the test function and set its drift to zero in steady-state. If the system satisfies the CRP condition, then this analysis gives the explicit closed-form expression for the Laplace transform of the heavy traffic distribution. For a non-CRP system, the same analysis gives an implicit equation which is termed the functional equation of the system. For the considered systems, we characterized their heavy traffic distribution using the functional equation and provided the solution to the functional equation under some specific conditions on the system parameters. Before concluding this paper, we present a few small remarks and future directions for this work.

% In this section, we provide a small discussion on the results presented in this paper and also present a few future directions and open problems.

% \subsubsection{Unbounded arrivals} In case of Input-queued 

\subsection{Uniqueness of functional equation for Input-queued switch} Proving the uniqueness of the solution the functional equation for the Input-queued switch has turned out to be a difficult task. The ideas presented in \cite{franceschi2019integral} are not enough to prove the statement given by Conjecture \ref{lem: switch_uniqueness}. One idea is to look at the extensions of the Carleman Boundary value problem and then attempt a similar technique as presented in \cite{franceschi2019integral}. Completing the proof for Conjecture \ref{lem: switch_uniqueness} is crucial to extend the results presented in this paper to more general SPNs.

\subsection{Heavy traffic distribution under general variance condition} For all three queueing systems considered in this paper, we have shown that the heavy traffic distribution of the steady-state scaled queue length vector can be represented by independent exponential random variables under some specific condition on the variance of the arrival process. The authors in \cite{franceschi2019integral} use the theory of the Carleman Boundary value problem to solve for the heavy traffic distribution when the corresponding functional equation consists of two variables, and provide the solution as a Cauchy integral which are hard to interpret. Finding the heavy traffic distribution under more general variance conditions is still an open problem.

% \subsubsection{Non-MaxWeight algorithms} In this work, we presented the heavy traffic distribution of the queueing systems operating under the MaxWeight scheduling algorithm. Similar results can be obtained if the system is operating under scheduling algorithms for which the state space collapse of the system is the same as that for the MaxWeight. For Input-queued switch, in \cite{jhunjhunwala2021low}, authors presented a class of scheduling algorithms that achieves the same state-space collapse as MaxWeight. As such, the results presented in this work are not restricted to a specific scheduling algorithm and are valid for the algorithms considered in \cite{jhunjhunwala2021low} also.

% Appendix here
% Options are (1) APPENDIX (with or without general title) or 
%             (2) APPENDICES (if it has more than one unrelated sections)
% Outcomment the appropriate case if necessary
%
% \begin{APPENDIX}{<Title of the Appendix>}
% \end{APPENDIX}
%
%   or 
%
% \begin{APPENDICES}
% \section{<Title of Section A>}
% \section{<Title of Section B>}
% etc
% \end{APPENDICES}

% Acknowledgments here
% \section*{Acknowledgments.}
% Enter the text of acknowledgments here

\bibliographystyle{informs2014} 
\bibliography{references.bib}
\renewcommand{\theHsection}{A\arabic{section}}
\begin{APPENDICES}

% \section{Convergence of Laplace transform}

\section{Proofs for Three-queue system}
\label{app: 3q}

Recall that $\pi^\epsilon$ denotes the steady state distribution of the Three-queue system with the heavy traffic parameter $\epsilon$, and as the underlying Markov chain is positive recurrent, $\pi^\epsilon$ exists and is unique. Further, we use $\mathbb E_{\pi^\epsilon}[\cdot]$ to denote the expectation under the probability distribution $\pi_\epsilon$. Further, $\epsilon_0$ and $\theta_0$ are as given in Definition \ref{def: 3q_ssc}.

\subsection{Properties of the projection}
\label{app: 3q_projection}
\begin{lemma}
    \label{lem: 3q_projection}
     We define following matrices: 
    \begin{align*}
        \mathbf B = \begin{bmatrix}
                        1 & 1\\
                        1 &  0\\
                        0 & 1
                \end{bmatrix}, && 
        \mathbf D = \mathbf B^T\mathbf B = \begin{bmatrix}
                        2 & 1\\
                        1 &  2
                \end{bmatrix}, && 
        \mathbf D^{-1} = \frac{1}{3}  \begin{bmatrix}
                        2 & -1 \\
                        -1 &  2
                \end{bmatrix}, && 
        \mathbf A = \mathbf{B}(\mathbf{B}^T\mathbf{B})^{-1}\mathbf{B}^T= \frac{1}{3} \begin{bmatrix}
                        2 & 1 & 1\\
                        1 &  2 & -1\\
                        1&-1&2
                \end{bmatrix}.       
    \end{align*}
    Let $\mathbf x \in \mathbb{C}^3$ and suppose $\mathbf x_{\|}$ denotes the projection of $\mathbf{x}$ onto the space $\mathcal{S}$, where $\mathcal{S}$ is the space spanned by the columns of $\mathbf B$. And $\mathbf x_{\perp} = \mathbf x - \mathbf x_{\|}$.
    \begin{enumerate}
    \item For any $\mathbf x \in \mathbb{C}^3$ and $\TT \in \mathcal{S}$, 
    \begin{equation*}
        \langle \boldsymbol \theta, \mathbf{x} \rangle = \langle \boldsymbol \theta, \mathbf{x}_{\|} \rangle =  (2\theta_2+\theta_3) x_{\| 2} + (\theta_2+2\theta_3) x_{\| 3}.
    \end{equation*}
    % where $\mathbf{d}_1$ and $\mathbf{d}_2$ are columns of $\mathbf D $.
    \item The closed form expression for $\mathbf x_{\|}$ is given by $\mathbf x_{\|} = \mathbf A \mathbf{x}$. And the perpendicular component $\mathbf x_\perp$ is 
    \begin{equation*}
    \mathbf x_{\perp} = \frac{1}{3}(x_2+x_3 -x_1)\begin{bmatrix}
        -1\\1\\1
    \end{bmatrix}.
\end{equation*}
    \end{enumerate}
\end{lemma}

Lemma \ref{lem: 3q_projection} provides the properties of the projection of any vector onto the space $\mathcal{S}$. By the definition of the matrix $\mathbf{A}$ we have the relation that $\mathbf{B}^T \mathbf{A} = \mathbf{B}^T$ and $\mathbf{A}\mathbf{B} = \mathbf{B}$. The proof of Lemma \ref{lem: 3q_projection} follows by simple application of linear algebra and the mathematical details, as provided below. 
% For the queue length vector $\q$, we denote $\q_{\|}$ as the projection of $\q$ on the space $\mathcal{S}$ and $\q_{\perp}$ as the perpendicular component, i.e., $\q_{\perp} = \q - \q_{\|}$. Using $\q_{\|} = \mathbf{A} \q$, we have that
% \begin{equation*}
%     \q_{\perp} = \frac{1}{3}(q_1+q_2 -q_3)\begin{bmatrix}
%         1\\1\\-1
%     \end{bmatrix}.
% \end{equation*}

\proof{\textit{Proof of Lemma \ref{lem: 3q_projection}}.}
 Part 1 follows by using the structure of the subspace $\mathcal{S}$. Observe that as $\boldsymbol \theta \in \mathcal{S}$, we have that $\langle \boldsymbol \theta, \mathbf{x}_\perp \rangle = 0$. This gives us that
 \begin{align*}
     \langle \boldsymbol \theta, \mathbf{x} \rangle = \langle \boldsymbol \theta, \mathbf{x}_{\|} \rangle &= \theta_1 (x_{\| 2}+x_{\| 3})+\theta_2 x_{\| 2} +\theta_3 x_{\| 3} \\
     & =  (\theta_2+\theta_3)(x_{\| 2}+x_{\| 3})+\theta_2 x_{\| 2} +\theta_3 x_{\| 3}\\
     & = (2\theta_2+\theta_3)x_{\| 2}+(\theta_2+2\theta_3)x_{\| 3 }.
 \end{align*}
 This completes the proof of Part 1. The proof of Part 2 follows simply by using the theory of projection to a linear subspace. \hfill $\blacksquare$
\endproof

\subsection{Required Lemma}
\label{app: 3q_mgf_equivalence}
Recall that 
\begin{align*}
    \boldsymbol \Theta & = \{\boldsymbol \theta \in \mathbb{C}^{3} : \boldsymbol \theta \in \mathcal{S}, \ Re(\mathbf B^T \boldsymbol \theta) \leq  \mathbf 0_{2}\}\\
    &= \{\boldsymbol \theta \in \mathbb{C}^{3} : \theta_1 = \theta_2+\theta_3, Re(2\theta_2+\theta_3)\leq 0, Re(\theta_2+2\theta_3)\leq 0\},\\
    \tilde{ \boldsymbol \Theta } & = \{\tilTT \in \mathbb{C}^{3}: Re(\tilTT) \leq \mathbf 0_{3}\}.
\end{align*}
Further, for any $\tilTT \in \tilde{ \boldsymbol \Theta }$, we can write $\tilTT = \TT +\TT_\perp$ such that $\TT\in \capTT$ and $\TT_\perp \in \mathcal{ S}^\perp$. 
 Before presenting the proof of the results for Three-queue system, we present a necessary Lemma as given below.
\begin{lemma}
\label{lem: 3q_mgf_equivalence}
Consider the Three-queue system as defined in Section \ref{sec: 3q_model} operating under scheduling policy that achieves SSC according to the Definition \ref{def: 3q_ssc}. 
\begin{enumerate}[label=(\alph*), ref=\ref{lem: 3q_mgf_equivalence}.\alph*]
    \item \label{lem: 3q_mgf_equivalence_a} For any $\epsilon< \epsilon_0$ and $\TT\in \boldsymbol \Theta^\epsilon = \boldsymbol \Theta \cap \{\TT \in \mathbb C^3: \|\TT\| \leq \theta_0/36\epsilon\}$, we have 
    \begin{align*}
        \big| \EEpe  \big[ e^{\epsilon \langle \TT , \q \rangle} \big] \big| < \infty
    \end{align*}
    \item \label{lem: 3q_mgf_equivalence_b} Suppose $\tilde{\boldsymbol \theta} \in \tilde{\boldsymbol \Theta}$, such that $\tilTT = \TT +\TT_{\perp}$ such that $\TT \in \capTT$ and $\TT_{\perp} \in \mathcal{ S}^\perp$.  Then, we have
    \begin{align*}
        \lim_{\epsilon\rightarrow 0}  \big| \EEpe  \big[ e^{\epsilon \langle \tilTT , \q \rangle} \big] \big| \leq 1, && \lim_{\epsilon \rightarrow 0}\EEpe[e^{\epsilon \langle \tilde{\boldsymbol \theta}, \q \rangle}] = \lim_{\epsilon \rightarrow 0} \EEpe[e^{\epsilon \langle \boldsymbol \theta, \q \rangle}] = \lim_{\epsilon \rightarrow 0} \EEpe[e^{\epsilon ((2\theta_2+\theta_3) q_2 + (\theta_2+2\theta_3) q_3)}].
    \end{align*}
    \item \label{lem: 3q_mgf_equivalence_c} For any $\boldsymbol \theta \in \boldsymbol \Theta$, we have 
    \begin{align*}
        \lim_{\epsilon \rightarrow 0} \frac{1}{\epsilon} \big|\EEpe[u_2 e^{\epsilon \langle \TT , \q \rangle}] \big|\leq 1, && \lim_{\epsilon \rightarrow 0} \frac{1}{\epsilon}\EEpe[u_2 e^{\epsilon \langle \boldsymbol \theta, \q \rangle}] 
    =  \lim_{\epsilon \rightarrow 0}\frac{1}{\epsilon} \EEpe[u_2 e^{\epsilon (\theta_2+2\theta_3) q_3}],
    \end{align*}
    and 
    \begin{align*}
        \lim_{\epsilon \rightarrow 0} \frac{1}{\epsilon} \big|\EEpe[u_3 e^{\epsilon \langle \TT , \q \rangle}] \big|\leq 1, && \lim_{\epsilon \rightarrow 0} \frac{1}{\epsilon}\EEpe[u_3 e^{\epsilon \langle \boldsymbol \theta, \q \rangle}] 
    =  \lim_{\epsilon \rightarrow 0}\frac{1}{\epsilon} \EEpe[u_3 e^{\epsilon (2\theta_2+\theta_3) q_2}].
    \end{align*}
\end{enumerate}
\end{lemma}

A main consequence of Lemma \ref{lem: 3q_mgf_equivalence} is that the limiting terms $(L(\TT),M_2(\TT),M_3(\TT))$ exists for any $\TT\in \capTT$. This condition is required to establish the functional equation in Theorem \ref{thm: 3q_functional_eq}. 
% Further, according to Lemma \ref{lem: 3q_mgf_equivalence_b}, in order to characterize the heavy traffic steady state distribution of the Three-queue system, we just need to the consider the set of $\boldsymbol \theta$ that lie in $\mathcal{S}$. 
The result in Lemma \ref{lem: 3q_mgf_equivalence} is a consequence of the state space collapse of the Three-queue system onto the subspace $\mathcal{S}$.

\proof{\textit{Proof of Lemma \ref{lem: 3q_mgf_equivalence}}.}
 As given, $\tilde{\boldsymbol \theta} \in \tilde{\boldsymbol \Theta}$, such that $\tilTT = \TT +\TT_{\perp}$ such that $\TT \in \capTT$ and $\TT_{\perp} \in \mathcal{ S}^\perp$. Then, 
\begin{align}
\label{eq: 3q_theta_relation}
  \langle \tilde{\boldsymbol \theta}, \q \rangle &= \langle \boldsymbol  \theta , \q \rangle  + \langle \boldsymbol  \theta_{\perp} , \q \rangle \nonumber\\
  & \stackrel{(a)}{=} \langle \boldsymbol  \theta , \q \rangle  + \langle \boldsymbol  \theta_{\perp} , \q_{\perp} \rangle \nonumber\\
  &\stackrel{(b)}{=} (\theta_2 +\theta_3)q_1 + \theta_2 q_2 + \theta_3 q_3 + \langle \boldsymbol  \theta_{\perp} , \q_{\perp} \rangle\nonumber\\
  & = (2\theta_2 +\theta_3)q_2 +(\theta_2 +2\theta_3) q_3 + (\theta_2 +\theta_3) (q_1 - q_2 - q_3) + \langle \boldsymbol  \theta_{\perp} , \q_{\perp} \rangle\nonumber\\
  & \stackrel{(c)}{=}(2\theta_2+\theta_3) q_2 + (\theta_2+2\theta_3) q_3 - \langle 3(\theta_2 +\theta_3)\mathbf{1}_3, \q_{\perp} \rangle+ \langle \boldsymbol  \theta_{\perp} , \q_{\perp} \rangle\nonumber\\
  & \stackrel{(d)}{=} (2\theta_2+\theta_3) q_2 + (\theta_2+2\theta_3) q_3 +  \langle \boldsymbol  \theta' , \q_{\perp} \rangle,
\end{align}
where (a) follows because $\langle \boldsymbol  \theta_{\perp} , \q_\| \rangle =0$ as $\boldsymbol  \theta_{\perp} \in \mathcal{S}^\perp$ and $\q_\| \in \mathcal S$; (b) follows by using $\boldsymbol \theta \in \mathcal{S}$ and so $\theta_1= \theta_2+\theta_3$; (c) follows by using Part 2 of Lemma \ref{lem: 3q_projection}; and (d) follows by taking $\boldsymbol  \theta' = \boldsymbol  \theta_{\perp} - 3(\theta_2+\theta_3)\mathbf 1_3$. 
As a consequence of calculation in Eq. \eqref{eq: 3q_theta_relation}, we also have $\langle \boldsymbol  \theta , \q \rangle = (2\theta_2+\theta_3) q_2 + (\theta_2+2\theta_3) q_3 - \langle 3(\theta_2 +\theta_3)\mathbf{1}_3, \q_{\perp} \rangle $ for any $\boldsymbol \theta \in \boldsymbol \Theta$. 
Now, suppose $X$ is a random variable, which is measurable with respect to the probability measure $\pi_\epsilon$. Then,
\begin{align}
    \label{eq: 3q_mgf_equi_part1}
        & \left| \EEpe[ X e^{\epsilon \langle \tilde{\boldsymbol \theta}, \q \rangle}] - \EEpe[X e^{\epsilon ((2\theta_2+\theta_3) q_2 + (\theta_2+2\theta_3) q_3)}] \right| \nonumber \\
        & \ \ \ \ = \EEpe\big[|X|\left|e^{\epsilon ((2\theta_2+\theta_3) q_2 + (\theta_2+2\theta_3) q_3)} \right| \big|  \big( 1- e^{\epsilon \langle \boldsymbol \theta', \q_{\perp} \rangle}   \big) \big|\big] \nonumber\\
        & \ \ \ \ \stackrel{(a)}{\leq} \EEpe\bigg[|X|\left|  1- e^{\epsilon \langle \boldsymbol \theta', \q_{\perp} \rangle}   \right| \bigg]\nonumber \allowdisplaybreaks \\
        & \ \ \ \  \stackrel{(b)}{\leq}  \EEpe\bigg[ X |\epsilon \langle \boldsymbol \theta', \q_{\perp} \rangle|  e^{\epsilon | \langle \boldsymbol \theta', \q_{\perp} \rangle|} \bigg] \nonumber \allowdisplaybreaks \\
        & \ \ \ \  \stackrel{(c)}{\leq} \EEpe[|X|^2]^{\frac{1}{2}} \EEpe\big[ |\epsilon \langle \boldsymbol \theta', \q_{\perp} \rangle|^{4} \big]^{\frac{1}{4}} \EEpe\bigg[ e^{4\epsilon | \langle \boldsymbol \theta', \q_{\perp} \rangle|} \bigg]^{\frac{1}{4}} \nonumber \allowdisplaybreaks \\
        & \ \ \ \ \stackrel{(d)}{\leq} \epsilon \| \boldsymbol \theta' \| \EEpe[|X|^2]^{\frac{1}{2}}  \EEpe\big[ \| \q_{\perp} \|^{4} \big]^{\frac{1}{4}} \EEpe\bigg[ e^{4\epsilon \|\TT'\|\| \q_{\perp}\|} \bigg]^{\frac{1}{4}},
\end{align}
where (a) follows by using $Re(2\theta_2+\theta_3 ) \leq 0$ and $Re(\theta_2+2\theta_3 ) \leq 0$ as $\TT \in \boldsymbol \Theta $; (b) holds because $|e^x-1| \leq |x|e^{|x|}$ for any $x\in \mathbb{C}$; (c) and (d) holds by using Cauchy-Schwarz inequality. 

By replacing $\tilTT$ with $\TT\in \boldsymbol\Theta$ and $\TT'$ with $-3(\theta_2+\theta_3)\mathbf 1_3$ in Eq. \eqref{eq: 3q_mgf_equi_part1}, and using $\|3(\theta_2+\theta_3)\mathbf 1_3\| \leq 9\|\TT\|$ for any $\TT\in \boldsymbol \Theta$, we have
\begin{align}
\label{eq: 3q_mgf_equi_part2}
   &\left| \EEpe[ X e^{\epsilon \langle \boldsymbol \theta, \q \rangle}] - \EEpe[X e^{\epsilon ((2\theta_2+\theta_3) q_2 + (\theta_2+2\theta_3) q_3)}] \right| \nonumber \\
   & \ \ \ \ \leq 9\epsilon \| \boldsymbol \theta \| \EEpe[|X|^2]^{\frac{1}{2}}  \EEpe\big[ \| \q_{\perp} \|^{4} \big]^{\frac{1}{4}} \EEpe\bigg[ e^{36\epsilon \|\TT\|\| \q_{\perp}\|} \bigg]^{\frac{1}{4}}.
\end{align}
From Definition \ref{def: 3q_ssc}, for any $\epsilon<\epsilon_0$ and $\|\TT\|< \theta_0/36\epsilon$, we know
\begin{align}
\label{eq: 3q_mgf_equiv_qperp}
    \EEpe\big[ \| \q_{\perp} \|^{4} \big]^{\frac{1}{4}} \EEpe\bigg[ e^{36\epsilon \|\TT\|\| \q_{\perp}\|} \bigg]^{\frac{1}{4}} < \infty.
\end{align}
Thus, from Eq. \eqref{eq: 3q_mgf_equi_part2}, by substituting $X=1$, for any $\TT \in \boldsymbol \Theta^\epsilon$, we have 
\begin{align*}
    \left| \EEpe[ e^{\epsilon \langle \boldsymbol \theta, \q \rangle}] \right|& \leq \left| \EEpe[ e^{\epsilon ((2\theta_2+\theta_3) q_2 + (\theta_2+2\theta_3) q_3)}] \right| + 9\epsilon \| \boldsymbol \theta \|  \EEpe\big[ \| \q_{\perp} \|^{4} \big]^{\frac{1}{4}} \EEpe\bigg[ e^{36\epsilon \|\TT\|\| \q_{\perp}\|} \bigg]^{\frac{1}{4}} < \infty.
\end{align*}
This proves Lemma \ref{lem: 3q_mgf_equivalence_a}. 
For Lemma \ref{lem: 3q_mgf_equivalence_b}, suppose $\tilTT \in \mathbb C^3$ such that its projection $\TT$ onto the subspace $\mathcal{S}$ satisfies $\TT \in \capTT$. By plugging $X =1$ in Eq. \eqref{eq: 3q_mgf_equi_part1}, and using $\eplim \EEpe\big[ \| \q_{\perp} \|^{4} \big]^{\frac{1}{4}} \EEpe\bigg[ e^{4\epsilon \|\TT'\|\| \q_{\perp}\|} \bigg]^{\frac{1}{4}} <\infty$ (see Definition \ref{def: 3q_ssc}),  we have that for any $\TT\in \capTT$,
\begin{align*}
   \eplim \left| \EEpe[ e^{\epsilon \langle \tilde{\boldsymbol \theta}, \q \rangle}] - \EEpe[ e^{\epsilon ((2\theta_2+\theta_3) q_2 + (\theta_2+2\theta_3) q_3)}] \right| = 0. 
\end{align*}
As a consequence, we also get,
\begin{align*}
    \eplim \left| \EEpe[ e^{\epsilon \langle \tilde{\boldsymbol \theta}, \q \rangle}]\right| \leq \eplim \left| \EEpe[ e^{\epsilon ((2\theta_2+\theta_3) q_2 + (\theta_2+2\theta_3) q_3)}] \right| \leq 1,
\end{align*}
where last inequality follows by using $\TT\in \capTT$. Further, the above argument follows for any $\TT\in \capTT$ as $\capTT\subset \mathcal{S}$ and so, projection of $\TT$ is $\TT$ itself. This proves Lemma \ref{lem: 3q_mgf_equivalence_b}.

Recall that both $\q$ and $\q^+$  follow the steady state distribution and so, 
\begin{align*}
    \EEpe[q_1^+ + q_2^+] - \EEpe[ q_1+q_2] = 0,
\end{align*}
as the drift is zero in steady state. By plugging $\q^+ = \q+ \mathbf a-\mathbf s +\mathbf u$, 
\begin{align}
\label{eq: 3q_unused_epsilon}
     \EEpe[ u_1+u_2] =  \EEpe[ s_1+s_2] -\lambda_1 - \lambda_2  = 1-\lambda_1 - \lambda_2 =  \epsilon,
\end{align}
where the second equality follows because the chosen schedule can either be $(1,0,0)$ or $(0,1,1)$; third equality follows because $\boldsymbol\lambda = (1-\epsilon)\boldsymbol\nu$ where $\boldsymbol \nu \in \mathcal{F}$ as mentioned in Section \ref{sec: 3q_model}. Similarly, $\EEpe[ u_1+u_3] = \epsilon$. By plugging $X = u_2$ in Eq. \eqref{eq: 3q_mgf_equi_part2},
\begin{align}
\label{eq: 3q_mgf_equi_part3}
   &\eplim \frac{1}{\epsilon}\left| \EEpe[ u_2 e^{\epsilon \langle \boldsymbol \theta, \q \rangle}] - \EEpe[u_2 e^{\epsilon ((2\theta_2+\theta_3) q_2 + (\theta_2+2\theta_3) q_3)}] \right| \nonumber \\
   & \ \ \ \  \leq  \eplim 9 \| \boldsymbol \theta \| \EEpe[u_2^2]^{\frac{1}{2}}  \EEpe\big[ \| \q_{\perp} \|^{4} \big]^{\frac{1}{4}} \EEpe\bigg[ e^{36\epsilon \|\TT\|\| \q_{\perp}\|} \bigg]^{\frac{1}{4}} \nonumber \\
   & \ \ \ \ \stackrel{(a)}{\leq} \eplim 9 \sqrt{\epsilon}\| \boldsymbol \theta \| \EEpe\big[ \| \q_{\perp} \|^{4} \big]^{\frac{1}{4}} \EEpe\bigg[ e^{36\epsilon \|\TT\|\| \q_{\perp}\|} \bigg]^{\frac{1}{4}} \nonumber \\
   & \ \ \ \ =0,
\end{align}
where (a) follows using $\EEpe[u_2^{2}]=\EEpe[u_2]\leq \epsilon$. Also, $u_2 = 1$ only if $q_2 = 0$, and so,
\begin{equation*}
    u_2e^{\epsilon ((2\theta_2+\theta_3) q_2 + (\theta_2+2\theta_3) q_3)} = u_2e^{\epsilon  (\theta_2+2\theta_3) q_3}.
\end{equation*}
Combining this with Eq. \eqref{eq: 3q_mgf_equi_part3}, we have
\begin{align*}
    \eplim \frac{1}{\epsilon} \EEpe[ u_2 e^{\epsilon \langle \boldsymbol \theta, \q \rangle}] = \eplim \frac{1}{\epsilon} \EEpe[u_2 e^{\epsilon  (\theta_2+2\theta_3) q_3}],
\end{align*}
and as a consequence,
\begin{align*}
    \eplim \frac{1}{\epsilon} \left|\EEpe[ u_2 e^{\epsilon \langle \boldsymbol \theta, \q \rangle}]\right| = \eplim \frac{1}{\epsilon} \left|\EEpe[u_2 e^{\epsilon  (\theta_2+2\theta_3) q_3}]\right| \stackrel{(a)}{\leq} \eplim \frac{1}{\epsilon} \EEpe[u_2 ] \leq 1,
\end{align*}
where (a) follows by using $Re(\theta_2+2\theta_3)\leq 0$ as $\TT\in \capTT$. The rest of Lemma \ref{lem: 3q_mgf_equivalence_c} follows in the similar manner. \hfill $\blacksquare$
\endproof

\subsection{Proof of Theorem \ref{thm: 3q_functional_eq}}
\label{app: 3q_functional_eq}

\begin{lemma}[Second Order Approximation]
\label{app: 3q_2ndorderapprox}
    Consider the Three-queue system as defined in Section \ref{sec: 3q_model} operating under a policy that achieves state space collapse according the Definition \ref{def: 3q_ssc}. 
\begin{enumerate}[label=(\alph*), ref=\ref{app: 3q_2ndorderapprox}.\alph*]
        \item \label{app: 3q_2ndorderapprox_1} For any $\boldsymbol\theta \in \boldsymbol\Theta $, we have
        \begin{align*}
          \eplim \frac{1}{\epsilon^2}  \EEpe \left[e^{ \epsilon \langle \boldsymbol{\theta}, \mathbf{q}^+ \rangle } \Big( e^{- \epsilon \langle \boldsymbol{\theta}, \mathbf{u} \rangle}  -1 \Big)\right] = -\eplim \frac{1}{\epsilon}  \EEpe \left[ \langle \boldsymbol{\theta}, \mathbf{u} \rangle e^{ \epsilon \langle \boldsymbol{\theta}, \mathbf{q} \rangle } \right].
        \end{align*}
    \item \label{app: 3q_2ndorderapprox_2} For any $\boldsymbol\theta \in \boldsymbol\Theta $, we have
    \begin{align*}
        \eplim \frac{1}{\epsilon^2}  \EEpe \left[  e^{ \epsilon \langle \boldsymbol{\theta}, \mathbf{a }- \mathbf{s } \rangle } - 1\right] = - \frac{1}{2} \langle \boldsymbol \theta, \mathbf 1_3 \rangle + \frac{1}{2} \langle \boldsymbol \theta, \boldsymbol \sigma^2 \boldsymbol \theta \rangle.
    \end{align*}
    \end{enumerate}
\end{lemma}

\proof{\textit{Proof of Lemma \ref{app: 3q_2ndorderapprox}}.} For any $x \in \mathbb C$, we have $\left| e^{x} - x - 1 \right|\leq |x|^2 e^{|x|} $. Using this, we have 
\begin{align}
\label{eq: 3q_unused_first_order}
   \left| \EEpe \left[e^{ \epsilon \langle \boldsymbol{\theta}, \mathbf{q}^+ \rangle } \Big( e^{- \epsilon \langle \boldsymbol{\theta}, \mathbf{u} \rangle}  + \epsilon \langle \boldsymbol{\theta}, \mathbf{u} \rangle -1 \Big)\right] \right| &\leq \EEpe \left[\left|e^{ \epsilon \langle \boldsymbol{\theta}, \mathbf{q}^+ \rangle }\right| \epsilon^2  |\langle \boldsymbol{\theta}, \mathbf{u} \rangle|^2 e^{ \epsilon |\langle \boldsymbol{\theta}, \mathbf{u} \rangle|} \right]\nonumber\\
   & \stackrel{(a)}{\leq} \epsilon^2 \|\boldsymbol\theta \|^2  \EEpe \left[ \|\mathbf u\|^2 e^{\epsilon\|\boldsymbol\theta \|\|\mathbf u\|  } \left|e^{ \epsilon \langle \boldsymbol{\theta}, \mathbf{q}^+ \rangle }\right| \right]\nonumber\\
   & \stackrel{(b)}{\leq} \epsilon^2 \|\boldsymbol\theta \|^2 e^{3\epsilon\|\boldsymbol\theta \|  }  \EEpe \left[ \|\mathbf u\|^2  \left|e^{ \epsilon \langle \boldsymbol{\theta}, \mathbf{q}^+ \rangle }\right| \right]\nonumber\\
   &\stackrel{(c)}{\leq} \epsilon^2 \|\boldsymbol\theta \|^2 e^{3\epsilon\|\boldsymbol\theta \|  }  \sum_{i=1}^3 \EEpe \left[  u_i \big| e^{ \epsilon \langle \boldsymbol{\theta}, \mathbf{q}^+ \rangle} \big| \right],
\end{align}
where (a) follows by using the Cauchy-Schwarz inequality; (b) follows by using $\|u\|\leq 3$ as $u_i$'s are binary; (c) also follows as $u_i$'s are binary. Further,  using $\q^+ = \q + \mathbf a -\mathbf s + \mathbf s$, we have
\begin{align}
\label{eq: 3q_plus_remove}
     \Big| \EEpe \Big[  \epsilon \langle \boldsymbol{\theta}, \mathbf{u} \rangle \Big( e^{ \epsilon \langle \boldsymbol{\theta}, \mathbf{q}^+ \rangle} - e^{ \epsilon \langle \boldsymbol{\theta}, \mathbf{q} \rangle }\Big) \Big] \Big| 
    & \leq \EEpe\Big[ \epsilon \langle \boldsymbol{\theta}, \mathbf{u} \rangle \big|e^{ \epsilon \langle \boldsymbol{\theta}, \mathbf{q} \rangle}\big| \Big( e^{ \epsilon |\langle \boldsymbol{\theta}, \ar - \s +\un \rangle| } - 1 \Big) \Big] \allowdisplaybreaks\nonumber\\
    & \stackrel{(a)}{\leq } 3\epsilon^2 \| \boldsymbol{\theta}\|^2 a_{\max} e^{3\epsilon \| \boldsymbol{\theta}\| a_{\max} } \EEpe\Big[ \|\mathbf u\| \big|e^{ \epsilon \langle \boldsymbol{\theta}, \mathbf{q} \rangle}\big| \Big] \nonumber\\
    & \stackrel{(b)}{\leq } 3\epsilon^2 \| \boldsymbol{\theta}\|^2 a_{\max} e^{3\epsilon \| \boldsymbol{\theta}\| a_{\max} } \sum_{i=1}^3 \EEpe \left[  u_i \big| e^{ \epsilon \langle \boldsymbol{\theta}, \mathbf{q} \rangle} \big| \right],
\end{align}
where (a) follows as $a_i -s_i +u_i \leq a_{\max}$ for all $i\in \{1,2,3\}$ and $|e^{x}-1| \leq |x| e^{|x|}$ for all $x\in \mathbb C$; and (b) follows by as $u_i$'s are binary. 

From Lemma \ref{lem: 3q_mgf_equivalence_c}, we know that $\eplim \frac{1}{\epsilon} \EEpe \left[  u_i \big| e^{ \epsilon \langle \boldsymbol{\theta}, \mathbf{q} \rangle} \big| \right]$ exists and so $\eplim \EEpe \left[  u_i \big| e^{ \epsilon \langle \boldsymbol{\theta}, \mathbf{q} \rangle} \big| \right] =0$. By combining this with \eqref{eq: 3q_plus_remove}, we also have $\eplim \EEpe \left[  u_i \big| e^{ \epsilon \langle \boldsymbol{\theta}, \mathbf{q}^+ \rangle} \big| \right] =0$. Thus,
\begin{align*}
   \eplim \frac{1}{\epsilon^2} & \left| \EEpe \left[e^{ \epsilon \langle \boldsymbol{\theta}, \mathbf{q}^+ \rangle } \Big( e^{- \epsilon \langle \boldsymbol{\theta}, \mathbf{u} \rangle}   -1 \Big)\right] + \EEpe \left[\epsilon \langle \boldsymbol{\theta}, \mathbf{u} \rangle e^{ \epsilon \langle \boldsymbol{\theta}, \mathbf{q} \rangle }   \right] \right| \\
   &\leq \eplim \frac{1}{\epsilon^2} \left| \EEpe \left[e^{ \epsilon \langle \boldsymbol{\theta}, \mathbf{q}^+ \rangle } \Big( e^{- \epsilon \langle \boldsymbol{\theta}, \mathbf{u} \rangle}  + \epsilon \langle \boldsymbol{\theta}, \mathbf{u} \rangle -1 \Big)\right] \right| +  \eplim \frac{1}{\epsilon^2} \Big| \EEpe \Big[  \epsilon \langle \boldsymbol{\theta}, \mathbf{u} \rangle \Big( e^{ \epsilon \langle \boldsymbol{\theta}, \mathbf{q}^+ \rangle} - e^{ \epsilon \langle \boldsymbol{\theta}, \mathbf{q} \rangle }\Big) \Big] \Big| \\
   & \leq \|\boldsymbol\theta \|^2 e^{3\epsilon\|\boldsymbol\theta \|  }  \sum_{i=1}^3 \eplim \EEpe \left[  u_i \big| e^{ \epsilon \langle \boldsymbol{\theta}, \mathbf{q}^+ \rangle} \big| \right] +   3 \| \boldsymbol{\theta}\|^2 a_{\max} e^{3\epsilon \| \boldsymbol{\theta}\| a_{\max} } \sum_{i=1}^3 \eplim \EEpe \left[  u_i \big| e^{ \epsilon \langle \boldsymbol{\theta}, \mathbf{q} \rangle} \big| \right]\\
   & = 0.
\end{align*}
This completes the proof of Lemma \ref{app: 3q_2ndorderapprox_1}.

In order to prove Lemma \ref{app: 3q_2ndorderapprox_2}, we first calculate the mean and variance of $\langle \boldsymbol{\theta}, \mathbf{a }- \mathbf{s } \rangle$.
When the chosen schedule is $(1,0,0)$, we have $\langle \TT,\s\rangle = \theta_1 = \theta_2 + \theta_3 = \frac{1}{2} \langle \boldsymbol \theta, \mathbf 1_3 \rangle$, where the second and third equality follow using $\TT \in \boldsymbol \Theta$. Similarly, when the chosen schedule is $(0,1,1)$, we have $\langle \TT,\s\rangle = \theta_2 + \theta_3=\frac{1}{2} \langle \boldsymbol \theta, \mathbf 1_3 \rangle$.  Thus, from the above two cases, for all $\TT \in \boldsymbol \Theta$, we have 
\begin{align}
\label{eq: 3q_thetadots}
   \langle \TT,\s\rangle = \frac{1}{2} \langle \boldsymbol \theta, \mathbf 1_3 \rangle.
\end{align}
Similarly, 
\begin{align*}
    \EEpe[\langle \TT,\ar\rangle] = \langle \TT,\boldsymbol \lambda\rangle &=  (1-\epsilon) \langle \TT,\boldsymbol \nu\rangle\\
    &\stackrel{(a)}{=} (1-\epsilon) \big((\theta_2+\theta_3)\nu_1 + \theta_2\nu_2+\theta_3\nu_3\big)\\
    & \stackrel{(b)}{=} (1-\epsilon) (\theta_2 + \theta_3)\\
    & = \frac{1}{2}(1-\epsilon) \langle \boldsymbol \theta, \mathbf 1_3 \rangle
\end{align*}
where (a) follows as $\TT \in \mathcal S$ and so $\theta_1 = \theta_2+ \theta_3$; and (b) follows because $\nu_1 + \nu_2 = \nu_1 + \nu_3 = 1$ as $\boldsymbol \nu \in \mathcal F$. Thus,
\begin{align}
\label{eq: 3q_func_eq_mean}
\EEpe [ \langle \boldsymbol{\theta}, \mathbf{a }- \mathbf{s } \rangle ]  = - \frac{1}{2} \epsilon \langle \boldsymbol \theta, \mathbf{1}_3 \rangle.
\end{align}
Next,
\begin{align}
\label{eq: 3q_func_eq_variance}
    \EEpe[ \langle \boldsymbol \theta , \mathbf{a }- \mathbf{s } \rangle^2 ] &= \text{Var} \big(\langle \boldsymbol \theta , \mathbf{a }- \mathbf{s } \rangle\big) + \big(\EEpe[ \langle \boldsymbol \theta , \mathbf{a }- \mathbf{s } \rangle ] \big)^2 \nonumber \allowdisplaybreaks\\
    & \stackrel{(a)}{=} \text{Var}\big(\langle \boldsymbol \theta ,  \ar \rangle\big) + \frac{1}{4}\epsilon^2\langle \boldsymbol \theta, \mathbf 1_3 \rangle^2\nonumber \allowdisplaybreaks\\
    & = \langle \boldsymbol \theta,  \boldsymbol \sigma^2 \boldsymbol \theta \rangle+ \frac{1}{4}\epsilon^2\langle \boldsymbol \theta, \mathbf 1_3 \rangle^2, 
    % \nonumber\\
    % & = \boldsymbol \phi^T \mathbf{B}^T \boldsymbol \sigma^2 \mathbf{B} \boldsymbol \phi + \epsilon^2 \langle \boldsymbol \phi, \mathbf{1}_2 \rangle^2 \nonumber\allowdisplaybreaks\\
    % & \stackrel{(b)}{=} \langle \boldsymbol \phi, \boldsymbol \Gamma \boldsymbol \phi \rangle + \epsilon^2 \langle \boldsymbol \phi, \mathbf{1}_2 \rangle^2,
\end{align}
where (a) follows as $\langle \TT , \s\rangle$ is constant. Finally, by using $\left|e^{x} - \frac{1}{2}x^2 - x - 1\right| \leq |x|^3 e^{|x|}$,
\begin{align*}
    \left|\EEpe \left[  e^{ \epsilon \langle \boldsymbol{\theta}, \mathbf{a }- \mathbf{s } \rangle } -  \frac{1}{2} \epsilon^2\langle \boldsymbol \theta , \mathbf{a } - \mathbf{s } \rangle^2 - \epsilon\langle \boldsymbol \theta , \mathbf{a }- \mathbf{s } \rangle - 1\right] \right|
    &\leq \epsilon^3 \EEpe \left[ |\langle \boldsymbol{\theta}, \mathbf{a }- \mathbf{s } \rangle|^3 e^{\epsilon |\langle \boldsymbol{\theta}, \mathbf{a }- \mathbf{s } \rangle|} \right] \\
    &\stackrel{(a)}{\leq} \epsilon^3 \|\TT\|^3 \EEpe \left[ \|\mathbf{a }- \mathbf{s }\|^3 e^{\epsilon \|\TT\| \|\mathbf{a }- \mathbf{s }\| } \right]\\
    &\stackrel{(b)}{\leq} 9\epsilon^3 \|\TT\|^3 a_{\max}^3 e^{3\epsilon \|\TT\| a_{\max} },
\end{align*}
where (a) follows by Cauchy-Schwarz inequality; and (b) follows as arrivals are bounded by $a_{\max}$, $|a_i -s_i| \leq a_{\max} $ and so $\|\mathbf{a }- \mathbf{s } \| \leq 3a_{\max}$.
By combining this with Eq. \eqref{eq: 3q_func_eq_mean} and \eqref{eq: 3q_func_eq_variance}, we have 
\begin{align*}
    \eplim \frac{1}{\epsilon^2} &  \left|\EEpe \left[  e^{ \epsilon \langle \boldsymbol{\theta}, \mathbf{a }- \mathbf{s } \rangle } - 1\right] +  \frac{1}{2} \epsilon^2 \langle \boldsymbol \theta, \mathbf{1}_3 \rangle - \frac{1}{2} \epsilon^2 \langle \boldsymbol \theta,  \boldsymbol \sigma^2 \boldsymbol \theta \rangle  \right| \\
    & \leq \eplim \frac{1}{\epsilon^2} \left|\EEpe \left[  e^{ \epsilon \langle \boldsymbol{\theta}, \mathbf{a }- \mathbf{s } \rangle } -  \frac{1}{2} \epsilon^2\langle \boldsymbol \theta , \mathbf{a } - \mathbf{s } \rangle^2 - \epsilon\langle \boldsymbol \theta , \mathbf{a }- \mathbf{s } \rangle - 1\right] \right| + \eplim \frac{1}{2} \left| \EEpe[ \langle \boldsymbol \theta , \mathbf{a }- \mathbf{s } \rangle^2 ] - \langle \boldsymbol \theta,  \boldsymbol \sigma^2 \boldsymbol \theta \rangle \right| \\
    &\leq \eplim 9\epsilon \|\TT\|^3 a_{\max}^3 e^{3\epsilon \|\TT\| a_{\max} } + \eplim \frac{1}{8} \epsilon^2 |\langle \boldsymbol \theta, \mathbf{1}_3\rangle |^2\\
    &= 0.
\end{align*}
This proves Lemma \ref{app: 3q_2ndorderapprox_2}. \hfill $\blacksquare$
% and (b) follows by taking $\boldsymbol \Gamma = \mathbf{B}^T\boldsymbol \sigma^2 \mathbf{B}$. 
\endproof

Now, we are equipped to provide the proof of Theorem \ref{thm: 3q_functional_eq}.

\proof{\textit{Proof of Theorem \ref{thm: 3q_functional_eq}}.}
% Suppose $\q$ follows the steady state distribution $\pi_\epsilon$ and $\q^+$ is the state of the Markov chain that follows the state $\q$, then, as the Markov chain is positive recurrent, $\q^+$ also follows the steady state distribution $\pi_\epsilon$. 
For any $\TT\in \capTT^\epsilon$, using $\q^+ = \q + \mathbf a -\mathbf s + \mathbf s$, we have
\begin{align}
\label{eq: 3q_fun_eq_theo_lhs}
  \EEpe \Big[  e^{ \epsilon \langle \boldsymbol{\theta}, \mathbf{q}^+ \rangle } \Big( e^{- \epsilon \langle \boldsymbol{\theta}, \mathbf{u} \rangle}  -1 \Big) \Big] 
  & = \EEpe \Big[ e^{ \epsilon \langle \boldsymbol{\theta}, \mathbf{q} \rangle } e^{ \epsilon \langle \boldsymbol{\theta}, \mathbf{a }- \mathbf{s } \rangle } \Big] - \EEpe \Big[ e^{ \epsilon \langle \boldsymbol{\theta}, \mathbf{q}^+ \rangle}\Big] \allowdisplaybreaks \nonumber \\
  & \stackrel{(a)}{=} \EEpe \Big[ e^{ \epsilon \langle \boldsymbol{\theta}, \mathbf{q} \rangle } \Big] \EEpe \Big[ e^{ \epsilon \langle \boldsymbol{\theta}, \mathbf{a }- \mathbf{s } \rangle } \Big] - \EEpe \Big[ e^{ \epsilon \langle \boldsymbol{\theta}, \mathbf{q}^+ \rangle}\Big] \allowdisplaybreaks \nonumber \\
  & \stackrel{(b)}{=}  \EEpe \Big[ e^{ \epsilon \langle \boldsymbol{\theta}, \mathbf{q} \rangle } \Big] \Bigg( \EEpe \Big[ e^{ \epsilon \langle \boldsymbol{\theta}, \mathbf{a }- \mathbf{s } \rangle } \Big] - 1\Bigg),
\end{align}
where (a) follows as the arrivals are independent of the queue length vector and $\langle \TT,\s\rangle = \frac{1}{2}\langle \TT, \mathbf 1_3 \rangle $ is independent of $\q$ from Eq. \eqref{eq: 3q_thetadots}; and (b) holds as $\q$ and $\q^{+}$ both follow steady state distribution $\pi_\epsilon$ and $\left|\EEpe \Big[ e^{ \epsilon \langle \boldsymbol{\theta}, \mathbf{q} \rangle } \Big]\right| < \infty $ for any $\TT\in \capTT^\epsilon$. 
Also, $\eplim \capTT^\epsilon = \capTT$, and thus, by taking $\epsilon \rightarrow 0$ in Eq. \eqref{eq: 3q_fun_eq_theo_lhs}, for all $\TT\in \capTT$, we have
\begin{align}
\label{eq: 3q_onestepbefore_func_eq}
    \eplim \frac{1}{\epsilon^2} \EEpe \Big[  e^{ \epsilon \langle \boldsymbol{\theta}, \mathbf{q}^+ \rangle } \Big( e^{- \epsilon \langle \boldsymbol{\theta}, \mathbf{u} \rangle}  -1 \Big) \Big] = \eplim \frac{1}{\epsilon^2} \EEpe \Big[ e^{ \epsilon \langle \boldsymbol{\theta}, \mathbf{q} \rangle } \Big] \Bigg( \EEpe \Big[ e^{ \epsilon \langle \boldsymbol{\theta}, \mathbf{a }- \mathbf{s } \rangle } \Big] - 1\Bigg).
\end{align}
From Lemma \ref{app: 3q_2ndorderapprox_1}, we have
\begin{equation}
\label{eq: 3q_marginal_m}
    \lim_{\epsilon \rightarrow 0} \frac{1}{\epsilon^2} \EEpe \Big[e^{ \epsilon \langle \boldsymbol{\theta}, \mathbf{q}^+ \rangle } \Big( e^{- \epsilon \langle \boldsymbol{\theta}, \mathbf{u} \rangle}  -1 \Big) \Big] =  - \Bigg \langle  \boldsymbol{\theta}, \lim_{\epsilon \rightarrow 0} \frac{1}{\epsilon} \EEpe \Big[  \mathbf{u} e^{ \epsilon \langle \boldsymbol{\theta}, \mathbf{q} \rangle} \Big] \Bigg \rangle =  - \langle \boldsymbol{\theta}, \mathbf{M}(\boldsymbol{\theta}) \rangle.
\end{equation}
Plugging this in the Eq. \eqref{eq: 3q_onestepbefore_func_eq}, using Lemma \ref{app: 3q_2ndorderapprox_2}, and the fact that $\lim_{\epsilon \rightarrow 0} \EEpe \Big[ e^{ \epsilon \langle \boldsymbol{\theta}, \mathbf{q} \rangle } \Big]$ exists by Lemma \ref{lem: 3q_mgf_equivalence} presented in Appendix \ref{app: 3q_mgf_equivalence}, we get that for any $\boldsymbol \theta \in \boldsymbol \Theta$,
\begin{align*}
    - \langle \boldsymbol{\theta}, \mathbf{M}(\boldsymbol{\theta}) \rangle
   & = \Big( - \frac{1}{2} \langle \boldsymbol \theta, \mathbf 1_3 \rangle + \frac{1}{2} \langle \boldsymbol \theta, \boldsymbol \sigma^2 \boldsymbol \theta \rangle \Big) \lim_{\epsilon \rightarrow 0} \EEpe \Big[ e^{ \epsilon \langle \boldsymbol{\theta}, \mathbf{q} \rangle } \Big] \\
   & = \left( - \frac{1}{2} \langle \boldsymbol \theta, \mathbf 1_3 \rangle + \frac{1}{2} \langle \boldsymbol \theta, \boldsymbol \sigma^2 \boldsymbol \theta \rangle \right) L(\boldsymbol{\theta}),
\end{align*}
% where $L(\boldsymbol{\theta}) = \lim_{\epsilon \rightarrow 0} \EEpe \Big[ e^{ \epsilon \langle \boldsymbol{\theta}, \mathbf{q} \rangle } \Big]$. 
%  The last inequality follows because for any $\boldsymbol \theta \in \mathcal{S}$, $\langle \boldsymbol \theta , \q \rangle = \langle \boldsymbol \theta , \q \rangle$. 
% Combining this with Eq. \eqref{eq: 3q_marginal_m}, for any $\boldsymbol \phi \in \Phi$,
% \begin{equation*}
%     \left( -\langle \boldsymbol{\phi},\mathbf{1}_2 \rangle + \frac{1}{2} \langle \boldsymbol{\phi} , \boldsymbol \Gamma \boldsymbol{\phi} \rangle \right) L(\boldsymbol{\phi}) + \langle \boldsymbol{\theta}, \mathbf{M}(\boldsymbol{\phi}) \rangle = 0
% \end{equation*}
where, 
\begin{align*}
    L(\boldsymbol \theta) = \lim_{\epsilon \rightarrow 0} \EEpe[e^{\epsilon \langle \boldsymbol \theta, \q \rangle}], && M_2(\boldsymbol \theta)= \lim_{\epsilon \rightarrow 0} \frac{1}{\epsilon} \EEpe [u_2e^{\epsilon \langle \boldsymbol \theta, \q \rangle}], && M_3(\boldsymbol \theta)= \lim_{\epsilon \rightarrow 0} \frac{1}{\epsilon} \EEpe [u_3e^{\epsilon \langle \boldsymbol \theta, \q \rangle}].
\end{align*}
% and
% \begin{align*}
%     M_2(\boldsymbol \theta) = \lim_{\epsilon \rightarrow 0} \frac{1}{\epsilon} \EEpe [u_2e^{\epsilon (\theta_2 + 2\theta_3) q_3}], && M_3(\boldsymbol \theta) = \lim_{\epsilon \rightarrow 0} \frac{1}{\epsilon} \EEpe [u_3e^{\epsilon(2\theta_2 + \theta_3) \rangle q_2}].
% \end{align*}
Note that $M_1(\boldsymbol\theta)=0$ because $u_1=0$ by the definition of the service process.
% Now, by choosing $\boldsymbol \phi \in \mathbb{C}^2$ to be such that $\boldsymbol \theta = \mathbf{B} \boldsymbol \phi$ and $\mathbf{w}$ to be such that $\q = \mathbf{B} \mathbf{w}$, we get that
% Finally, as $\boldsymbol \theta = \mathbf B \boldsymbol \phi$, we get that $\theta_2 = \phi_1$ and $\theta_3 = \phi_2$. Thus, $ \langle \boldsymbol{\theta}, \mathbf{M}(\boldsymbol{\phi}) \rangle = \phi_1 M_2(\boldsymbol\phi)+\phi_2 M_3(\boldsymbol\phi)$. 
This gives us the functional equation in Eq. \eqref{eq: 3q_functional_eq}. \hfill $\blacksquare$
\endproof

\subsection{Proof of  Lemma \ref{lem: 3q_analytic} and Lemma \ref{lem: 3q_uniqueness}}
\label{app: 3q_uniqueness}

\proof{\textit{Proof of Lemma \ref{lem: 3q_analytic}}.}
It is easy to observe that the function $L(\TT)$, $M_2(\TT)$ and $M_3(\TT)$ are non-zero simply by substituting $\TT =\mathbf 0$ to get that $L(\mathbf 0) =1$, and similarly, $M_2(\mathbf 0) =1$ and $M_3(\mathbf 0)=1$ by using the fact that $\EEpe[u_2] = \EEpe[u_3] =\epsilon$ from Eq. \eqref{eq: 3q_unused_epsilon}.

In order to prove the remaining of Lemma \ref{lem: 3q_uniqueness}, we are going to use Lemma \ref{lem: functional_uniqueness}. Note that, by using Lemma \ref{lem: 3q_mgf_equivalence_b}, we have that the absolute value of $L(\TT)$, $M_2(\TT)$ and $M_3(\TT)$ is bounded. Next, by using Lemma \ref{lem: 3q_mgf_equivalence_b},
\begin{align*}
    L(\boldsymbol \theta) = \lim_{\epsilon \rightarrow 0} \EEpe[e^{\epsilon \langle \boldsymbol \theta, \q \rangle}] = \lim_{\epsilon \rightarrow 0} \EEpe\left[ e^{\epsilon( (2\theta_2 + \theta_3) q_2+  (\theta_2 + 2\theta_3) q_3)}\right],
\end{align*}
and by Lemma \ref{lem: 3q_mgf_equivalence_c},
\begin{align*}
    M_2(\boldsymbol \theta)= \lim_{\epsilon \rightarrow 0} \frac{1}{\epsilon} \EEpe [u_2e^{\epsilon \langle \boldsymbol \theta, \q \rangle}] = \lim_{\epsilon \rightarrow 0} \frac{1}{\epsilon} \EEpe [u_2e^{\epsilon (\theta_2 + 2\theta_3) q_3}],\\ M_3(\boldsymbol \theta)= \lim_{\epsilon \rightarrow 0} \frac{1}{\epsilon} \EEpe [u_3e^{\epsilon \langle \boldsymbol \theta, \q \rangle}] = \lim_{\epsilon \rightarrow 0} \frac{1}{\epsilon} \EEpe [u_3e^{\epsilon(2\theta_2 + \theta_3)  q_2}].
\end{align*}
By using the result in \cite{Hurtado-gen-switch-SIGMETRICS}, we know that the first order terms, $\EEpe[\epsilon q_2]$, $\EEpe[\epsilon q_3]$, $\EEpe[u_3 q_2]$ and $\EEpe[u_2 q_3]$ exists as $\epsilon \rightarrow 0$. As such, the functions $L(\TT)$, $M_2(\TT)$ and $M_3(\TT)$ are complex differentiable in the set $\capTT$. Thus, the functions $L(\TT)$, $M_2(\TT)$ and $M_3(\TT)$ are holomorphic in the interior of $\capTT$. The continuity of $L(\TT)$, $M_2(\TT)$ and $M_3(\TT)$ over $\capTT$, (including the boundary) again follows by the same argument.

\endproof

\proof{\textit{Proof of Lemma \ref{lem: 3q_uniqueness}}.}
In order to prove Lemma \ref{lem: 3q_uniqueness}, we are going to use Lemma \ref{lem: functional_uniqueness}. First, note that by using Lemma \ref{lem: 3q_mgf_equivalence_b},
\begin{align*}
    L(\boldsymbol \theta) = \lim_{\epsilon \rightarrow 0} \EEpe[e^{\epsilon \langle \boldsymbol \theta, \q \rangle}] = \lim_{\epsilon \rightarrow 0} \EEpe\left[ e^{\epsilon( (2\theta_2 + \theta_3) q_2+  (\theta_2 + 2\theta_3) q_3)}\right],
\end{align*}
and by Lemma \ref{lem: 3q_mgf_equivalence_c},
\begin{align*}
    M_2(\boldsymbol \theta)= \lim_{\epsilon \rightarrow 0} \frac{1}{\epsilon} \EEpe [u_2e^{\epsilon \langle \boldsymbol \theta, \q \rangle}] = \lim_{\epsilon \rightarrow 0} \frac{1}{\epsilon} \EEpe [u_2e^{\epsilon (\theta_2 + 2\theta_3) q_3}],\\ M_3(\boldsymbol \theta)= \lim_{\epsilon \rightarrow 0} \frac{1}{\epsilon} \EEpe [u_3e^{\epsilon \langle \boldsymbol \theta, \q \rangle}] = \lim_{\epsilon \rightarrow 0} \frac{1}{\epsilon} \EEpe [u_3e^{\epsilon(2\theta_2 + \theta_3)  q_2}].
\end{align*}
% We already know that the heavy traffic distribution exists.
% of the scaled queue length vector i.e., the distribution of $\lim_{\epsilon \rightarrow 0} \epsilon \q$ exists. 
We do a linear transform of the variable $\boldsymbol \theta$ so that the Laplace transform $M_2(\cdot)$ and $M_3(\cdot)$ depends only on one variable. We pick $\psi_1 = 2\theta_2+\theta_3$ and $\psi_2 = \theta_2+2\theta_3$, i.e., $\boldsymbol \psi = (\psi_1,\psi_2) = \mathbf B^T \boldsymbol \theta$ (as $\TT\in \capTT  $). Thus, $\theta_2 = \frac{1}{3} (2\psi_1 - \psi_2)$ and $\theta_3 =\frac{1}{3} ( 2\psi_2 - \psi_1)$. 
With slight abuse of notation, we replace $L(\boldsymbol \theta)$, $M_2(\boldsymbol \theta) $ and $M_3(\boldsymbol \theta)$ with $L(\boldsymbol \psi)$, $M_2(\psi_2)$ and  $M_3(\psi_1)$ respectively. 
By using this notation,
\begin{align*}
    L(\boldsymbol \psi) = \lim_{\epsilon \rightarrow 0} \EEpe[e^{\epsilon (\psi_1 q_2 + \psi_2 q_3)}], && M_2(\psi_2) = \lim_{\epsilon \rightarrow 0} \frac{1}{\epsilon} \EEpe [u_2e^{\epsilon \psi_2 q_3}], && M_3(\psi_1) = \lim_{\epsilon \rightarrow 0} \frac{1}{\epsilon} \EEpe [u_3e^{\epsilon \psi_1 q_3}].
\end{align*}
% The analysis presented in the proof of Theorem \ref{thm: 3q_functional_eq} can be easily extended so that the functional equation holds as long as $Re(\boldsymbol\psi) \leq 0$ because under this condition, the Laplace transforms $L(\boldsymbol \psi)$, $M_1(\psi_2)$ and $M_2(\psi_1)$ exists. 
Further, by using Lemma \ref{lem: 3q_analytic}, we have that $L(\boldsymbol \psi)$, $M_2(\psi_2)$ and  $M_3(\psi_1)$ are holomorphic, continuous and bounded, for all $\boldsymbol \psi \in \boldsymbol \Psi $, where
\[\boldsymbol\Psi =\mathbf B^T \capTT = \{ \boldsymbol \psi \in \mathbb C^2: Re(\boldsymbol \psi ) \leq \mathbf 0_2\}. \]
The functional equation can be rewritten as,
    \begin{equation}
    \label{eq: 3q_rewritten_func_eq}
       \Tilde{\mathcal{P}}(\boldsymbol\psi):= \left( - \frac{1}{3} \langle \boldsymbol{\psi},\mathbf{1}_2 \rangle + \frac{1}{2} \langle \boldsymbol{\psi} , \boldsymbol \Gamma \boldsymbol{\psi} \rangle \right) L(\boldsymbol{\psi}) + 
        \frac{1}{3} (2\psi_1 - \psi_2)M_2(\psi_2) +\frac{1}{3} (2\psi_2 - \psi_1)M_3(\psi_1)  = 0,
    \end{equation}
where $\boldsymbol \Gamma =\mathbf D^{-1}  \mathbf{B}^T\boldsymbol \sigma^2 \mathbf{B} \mathbf D^{-1} $. Note that, here we get $\boldsymbol  \sigma^2 = \mathbf{B}\boldsymbol \Gamma \mathbf{B}^T$.

As the functional equation in Eq. \eqref{eq: 3q_functional_eq}, that is $\mathcal{P}(\TT)=0$, holds for any $\boldsymbol \theta \in \boldsymbol \Theta$, so the rewritten functional equation in Eq. \eqref{eq: 3q_rewritten_func_eq}, that is $\Tilde{\mathcal{P}}(\boldsymbol\psi)=0$, holds for any $\boldsymbol \psi \in \boldsymbol \Psi$. 
Also, as one can easily check, the conditions in Eq. \eqref{eq: funceqconditions} in Lemma \ref{lem: functional_uniqueness} holds for the functional equation in \eqref{eq: 3q_rewritten_func_eq}.
% , and by using Lemma \ref{lem: 3q_mgf_equivalence} presented in Appendix \ref{app: 3q_mgf_equivalence},
% \begin{align*}
%     L(\boldsymbol \psi) = \lim_{\epsilon \rightarrow 0} \EEpe[e^{\epsilon (\psi_1 q_2 + \psi_2 q_3)}], 
%     && M_2(\psi_2) = \lim_{\epsilon \rightarrow 0} \frac{1}{\epsilon} \EEpe [u_2e^{\epsilon \psi_2 q_3}] && M_3(\psi_1) = \lim_{\epsilon \rightarrow 0} \frac{1}{\epsilon} \EEpe [u_3e^{\epsilon\psi_1 q_2}].
%  \end{align*}
% Now, note that $M_2(\psi_2)$ can be rewritten as 
% \begin{equation*}
%     M_2(\psi_2) = \lim_{\epsilon \rightarrow 0} \frac{1}{\epsilon} \mathbb{E} [u_2\mathbf{1}_{\{q_2=0\}} e^{\epsilon \psi_2 q_3}].
% \end{equation*}
% This holds because if $u_2=1$ then there is unused service, which implies that $q_2 =0$ and so, $u_2 = u_2\mathbf{1}_{\{q_2=0\}}$. Thus, $M_2(\psi_2)$ is a Laplace transform of a boundary measure that is restricted to the axes $q_2=0$. Similarly, $ M_3(\psi_1)$ is a Laplace transform of a boundary measure restricted to the axes $q_3=0$.
Thus, the functional equation in Eq. \eqref{eq: 3q_rewritten_func_eq} satisfies the conditions mentioned in Lemma \ref{lem: functional_uniqueness} and so, the functional equation in Eq. \eqref{eq: 3q_rewritten_func_eq} has a unique solution. This in turn implies that functional equation for Three-queue system has a unique solution. \hfill $\blacksquare$
\endproof

\subsection{Proof of Theorem \ref{thm: 3q_dist}}
\label{app: 3q_dist}

\proof{\textit{Proof of Theorem \ref{thm: 3q_dist}}.}
% Suppose $\tilTT \in \mathbb{C}^3 $ such that the projection of $\tilTT $ onto the subspace $\mathcal S$ is given by $\TT$ and $\TT \in \boldsymbol \Theta$. Then, from Lemma \ref{lem: 3q_mgf_equivalence_b},
% \begin{equation*}
%     \lim_{\epsilon \rightarrow 0}\EEpe[e^{\epsilon \langle \tilTT, \q \rangle}] = \lim_{\epsilon \rightarrow 0} \EEpe[e^{\epsilon \langle \boldsymbol \theta, \q \rangle}] = L(\boldsymbol \theta).
% \end{equation*}
% This implies that it is enough to characterize the Laplace transform of the queue length vector for $\boldsymbol \theta \in \mathcal{S}$. 
% Now, for the remaining part of the proof, we assume $\boldsymbol \theta \in \mathcal{S}$.
% Pick $\boldsymbol \phi $ such that $\boldsymbol \theta = \mathbf{B}\boldsymbol\phi \in \boldsymbol \Theta$. 
The Laplace transform of the considered distribution, for any $\TT \in \boldsymbol \Theta$, is given by 
\begin{align*}
    \mathbb{E}[e^{\theta_1 (\Upsilon_1+\Upsilon_2)+\theta_2 \Upsilon_1 + \theta_3 \Upsilon_2 }]  &= \mathbb{E}[e^{(2\theta_2 + \theta_3) \Upsilon_1+(\theta_2 + 2\theta_3)\Upsilon_2}] \\
    &= \frac{1}{\bigg( 1- (2\theta_2 + \theta_3) \frac{3\sigma_2^2 + \sigma_3^2}{8}\bigg)\bigg( 1- (\theta_2 + 2\theta_3)\frac{\sigma_2^2 + 3\sigma_3^2}{8}\bigg)}.
\end{align*}
Now pick
\begin{align}
\label{eq: 3q_laplacesolution}
    L(\TT) = \frac{1}{\bigg( 1- (2\theta_2+\theta_3) \frac{3\sigma_2^2 + \sigma_3^2}{8}\bigg)\bigg( 1- (\theta_2+2\theta_3)\frac{\sigma_2^2 + 3\sigma_3^2}{8}\bigg)},\nonumber\\ M_2 (\TT)  = \frac{1}{1- (\theta_2+2\theta_3)\frac{\sigma_2^2 + 3\sigma_3^2}{8}}, && M_3 (\TT) = \frac{1}{1- (2\theta_2+\theta_3)\frac{3\sigma_2^2 + \sigma_3^2}{8}}.
\end{align}
For this to satisfy the functional equation given in Eq. \eqref{eq: 3q_functional_eq}, we need
\begin{equation*}
     \left( -\frac{1}{2}\langle \boldsymbol \theta ,\mathbf{1} \rangle + \frac{1}{2} \langle \TT , \boldsymbol \sigma^2 \TT \rangle \right)  + \theta_2 \bigg( 1- (2\theta_2+\theta_3)\frac{3\sigma_2^2 + \sigma_3^2}{8}\bigg)+\theta_3 \bigg( 1- (\theta_2+2\theta_3)\frac{\sigma_2^2 + 3\sigma_3^2}{8}\bigg)  = 0,
\end{equation*}
which can be simplified to the condition,
\begin{align}
\label{eq: 3q_provefunctional}
     4 \langle \TT , \boldsymbol \sigma^2 \TT \rangle  &= \theta_2  (2\theta_2+\theta_3)(3\sigma_2^2 + \sigma_3^2)+\theta_3  (\theta_2+2\theta_3)(\sigma_2^2 + 3\sigma_3^2)\nonumber\\
     &=2 \theta_2^2 ( 3\sigma_2^2+\sigma_3^2) +2\theta_3^2( \sigma_2^2+3\sigma_3^2) + 4\theta_2\theta_3 (\sigma_2^2+\sigma_3^2).
\end{align}
Further, by using the condition $\sigma_1^2 = \frac{\sigma_2^2 +\sigma_3^2}{2}$ and $\theta_1 = \theta_2+\theta_3$, we have 
\begin{align}
\label{eq: 3q_provefunctionalterm1}
    \langle \TT , \boldsymbol \sigma^2 \TT \rangle  &= \begin{bmatrix}
    \theta_2 + \theta_3& \theta_2 & \theta_3
    \end{bmatrix} 
    \begin{bmatrix}
    \frac{\sigma_2^2 +\sigma_3^2}{2} & 0 & 0\\
    0& \sigma_2^2 & 0\\
    0 & 0 & \sigma_3^2
    \end{bmatrix} 
        \begin{bmatrix}
    \theta_2+ \theta_3 \\ \theta_2 \\ \theta_3 
    \end{bmatrix} \nonumber \allowdisplaybreaks\\
    &=\theta_2^2 \Big( \frac{3\sigma_2^2+\sigma_3^2}{2} \Big) +\theta_3^2 \Big( \frac{\sigma_2^2+3\sigma_3^2}{2} \Big) + \theta_2\theta_3 (\sigma_2^2+\sigma_3^2).
\end{align}
% Next,
% \begin{align}
% \label{eq: 3q_provefunctionalterm2}
%     \theta_2  (2\theta_2+\theta_3)(3\sigma_2^2 + \sigma_3^2)+\theta_3  (\theta_2+2\theta_3)(\sigma_2^2 + 3\sigma_3^2) &= \theta_2  (2\theta_2 + \theta_3)(3\sigma_2^2 + \sigma_3^2)+\theta_3  (\theta_2 + 2\theta_3)(\sigma_2^2 + 3\sigma_3^2)\nonumber\\
%     & = 2 \theta_2^2 ( 3\sigma_2^2+\sigma_3^2) +2\theta_3^2( \sigma_2^2+3\sigma_3^2) + 4\theta_2\theta_3 (\sigma_2^2+\sigma_3^2).
% \end{align}
From Eq. \eqref{eq: 3q_provefunctionalterm1}, we can easily observe that Eq. \eqref{eq: 3q_provefunctional} is satisfied. Thus, $L(\TT)$, $M_2(\TT)$ and $M_3(\TT)$ given in Eq. \eqref{eq: 3q_laplacesolution} solves the functional equation given in Eq. \eqref{eq: 3q_functional_eq}.  Further, it is easy observe that the choice of $L(\TT)$, $M_2(\TT)$ and $M_3(\TT)$ given in Eq. \eqref{eq: 3q_laplacesolution} lie in the set $\mathcal{L}(\TT)$, as they are rational functions with poles outside $\TT$. Then, from Lemma \ref{lem: 3q_uniqueness}, we get that the solution given by Eq. \eqref{eq: 3q_laplacesolution} is unique and so $L(\TT)$ in Eq. \eqref{eq: 3q_laplacesolution} gives the Laplace transform of the limiting distribution for the Three-queue system. Further, $L(\TT)$ matches with the Laplace transform of $\mathbf B \boldsymbol \Upsilon$ for any $\TT \in \boldsymbol \Theta$, that is 
\begin{align*}
    L(\TT) = \EE \big[e^{\langle \TT,\mathbf B \boldsymbol \Upsilon\rangle}\big], \ \forall \TT\in\capTT.
\end{align*}
Now, the result follows from Proposition \ref{prop: 3q_mgf_convergence}. This completes the proof. \hfill $\blacksquare$
% By simple algebraic manipulation, it can be verified that under the condition $2\sigma_3^2 = \sigma_1^2 + \sigma_2^2 $, 
% \begin{align*}
%   L(\TT) &= \lim_{\epsilon \rightarrow 0} \mathbb{E}[e^{\epsilon ( (2\theta_2+\theta_3) q_{\|1} +(\theta_2+2\theta_3) q_{\|2} )}] = \frac{1}{\bigg( 1- (2\theta_2+\theta_3) \frac{3\sigma_1^2 + \sigma_2^2}{8}\bigg)\bigg( 1- (\theta_2+2\theta_3)\frac{\sigma_1^2 + 3\sigma_2^2}{8}\bigg)},\\
%   M_1 (\TT) & =\lim_{\epsilon \rightarrow 0} \frac{1}{\epsilon} \mathbb{E} [u_1e^{\epsilon (\theta_2+2\theta_3) q_2}] = \frac{1}{\bigg( 1- (\theta_2+2\theta_3)\frac{\sigma_1^2 + 3\sigma_2^2}{8}\bigg)},\\
%   M_2 (\TT) & =\lim_{\epsilon \rightarrow 0} \frac{1}{\epsilon} \mathbb{E} [u_2e^{\epsilon (2\theta_2+\theta_3) q_1}] = \frac{1}{\bigg( 1- (2\theta_2+\theta_3)\frac{3\sigma_1^2 + \sigma_2^2}{8}\bigg)},
% \end{align*}
% satisfy the Laplace equation given in Eq. \eqref{eq: 3q_functional_eq}. And from Lemma \ref{lem: 3q_uniqueness}, we know that there is a unique solution to the Eq. \eqref{eq: 3q_functional_eq}. It is easy to observe that the term in the above equation is the Laplace transform of sum of two independent exponential. This gives us the distribution of $\lim_{\epsilon\rightarrow 0} \epsilon \q_{\|}$. And finally, from state space collapse, we know that $\lim_{\epsilon\rightarrow 0} \epsilon \q_{\|} = \lim_{\epsilon\rightarrow 0} \epsilon \q$. This completes the proof.
\endproof

\section{Proof of Results for Input-queued switch}
\label{app: switch}

\subsection{Properties of the projection}
\label{sec: switch_projection}
\begin{lemma}
    \label{lem: switch_projection}
     Let $\mathbf B \in \{0,1\}^{n^2\times2n}$ is such that for any $1\leq i,j\leq n$,  
    \begin{equation*}
        B_{i+n(j-1),i} =B_{i+n(j-1),n+j} = 1, 
    \end{equation*}
    and all other elements are zero. 
    Let $\mathbf x \in \mathbb{C}^{n^2}$ and suppose $\mathbf x_{\|}$ denotes the projection of $\mathbf{x}$ onto the space $\mathcal{S}$ where 
\begin{equation*}
    \mathcal{S} = \Big\{ \mathbf{x} \in \mathbb{R}^{n^2} : \exists \mathbf{w} \in \mathbb{R}^{2n} \ s.t. \ \mathbf x = \mathbf 
    B \mathbf w\Big\}. 
\end{equation*}
    And $\mathbf x_{\perp} = \mathbf x - \mathbf x_{\|}$.
    Further, define $\mathbf{D} = \mathbf{B}^T\mathbf{B}$. Then,
    \begin{enumerate}
    \item The matrix $\mathbf B$ satisfies 
    \begin{align*}
    \mathbf{B}\begin{bmatrix} \mathbf{1}_n \\\mathbf{0}_n \end{bmatrix} = \mathbf{B}\begin{bmatrix} \mathbf{0}_n \\ \mathbf{1}_n \end{bmatrix} = \mathbf{1}_{n^2}, && \mathbf{B}^T \mathbf{1}_{n^2} = n\mathbf{1}_{2n}.
    \end{align*}
    This also gives us that
    \begin{align*}
    \sum_{i=1}^n \mathbf d_i =\mathbf{B}^T\mathbf{B}\begin{bmatrix} \mathbf{1}_n \\\mathbf{0}_n \end{bmatrix}=  n\mathbf 1_{2n}, &&\sum_{j=1}^n \mathbf d_{n+j} =\mathbf{B}^T\mathbf{B}\begin{bmatrix} \mathbf{0}_n \\\mathbf{1}_n \end{bmatrix}=  n\mathbf 1_{2n}, &&
        \sum_{i=1}^{2n}\mathbf d_i =   2n \mathbf 1_{2n},
    \end{align*}
    where $\{\mathbf d_1, \dots , \mathbf d_{2n}\}$ are columns of $ \mathbf{D}$.
% \item The closed form expression for $\mathbf x_{\|}$ is given by $\mathbf x_{\|} = \mathbf A \mathbf{x}$.
\item For any $\boldsymbol \phi \in \mathbb{C}^{2n}$, the vector $\boldsymbol \theta = \mathbf B \boldsymbol \phi$ lies in the space $\mathcal{S}$. Also, for any $\mathbf x \in \mathbb{C}^{n^2}$, and suppose $\mathbf w \in \mathbb{C}^{2n}$ is such that $\mathbf x_{\|} =\mathbf B\mathbf{w}$. Then,
     \begin{equation*}
        \langle \boldsymbol \theta, \mathbf{x} \rangle = \langle \boldsymbol \theta, \mathbf{x}_{\|} \rangle = \boldsymbol \phi^T\mathbf{B}^T\mathbf{B} \mathbf w = \sum_{i=1}^{2n} \langle \boldsymbol \phi , \mathbf d_i \rangle w_i.
    \end{equation*}
    \end{enumerate}
\end{lemma}

Lemma \ref{lem: switch_projection} provides the properties of the projection of any vector onto the space $\mathcal{S}$. 
% By the definition of the matrix $\mathbf{A}$ we have the relation that $\mathbf{B}^T \mathbf{A} = \mathbf{B}^T$ and $\mathbf{A}\mathbf{B} = \mathbf{B}$. 
% Also, note that the vector $\mathbf w$ in Part 3 of Lemma \ref{lem: switch_projection} need not be unique. A possible candidate is $\mathbf w = (\mathbf{B}^T\mathbf{B})^{-1}\mathbf{B}^T \mathbf x$ using Part 2 of Lemma \ref{lem: switch_projection}. Although for any $\mathbf w$ that satisfies $\mathbf{x} = \mathbf{B}\mathbf{w}$, we can pick $\mathbf{w}' = \mathbf{w} - w\begin{bmatrix} \mathbf{1}_n \\-\mathbf{1}_n \end{bmatrix}$ for any $w\in \mathbb R$ and as $\mathbf{B}\begin{bmatrix} \mathbf{1}_n \\-\mathbf{1}_n \end{bmatrix} = \mathbf{0}_{n^2}$ using Part 1 of Lemma \ref{lem: switch_projection}. By this argument, we can see that the dimension of the subspace $\mathcal{S}$ is $(2n-1)$ as you can fix one of the elements of the vector $\mathbf{w}$ to be zero. Next, we present the definition of state space collapse for switch system.

 \proof{\textit{Proof of Lemma \ref{lem: switch_projection}}}
  Part 1 follows directly from the structure of the matrix $\mathbf B$. For Part 2, note that $\mathbf x_{\perp}$ is perpendicular to the subspace $\mathcal S$. Then, as $\boldsymbol \theta \in \mathcal{S}$, $\langle \boldsymbol \theta, \mathbf{x}_{\perp} \rangle =0$ and so,
 \begin{align*}
      \langle \boldsymbol \theta, \mathbf{x} \rangle &= \langle \boldsymbol \theta, \mathbf{x} - \mathbf{x}_\perp \rangle =\langle \boldsymbol \theta,  \mathbf{x}_\| \rangle = \boldsymbol \phi^T\mathbf{B}^T\mathbf{B} \mathbf w = \sum_{i=1}^{2n} \langle \boldsymbol \phi , \mathbf d_i \rangle w_i,
 \end{align*}
 where third equality follows as $\boldsymbol \theta = \mathbf B \boldsymbol \phi$ and $\mathbf x_{\|} =\mathbf B\mathbf{w}$; and the last equality follows by using the definition of $\mathbf D$. \hfill $\blacksquare$
\endproof

\subsection{Required Lemma}
\label{app: switch_mgf_equivalence}

Before presenting the results for the Input-queued switch, we present a required lemma as given below. Recall that $\TT\in \capTT$, where 
\[\capTT = \{\TT\in \mathbb C^{n^2}: \TT\in \mathcal S, \ Re(\mathbf B^T\TT) \leq \mathbf 0_{2n} \}.\]
As $\TT\in \mathcal{S}$ for any $\TT\in \capTT$, we have that there exists $\boldsymbol\phi \in \mathbb C^{2n}$ such that $\TT = \mathbf B \boldsymbol \phi$. Suppose
\[\boldsymbol\Phi = \{\boldsymbol \phi\in \mathbb C^{2n}: Re(\langle \mathbf d_i, \boldsymbol \phi \rangle \leq 0, \ \forall 1\leq i\leq 2n)\}, \]
where $\mathbf d_i$'s are the columns of the matrix $\mathbf D = \mathbf B^T \mathbf B$.
Then, for any $\boldsymbol \phi \in \boldsymbol \Phi$, we have that $\boldsymbol\theta = \mathbf B \boldsymbol \phi \in \boldsymbol \Theta$. Conversely, for any $\TT\in\capTT$, there exists $\boldsymbol \phi \in \boldsymbol \Phi$ such that $\TT = \mathbf B \boldsymbol \phi $. 
So, for the ease of notations, we use $\TT$ and $\boldsymbol \phi$ interchangeably, where $\boldsymbol \phi \in \boldsymbol \Phi$ and $\TT =\mathbf B \boldsymbol \phi \in \capTT $.

\begin{lemma}
\label{lem: switch_mgf_equivalence}
Consider the Input-queued switch system as defined in Section \ref{sec: switch_model} operating under scheduling policy that achieves state space collapse according to the Definition \ref{def: switch_ssc}.
\begin{enumerate}[label=(\alph*), ref=\ref{lem: switch_mgf_equivalence}.\alph*]
    \item \label{lem: switch_mgf_equivalence_a} For any $\epsilon< \epsilon_0$ and $\TT\in \boldsymbol \Theta^\epsilon = \boldsymbol \Theta \cap \{\TT \in \mathbb C^3: \|\TT\| \leq \theta_0/4\epsilon\}$, we have 
    \begin{align*}
        \big| \EEpe  \big[ e^{\epsilon \langle \TT , \q \rangle} \big] \big| < \infty.
    \end{align*}
    \item \label{lem: switch_mgf_equivalence_b} Suppose $\tilde{\boldsymbol \theta} \in \mathbb{C}^3$, let $\boldsymbol \theta$ be its projection onto the space $\mathcal{S}$ such that $\boldsymbol \theta = \mathbf B \boldsymbol \phi \in \boldsymbol \Theta $, where $\boldsymbol \phi \in \boldsymbol \Phi$. Also, let $\mathbf r \in \mathbb R^{2n}_{+}$ is such that $\q_{\| \mathcal{ K}} = \mathbf B \mathbf r$. Then, we have
    \begin{align*}
        \lim_{\epsilon\rightarrow 0}  \big| \EEpe  \big[ e^{\epsilon \langle \tilTT , \q \rangle} \big] \big| \leq 1, && \lim_{\epsilon \rightarrow 0}\EEpe[e^{\epsilon \langle \tilde{\boldsymbol \theta}, \q \rangle}] = \lim_{\epsilon \rightarrow 0} \EEpe[e^{\epsilon \langle \boldsymbol \theta, \q \rangle}] = \lim_{\epsilon \rightarrow 0} \EEpe[e^{\epsilon \sum_{l=1}^{2n} \langle \boldsymbol \phi,\mathbf d_l \rangle r_l}].
    \end{align*}
    \item \label{lem: switch_mgf_equivalence_c} Suppose $k =i+n(j-1)$ for  $i,j\in \{1,2,\dots,n\}$. For any $\boldsymbol \theta \in \boldsymbol \Theta$, such that $\TT = \mathbf B \boldsymbol \phi$, we have 
    \begin{align*}
        \lim_{\epsilon \rightarrow 0} \frac{1}{\epsilon} \big|\EEpe[u_k e^{\epsilon \langle \TT , \q \rangle}] \big|\leq 1, && \lim_{\epsilon \rightarrow 0} \frac{1}{\epsilon}\EEpe [u_k e^{\epsilon \langle \boldsymbol \theta, \q \rangle}] 
    =  \lim_{\epsilon \rightarrow 0}\frac{1}{\epsilon} \mathbb{E}[u_k e^{\epsilon \sum_{l=1, l\neq i,l\neq n+j}^{2n} \langle \boldsymbol \phi,\mathbf d_l \rangle r_l}].
    \end{align*}
\end{enumerate}
\end{lemma}

Lemma \ref{lem: 3q_mgf_equivalence_a} provides the existence of the Laplace transform for any $\TT\in\capTT^\epsilon$, which is a necessary to perform the Lyapunov drift analysis. Further, the existence of limiting quantities in Lemma \ref{lem: 3q_mgf_equivalence_b} and Lemma \ref{lem: 3q_mgf_equivalence_c} is essential to establish the  functional equation.

\proof{\textit{Proof of Lemma \ref{lem: switch_mgf_equivalence}}.}
 As $\mathcal{S}$ is a linear subspace, suppose $\tilde{\boldsymbol \theta} = \boldsymbol \theta +\boldsymbol \theta_{\perp}$. Also, suppose $\mathbf r \in \mathbb R^{2n}_{+}$ is such that $\q_{\| \mathcal{ K}} = \mathbf B \mathbf r$. Then, 
\begin{align}
\label{eq: switch_theta_relation}
  \langle \tilde{\boldsymbol \theta}, \q \rangle &= \langle \tilTT , \q_{\|\mathcal K} \rangle  + \langle \tilTT , \q_{\perp \mathcal K} \rangle \nonumber\\
  &= \langle \TT , \q_{\|\mathcal K} \rangle  + \langle \TT_{\perp} , \q_{\|\mathcal K} \rangle+ \langle \tilTT , \q_{\perp \mathcal K} \rangle \nonumber\\
  &\stackrel{(a)}{=} \langle \TT , \q_{\|\mathcal K} \rangle  + \langle \tilTT , \q_{\perp \mathcal K} \rangle \nonumber\\
  & \stackrel{(b)}{=} \sum_{i=1}^{2n} \langle \boldsymbol \phi,\mathbf d_i \rangle r_i + \langle \tilTT , \q_{\perp \mathcal K} \rangle,
\end{align}
where (a) follows because $\q_{\|\mathcal K} \in \mathcal{ K} \subset \mathcal S$ and $\TT_{\perp}$ is orthogonal to the subspace $\mathcal{S}$ and so, $\langle \TT_{\perp} , \q_{\|\mathcal K} \rangle =0$, and (b) follows by using $\q_{\| \mathcal{ K}} = \mathbf B \mathbf r$ and $\boldsymbol \theta = \mathbf B \boldsymbol \phi$.

Now, suppose $X$ is a random variable, which is measurable with respect to the probability measure $\pi_\epsilon$. Then, by using similar argument as in Eq. \eqref{eq: 3q_mgf_equi_part1}, 
\begin{align}
    \label{eq: switch_mgf_equi_part1}
        & \left| \EEpe[ X e^{\epsilon \langle \tilde{\boldsymbol \theta}, \q \rangle}] - \EEpe[X e^{\epsilon \sum_{l=1}^{2n} \langle \boldsymbol \phi,\mathbf d_l \rangle r_l}] \right|  \leq \epsilon \| \tilTT \| \EEpe[|X|^2]^{\frac{1}{2}}  \EEpe\big[ \| \q_{\perp} \|^{4} \big]^{\frac{1}{4}} \EEpe\bigg[ e^{4\epsilon \|\tilTT\|\| \q_{\perp}\|} \bigg]^{\frac{1}{4}}.
\end{align}
By replacing $\tilTT$ with $\TT\in \boldsymbol\Theta$ in the above equation, we have
\begin{align}
    \label{eq: switch_mgf_equi_part2}
        & \left| \EEpe[ X e^{\epsilon \langle \TT, \q \rangle}] - \EEpe[X e^{\epsilon \sum_{l=1}^{2n} \langle \boldsymbol \phi,\mathbf d_l \rangle r_l}] \right|  \leq \epsilon \| \TT \| \EEpe[|X|^2]^{\frac{1}{2}}  \EEpe\big[ \| \q_{\perp} \|^{4} \big]^{\frac{1}{4}} \EEpe\bigg[ e^{4\epsilon \|\TT\|\| \q_{\perp}\|} \bigg]^{\frac{1}{4}}.
\end{align}
From Definition \ref{def: switch_ssc}, for any $\epsilon<\epsilon_0$ and $\|\TT\|< \theta_0/4\epsilon$, we know
\begin{align}
\label{eq: switch_mgf_equiv_qperp}
    \EEpe\big[ \| \q_{\perp} \|^{4} \big]^{\frac{1}{4}} \EEpe\bigg[ e^{4\epsilon \|\TT\|\| \q_{\perp}\|} \bigg]^{\frac{1}{4}} < \infty.
\end{align}
Thus, from Eq. \eqref{eq: switch_mgf_equi_part2}, by substituting $X=1$, for any $\TT \in \boldsymbol \Theta^\epsilon$, we have 
\begin{align*}
    \left| \EEpe[ e^{\epsilon \langle \boldsymbol \theta, \q \rangle}] \right|& \leq \left| \EEpe[ e^{\epsilon \sum_{l=1}^{2n} \langle \boldsymbol \phi,\mathbf d_l \rangle r_l}] \right| + \epsilon \| \boldsymbol \theta \|  \EEpe\big[ \| \q_{\perp} \|^{4} \big]^{\frac{1}{4}} \EEpe\bigg[ e^{4\epsilon \|\TT\|\| \q_{\perp}\|} \bigg]^{\frac{1}{4}} < \infty.
\end{align*}
This proves Lemma \ref{lem: switch_mgf_equivalence_a}.

For Lemma \ref{lem: switch_mgf_equivalence_b}, suppose $\tilTT \in \mathbb C^{n^2}$ such that its projection $\TT$ onto the subspace $\mathcal{S}$ satisfies $\TT \in \capTT$. By plugging $X =1$ in Eq. \eqref{eq: switch_mgf_equi_part1}, and using $\eplim \EEpe\big[ \| \q_{\perp} \|^{4} \big]^{\frac{1}{4}} \EEpe\bigg[ e^{4\epsilon \|\tilTT\|\| \q_{\perp}\|} \bigg]^{\frac{1}{4}} <\infty$ (see Definition \ref{def: switch_ssc}),  we have that for any $\TT\in \capTT$,
\begin{align*}
   \eplim \left| \EEpe[ e^{\epsilon \langle \tilde{\boldsymbol \theta}, \q \rangle}] - \EEpe[ e^{\epsilon \sum_{l=1}^{2n} \langle \boldsymbol \phi,\mathbf d_l \rangle r_l}] \right| = 0. 
\end{align*}
As a consequence, we also get,
\begin{align*}
    \eplim \left| \EEpe[ e^{\epsilon \langle \tilde{\boldsymbol \theta}, \q \rangle}]\right| \leq \eplim \left| \EEpe[ e^{\epsilon \sum_{l=1}^{2n} \langle \boldsymbol \phi,\mathbf d_l \rangle r_l}] \right| \leq 1,
\end{align*}
where last inequality follows by using $\boldsymbol \phi \in \boldsymbol \Phi$ and so $Re(\langle \boldsymbol \phi,\mathbf d_l \rangle) \leq 0$ for all $l \in \{1,2,\dots,n^2\}$. Further, the above argument follows for any $\TT\in \capTT$ as $\capTT\subset \mathcal{S}$ and so, projection of $\TT$ is $\TT$ itself. This proves Lemma \ref{lem: switch_mgf_equivalence_b}.

% Using similar argument as in Eq. \eqref{eq: 3q_unused_epsilon}, we have $\EEpe[u_k] \leq \epsilon$ for all $k \in \{1,2,\dots,n^2\}$.
% Thus, by substituting $X= u_k$ in Eq. \eqref{eq: switch_mgf_equi_part2}, we have that for any $\TT\in \capTT$,
% \begin{align*}
% \label{eq: switch_mgf_equi_part3}
%    \eplim \frac{1}{\epsilon} \left| \EEpe[ u_k e^{\epsilon \langle \TT, \q \rangle}] - \EEpe[u_k e^{\epsilon \sum_{l=1}^{2n} \langle \boldsymbol \phi,\mathbf d_l \rangle r_l}] \right| \leq \eplim \sqrt{\epsilon} \| \TT \|  \EEpe\big[ \| \q_{\perp} \|^{4} \big]^{\frac{1}{4}} \EEpe\bigg[ e^{4\epsilon \|\TT\|\| \q_{\perp}\|} \bigg]^{\frac{1}{4}} = 0
% \end{align*}
Suppose $k = i+n(j-1)$ where $i,j\in\{1,2,\dots, n\}$. Also, if $u_k = 1$ then $q_k = 0$, which implies that $q_{\perp \mathcal K, i+n(j-1)}+ r_i +r_{n+j} = 0$. Now, as $r_i \geq 0$ and $r_{n+j} \geq 0$, we get
\begin{align*}
        |r_i| \leq |\q_{\perp \mathcal K, i+n(j-1)}|\leq \|\q_{\perp \mathcal K}\|, && |r_{n+j}| \leq |\q_{\perp \mathcal K, i+n(j-1)}|\leq \|\q_{\perp \mathcal K}\|.
\end{align*}
This gives us that,
\begin{align*}
    |\langle \TT, \q_{\perp \mathcal K} \rangle + \langle \boldsymbol \phi,\mathbf d_i \rangle r_i +\langle \boldsymbol \phi,\mathbf d_{n+j} \rangle r_{n+j}| \leq \big( \| \TT\|  + |\langle \boldsymbol \phi,\mathbf d_i \rangle| + |\langle \boldsymbol \phi,\mathbf d_{n+j} \rangle| \big) \| \q_{\perp \mathcal K} \| = \theta'\| \q_{\perp \mathcal K} \|,
    \end{align*}
where $\theta' = \| \TT\|  + |\langle \boldsymbol \phi,\mathbf d_i \rangle| + |\langle \boldsymbol \phi,\mathbf d_{n+j} \rangle|$. 
Using similar argument as in Eq. \eqref{eq: 3q_mgf_equi_part1}, we have  
\begin{align*}
    \frac{1}{\epsilon}\Big| \EEpe[u_k e^{\epsilon \langle \TT, \q \rangle}] & - \EEpe[u_k e^{\epsilon \sum_{l=1,  l\neq i,l\neq n+j}^{2n} \langle \boldsymbol \phi,\mathbf d_l \rangle r_l}] \Big| \\
    &=\frac{1}{\epsilon}\EEpe\Big[u_k \left|e^{\epsilon \sum_{l=1,l\neq i,l\neq n+j}^{2n} \langle \boldsymbol \phi,\mathbf d_l \rangle r_l } \right| \big|  \big( 1- e^{\epsilon (\langle \|\TT\|, \q_{\perp \mathcal K} \rangle + \langle \boldsymbol \phi,\mathbf d_i \rangle r_i +\langle \boldsymbol \phi,\mathbf d_{n+j} \rangle r_{n+j})}   \big) \big|\Big]  \nonumber \allowdisplaybreaks\\
    &\stackrel{(a)}{\leq} \frac{1}{\epsilon}\EEpe\Big[u_k \big|   1- e^{\epsilon (\langle \|\TT\|, \q_{\perp \mathcal K} \rangle + \langle \boldsymbol \phi,\mathbf d_i \rangle r_i +\langle \boldsymbol \phi,\mathbf d_{n+j} \rangle r_{n+j})}   \big|\Big]  \nonumber \allowdisplaybreaks\\
        &\stackrel{(b)}{\leq} \frac{1}{\epsilon}\EEpe\Big[u_k \big(  e^{\epsilon \theta' \|\q_{\perp \mathcal K}\|} -1   \big)\Big]  \nonumber \allowdisplaybreaks\\
        & \stackrel{(c)}{\leq} \EEpe\bigg[ u_k \theta' \|\q_{\perp \mathcal K}\| e^{\epsilon \theta' \|\q_{\perp \mathcal K}\|} \bigg] \nonumber \allowdisplaybreaks \\
        & \leq \EEpe[u_k^{2}]^{\frac{1}{2}} \EEpe\bigg[  (\theta' \|\q_{\perp \mathcal K}\|)^2 e^{2\epsilon \theta' \|\q_{\perp \mathcal K}\|} \bigg]^{\frac{1}{2}} \nonumber \allowdisplaybreaks \\
        & \leq \sqrt{\epsilon} \theta' \EEpe\big[ \| \q_{\perp \mathcal K} \|^{4} \big]^{\frac{1}{4}} \EEpe\bigg[ e^{4\epsilon \theta'\| \q_{\perp \mathcal K}\|} \bigg]^{\frac{1}{4}},
    \end{align*}
    where (a) follows by using $Re(\langle \mathbf{d}_i, \boldsymbol \phi \rangle ) \leq 0$ for all $i \in \{1,2,\dots,2n\}$ for any $\boldsymbol \phi \in \boldsymbol \Phi $; (b) follows as $\| \TT\|  + |\langle \boldsymbol \phi,\mathbf d_i \rangle| + |\langle \boldsymbol \phi,\mathbf d_{n+j} \rangle| = \theta'$; (c) holds because $|e^x-1| \leq |x|e^{|x|}$ for any $x\in \mathbb{C}$; and last two inequalities follow by Cauchy-Schwarz inequality and using $\EEpe[u_k^{2}]=\EEpe[u_k]\leq \epsilon$ (similar argument as in Eq. \eqref{eq: 3q_unused_epsilon}). Now, by using Definition \ref{def: switch_ssc},
\begin{align*}
    \lim_{\epsilon\rightarrow 0} \EEpe \big[ \| \q_{\perp \mathcal K}\|^4 \big] < \infty, &&  \lim_{\epsilon\rightarrow 0}\EEpe \big[ e^{4\epsilon \theta'\| \q_{\perp \mathcal K}\| }\big] < \infty.
\end{align*}
Combining these with the above argument gives us,
\begin{align*}
    \eplim \frac{1}{\epsilon}\Big| \EEpe[u_k e^{\epsilon \langle \TT, \q \rangle}] & - \EEpe[u_k e^{\epsilon \sum_{l=1,  l\neq i,l\neq n+j}^{2n} \langle \boldsymbol \phi,\mathbf d_l \rangle r_l}] \Big| = \eplim \sqrt{\epsilon} \theta' \EEpe\big[ \| \q_{\perp \mathcal K} \|^{4} \big]^{\frac{1}{4}} \mathbb{E}\bigg[ e^{4\epsilon \theta'\| \q_{\perp \mathcal K}\|} \bigg]^{\frac{1}{4}} = 0.
\end{align*}
As a consequence, we also get,
\begin{align*}
    \eplim \frac{1}{\epsilon}\Big| \EEpe[u_k e^{\epsilon \langle \TT, \q \rangle}] \Big| = \eplim  \frac{1}{\epsilon}\Big| \EEpe[u_k e^{\epsilon \sum_{l=1,  l\neq i,l\neq n+j}^{2n} \langle \boldsymbol \phi,\mathbf d_l \rangle r_l}] \Big| \leq 1.
\end{align*}
This completes the proof of Lemma \ref{lem: switch_mgf_equivalence_c}.
\hfill $\blacksquare$

\endproof

\subsection{Proof of Theorem \ref{thm: switch_functional_eq}}
\label{app: switch_functional_eq}

\begin{lemma}[Second Order Approximation]
\label{lem: switch_2ndorderapprox}
    Consider the Input-queued switch system as defined in Section \ref{sec: switch_model} operating under a policy that achieves state space collapse according the Definition \ref{def: switch_ssc}. 
\begin{enumerate}[label=(\alph*), ref=\ref{lem: switch_2ndorderapprox}.\alph*]
        \item \label{lem: switch_2ndorderapprox_1} For any $\boldsymbol\theta \in \boldsymbol\Theta $, we have
        \begin{align*}
          \eplim \frac{1}{\epsilon^2}  \EEpe \left[e^{ \epsilon \langle \boldsymbol{\theta}, \mathbf{q}^+ \rangle } \Big( e^{- \epsilon \langle \boldsymbol{\theta}, \mathbf{u} \rangle}  -1 \Big)\right] = -\eplim \frac{1}{\epsilon}  \EEpe \left[ \langle \boldsymbol{\theta}, \mathbf{u} \rangle e^{ \epsilon \langle \boldsymbol{\theta}, \mathbf{q} \rangle } \right].
        \end{align*}
    \item \label{lem: switch_2ndorderapprox_2} For any $\boldsymbol\theta \in \boldsymbol\Theta $, we have
    \begin{align*}
        \eplim \frac{1}{\epsilon^2}  \EEpe \left[  e^{ \epsilon \langle \boldsymbol{\theta}, \mathbf{a }- \mathbf{s } \rangle } - 1\right] = - \frac{1}{n} \langle \boldsymbol \theta, \mathbf 1_{n^2} \rangle + \frac{1}{2} \langle \boldsymbol \theta, \boldsymbol \sigma^2 \boldsymbol \theta \rangle.
    \end{align*}
    \end{enumerate}
\end{lemma}

The proof of Lemma \ref{lem: switch_2ndorderapprox} follows on very similar lines as the proof of Lemma \ref{app: 3q_2ndorderapprox}. One thing to note is that in Lemma \ref{lem: switch_2ndorderapprox_2}, we use $\langle \boldsymbol \nu, \TT \rangle = \frac{1}{n} \langle \boldsymbol \theta, \mathbf 1_{n^2} \rangle$ for any $\boldsymbol\nu \in \mathcal{F}$, and this also implies that $\langle \mathbf s, \TT \rangle = \frac{1}{n} \langle \boldsymbol \theta, \mathbf 1_{n^2} \rangle$ for any $\mathbf s \in \mathcal{X}$ as $\mathcal{X}\subset \mathcal{F}$.  For the sake of brevity, we skip the proof of Lemma \ref{lem: switch_2ndorderapprox}.

\proof{\textit{Proof of Theorem \ref{thm: switch_functional_eq}}.} 
Using similar arguments as in Eq. \eqref{eq: 3q_fun_eq_theo_lhs} and \eqref{eq: 3q_onestepbefore_func_eq} in combination with the results from Lemma \ref{lem: switch_mgf_equivalence}, we have that for any $\TT \in \capTT$,
\begin{align}
\label{eq: switch_onestepbefore_func_eq}
    \eplim \frac{1}{\epsilon^2} \EEpe \Big[  e^{ \epsilon \langle \boldsymbol{\theta}, \mathbf{q}^+ \rangle } \Big( e^{- \epsilon \langle \boldsymbol{\theta}, \mathbf{u} \rangle}  -1 \Big) \Big] = \eplim \frac{1}{\epsilon^2} \EEpe \Big[ e^{ \epsilon \langle \boldsymbol{\theta}, \mathbf{q} \rangle } \Big] \Bigg( \EEpe \Big[ e^{ \epsilon \langle \boldsymbol{\theta}, \mathbf{a }- \mathbf{s } \rangle } \Big] - 1\Bigg).
\end{align}
From Lemma \ref{lem: switch_2ndorderapprox_1}, we have
\begin{equation}
\label{eq: switch_marginal_m}
    \lim_{\epsilon \rightarrow 0} \frac{1}{\epsilon^2} \EEpe \Big[e^{ \epsilon \langle \boldsymbol{\theta}, \mathbf{q}^+ \rangle } \Big( e^{- \epsilon \langle \boldsymbol{\theta}, \mathbf{u} \rangle}  -1 \Big) \Big] =  - \Bigg \langle  \boldsymbol{\theta}, \lim_{\epsilon \rightarrow 0} \frac{1}{\epsilon} \EEpe \Big[  \mathbf{u} e^{ \epsilon \langle \boldsymbol{\theta}, \mathbf{q} \rangle} \Big] \Bigg \rangle =  - \langle \boldsymbol{\theta}, \mathbf{M}(\boldsymbol{\theta}) \rangle.
\end{equation}
and from Lemma \ref{lem: switch_2ndorderapprox_2}, we have 
\begin{align}
\label{eq: switch_arrival_2ndorder}
    \eplim \frac{1}{\epsilon^2}\EEpe \Big[ e^{ \epsilon \langle \boldsymbol{\theta}, \mathbf{a }- \mathbf{s } \rangle } -1 \Big]  = - \frac{1}{n} \langle \boldsymbol \theta, \mathbf 1_{n^2} \rangle + \frac{1}{2} \langle \boldsymbol \theta, \boldsymbol \sigma^2 \boldsymbol \theta \rangle.
\end{align}
Plugging these in the Eq. \eqref{eq: switch_onestepbefore_func_eq}, and using the fact that $\lim_{\epsilon \rightarrow 0} \EEpe \left[ e^{ \epsilon \langle \boldsymbol{\theta}, \mathbf{q} \rangle } \right]$ exists by Lemma \ref{lem: switch_mgf_equivalence}, we get that for any $\boldsymbol \theta \in \boldsymbol \Theta$,
\begin{align*}
    - \langle \boldsymbol{\theta}, \mathbf{M}(\boldsymbol{\theta}) \rangle
   & = \Big( - \frac{1}{n} \langle \boldsymbol \theta, \mathbf 1_{n^2} \rangle + \frac{1}{2} \langle \boldsymbol \theta, \boldsymbol \sigma^2 \boldsymbol \theta \rangle \Big) \lim_{\epsilon \rightarrow 0} \EEpe \Big[ e^{ \epsilon \langle \boldsymbol{\theta}, \mathbf{q} \rangle } \Big] \\
   & = \left( - \frac{1}{n} \langle \boldsymbol \theta, \mathbf 1_{n^2} \rangle + \frac{1}{2} \langle \boldsymbol \theta, \boldsymbol \sigma^2 \boldsymbol \theta \rangle \right) L(\boldsymbol{\theta}),
\end{align*}
where, 
\begin{align*}
    L(\boldsymbol \theta) = \lim_{\epsilon \rightarrow 0} \EEpe[e^{\epsilon \langle \boldsymbol \theta, \q \rangle}], && M_k(\boldsymbol \theta) = \lim_{\epsilon \rightarrow 0} \frac{1}{\epsilon} \EEpe [u_k e^{\epsilon \langle \boldsymbol \theta, \q \rangle}], \ \forall k\in \{1,2,\dots,n^2\}
\end{align*}
This gives us the functional equation in Eq. \eqref{eq: switch_functional_eq}. \hfill $\blacksquare$
\endproof

\subsection{Proof of Theorem \ref{thm: switch_dist}}
\label{app: switch_dist}
 Recall that $\TT\in \capTT$, where 
\[\capTT = \{\TT\in \mathbb C^{n^2}: \TT\in \mathcal S, \ Re(\mathbf B^T\TT) \leq \mathbf 0_{2n} \}.\]
As $\TT\in \mathcal{S}$ for any $\TT\in \capTT$, we have that there exists $\boldsymbol\phi \in \mathbb C^{2n}$ such that $\TT = \mathbf B \boldsymbol \phi$. Suppose
\[\boldsymbol\Phi = \{\boldsymbol \phi\in \mathbb C^{2n}: Re(\langle \mathbf d_i, \boldsymbol \phi \rangle \leq 0, \ \forall 1\leq i\leq 2n)\}, \]
where $\mathbf d_i$'s are the columns of the matrix $\mathbf D = \mathbf B^T \mathbf B$.
Then, for any $\boldsymbol \phi \in \boldsymbol \Phi$, we have that $\boldsymbol\theta = \mathbf B \boldsymbol \phi \in \boldsymbol \Theta$. Conversely, for any $\TT\in\capTT$, there exists $\boldsymbol \phi \in \boldsymbol \Phi$ such that $\TT = \mathbf B \boldsymbol \phi $. 
So, for the ease of notations, we use $\TT$ and $\boldsymbol \phi$ interchangeably for the proof of Theorem \ref{thm: switch_dist}, where $\boldsymbol \phi \in \boldsymbol \Phi$ and $\TT =\mathbf B \boldsymbol \phi \in \capTT $.

\proof{\textit{Proof for Theorem \ref{thm: switch_dist}}.}
% For the ease of notations, we consider $\boldsymbol \Phi = \{\boldsymbol \phi\in \mathbb C^{2n}: Re(\langle \mathbf d_i, \boldsymbol \phi \rangle \leq 0, \ \forall 1\leq i\leq 2n)\}$, where $\mathbf d_i$'s are the columns of the matrix $\mathbf D = \mathbf B^T \mathbf B$. Then, for any $\boldsymbol\phi \in \boldsymbol\Phi$, $\boldsymbol\theta = \mathbf B \boldsymbol \phi \in \boldsymbol \Theta$. Also, with slight abuse of notation, we use $\boldsymbol \theta$ and $\boldsymbol \phi$ interchangeably. 
Using Lemma \ref{lem: switch_projection}, and as $\boldsymbol \theta = \mathbf{B} \boldsymbol \phi$,
 \begin{equation*}
       \epsilon \langle \boldsymbol \theta, \q\rangle \stackrel{d}{\rightarrow} \sum_{i=1}^{2n} \langle \boldsymbol \phi , \mathbf d_{i} \rangle (\Upsilon_i - \Tilde{\Upsilon} ),
 \end{equation*} 
where $\Tilde{\Upsilon}  = \min_{1\leq k\leq 2n} \Upsilon_k $.
 For any $i$ and $j \neq i$, due to the strong memoryless property of exponential random variables, $\{\Upsilon_{j} - \Upsilon_{i} | \Tilde{\Upsilon} = \Upsilon_{i}\}$ is an exponential random variable with mean $\frac{\sigma^2}{2}$. Also, for any $j$ and $k$, $\{\Upsilon_{j} - \Upsilon_{i} | \Tilde{\Upsilon} = \Upsilon_{i}\}$ and $\{\Upsilon_{k} - \Upsilon_{i} | \Tilde{\Upsilon} = \Upsilon_{i}\}$ are independent of each other. And as $\{\Upsilon_1,\dots,\Upsilon_{2n}\}$ are independent and identically distributed, $\mathbb{P}(\Tilde{\Upsilon} = \Upsilon_{i}) = \frac{1}{2n}$ for any $i$. Thus, the Laplace transform of $\mathbf B (\boldsymbol \Upsilon - \tilde{\Upsilon} \mathbf 1_{2n})$ is given by 
 \begin{align*}
      \mathbb{E}[ e^{\sum_{j=1}^{2n} \langle \boldsymbol \phi , \mathbf d_{j} \rangle (\Upsilon_j - \Tilde{\Upsilon} )}]
      & = \sum_{i=1}^{2n} \mathbb{P}(\Tilde{\Upsilon} = \Upsilon_{i}) \mathbb{E}[ e^{\sum_{j=1}^{2n} \langle \boldsymbol \phi , \mathbf d_{j} \rangle (\Upsilon_j - \Tilde{\Upsilon} )} | \Tilde{\Upsilon} = \Upsilon_{i}] \allowdisplaybreaks\\
      & = \sum_{i=1}^{2n} \mathbb{P}(\Tilde{\Upsilon} = \Upsilon_{i}) \prod_{j\neq i} \mathbb{E}[ e^{\langle \boldsymbol \phi , \mathbf d_{j} \rangle (\Upsilon_j - \Tilde{\Upsilon} )} | \Tilde{\Upsilon} = \Upsilon_{i}]\allowdisplaybreaks\\
      & = \sum_{i=1}^{2n} \frac{1}{2n} \times \frac{ 1- \langle \boldsymbol \phi , \mathbf d_{i} \rangle \frac{\sigma^2}{2} }{ \prod_{j}\big( 1- \langle \boldsymbol \phi , \mathbf d_{j} \rangle \frac{\sigma^2}{2}\big)}\\
      & = \frac{ 1 - \langle \boldsymbol \phi , \mathbf 1_{2n}\rangle \frac{\sigma^2}{2}\ }{\prod_{j}\big( 1- \langle \boldsymbol \phi , \mathbf d_{j} \rangle \frac{\sigma^2}{2}\big)}.
 \end{align*}
 Thus, for any $i,j \in \{1,\dots,n\}$
 \begin{align}
 \label{eq: switch_laplacesolution}
     L(\boldsymbol \theta) = \frac{ 1 - \langle \boldsymbol \phi , \mathbf 1_{2n}\rangle \frac{\sigma^2}{2}\ }{\prod_{j}\big( 1- \langle \boldsymbol \phi , \mathbf d_{j} \rangle \frac{\sigma^2}{2}\big)}, && M_{i+n(j-1)}(\boldsymbol \theta)= \frac{\big( 1- \langle \boldsymbol \phi , \mathbf d_{i} \rangle \frac{\sigma^2}{2}\big)\times\big( 1- \langle \boldsymbol \phi , \mathbf d_{n+j} \rangle \frac{\sigma^2}{2}\big)}{n\prod_{k}\big( 1- \langle \boldsymbol \phi , \mathbf d_{k} \rangle \frac{\sigma^2}{2}\big)}.
 \end{align}
 Now, for this to satisfy the functional equation given in Eq. \eqref{eq: switch_functional_eq}, we need,
 \begin{align}
 \label{eq: switch_provefunctional}
     \Big( -\frac{1}{n}\langle \boldsymbol{\theta}, \mathbf{1}_{n^2} \rangle + &\frac{1}{2} \langle \boldsymbol{\theta} ,  \boldsymbol\sigma^2\boldsymbol{\theta} \rangle \Big) \left(1 - \langle \boldsymbol \phi , \mathbf 1_{2n}\rangle \frac{\sigma^2}{2}\right) \nonumber\\
     &= - \frac{1}{n}\sum_{i=1}^n\sum_{j=1}^n \theta_{i+n(j-1)} \times\left( 1- \langle \boldsymbol \phi , \mathbf d_{i} \rangle \frac{\sigma^2}{2}\right)\times\left( 1- \langle \boldsymbol \phi , \mathbf d_{n+j} \rangle \frac{\sigma^2}{2}\right).
 \end{align}
 Under the symmetric variance condition, that is $\boldsymbol \sigma^2 = \sigma^2 \mathbf{I}_{n^2}$, we have $\langle \boldsymbol{\theta} ,  \boldsymbol\sigma^2\boldsymbol{\theta} \rangle = \sigma^2 \langle\boldsymbol{\theta}, \boldsymbol{\theta} \rangle $
 % then,
 % \begin{align}
 % \label{eq: switch_lhs}
 %     \langle \boldsymbol{\phi} ,  \boldsymbol\Gamma \boldsymbol{\phi} \rangle = \boldsymbol{\phi}^T \mathbf{B}^T \boldsymbol \sigma^2 \mathbf{B}\boldsymbol{\phi}  = \sigma^2\boldsymbol{\phi}^T \mathbf{B}^T \mathbf{I}_{n^2} \mathbf{B}\boldsymbol{\phi} = \sigma^2\boldsymbol{\phi}^T \mathbf{B}^T  \mathbf{B}\boldsymbol{\phi} = \sigma^2 \langle\boldsymbol{\theta}, \boldsymbol{\theta} \rangle 
 % \end{align}
 Further, the RHS in the Eq. \eqref{eq: switch_provefunctional} can be simplified as follows,
 \begin{align*}
    \frac{1}{n} \sum_{i=1}^n\sum_{j=1}^n &\theta_{i+n(j-1)} \times \left( 1- \langle \boldsymbol \phi , \mathbf d_{i} \rangle \frac{\sigma^2}{2}\right) \times\left( 1- \langle \boldsymbol \phi , \mathbf d_{n+j} \rangle \frac{\sigma^2}{2}\right)\\ 
    &\stackrel{(a)}{=} \frac{1}{n}\sum_{i=1}^n\sum_{j=1}^n (\phi_{i} +\phi_{n+j}) \times \left( 1- \langle \boldsymbol \phi , \mathbf d_{i} \rangle \frac{\sigma^2}{2}\right)\times\left( 1- \langle \boldsymbol \phi , \mathbf d_{n+j} \rangle \frac{\sigma^2}{2}\right)\allowdisplaybreaks\\
     & = \frac{1}{n}\sum_{i=1}^n \phi_{i}\left( 1- \langle \boldsymbol \phi , \mathbf d_{i} \rangle \frac{\sigma^2}{2}\right)\sum_{j=1}^n\left( 1- \langle \boldsymbol \phi , \mathbf d_{n+j} \rangle \frac{\sigma^2}{2}\right)\allowdisplaybreaks\\
     & \quad \quad+\frac{1}{n}\sum_{j=1}^n \phi_{n+j}\left( 1- \langle \boldsymbol \phi , \mathbf d_{n+j} \rangle \frac{\sigma^2}{2}\right)\sum_{i=1}^n\left( 1- \langle \boldsymbol \phi , \mathbf d_{i} \rangle \frac{\sigma^2}{2}\right)\allowdisplaybreaks\\
     & \stackrel{(b)}{=} \sum_{i} \phi_{i}\left( 1- \langle \boldsymbol \phi , \mathbf d_{i} \rangle \frac{\sigma^2}{2}\right)\left( 1- \langle \boldsymbol \phi , \mathbf 1_{2n} \rangle \frac{\sigma^2}{2}\right)+ \sum_{j} \phi_{j}\left( 1- \langle \boldsymbol \phi , \mathbf d_{j} \rangle \frac{\sigma^2}{2}\right)\left( 1- \langle \boldsymbol \phi , \mathbf 1_{2n} \rangle\frac{\sigma^2}{2}\right)\allowdisplaybreaks\\
     & = \left( 1- \langle \boldsymbol \phi , \mathbf 1_{2n} \rangle \frac{\sigma^2}{2}\right)\sum_{i=1}^{2n} \phi_{i}\left( 1- \langle \boldsymbol \phi , \mathbf d_{i} \rangle \frac{\sigma^2}{2}\right)\allowdisplaybreaks\\
     & = \left( 1- \langle \boldsymbol \phi , \mathbf 1_{2n} \rangle \frac{\sigma^2}{2}\right) \left( \langle \boldsymbol \phi , \mathbf 1_{2n} \rangle -\frac{\sigma^2}{2} \sum_{i=1}^{2n}  \langle \boldsymbol \phi , \mathbf d_{i} \rangle \phi_i \right)\allowdisplaybreaks\\
     & \stackrel{(c)}{=} \left( 1- \langle \boldsymbol \phi , \mathbf 1_{2n} \rangle \frac{\sigma^2}{2}\right) \left( \frac{1}{n} \langle \TT , \mathbf 1_{n^2} \rangle -\frac{\sigma^2}{2}   \langle \boldsymbol \theta, \boldsymbol \theta \rangle\right),
 \end{align*}
 where (a) follows by using $\TT = \mathbf B \boldsymbol \phi $ and so $\theta_{i+n(j-1)} = \phi_i + \phi_{n+j}$; (b) follows by using $\sum_{i=1}^n \langle \boldsymbol \phi , \mathbf d_{i} \rangle = \sum_{j=1}^n \langle \boldsymbol \phi , \mathbf d_{n+j} \rangle = \langle \boldsymbol \phi , \mathbf 1_{2n} \rangle $; and (c) follows by using $\langle \boldsymbol \phi , \mathbf 1_{2n} \rangle = \frac{1}{n} \langle \boldsymbol \theta , \mathbf 1_{n^2} \rangle $ and \[\sum_{i=1}^{2n}  \langle \boldsymbol \phi , \mathbf d_{i} \rangle \phi_i =\boldsymbol \phi^T \mathbf D \boldsymbol \phi = \boldsymbol \phi^T \mathbf B^T \mathbf B \boldsymbol \phi = \langle \mathbf B \boldsymbol \phi , \mathbf B \boldsymbol \phi \rangle = \langle \TT,\TT \rangle.\]
This gives us that Eq. \eqref{eq: switch_provefunctional} is satisfied for any $\TT\in \capTT$. This shows that $L(\boldsymbol \theta) $ and $\mathbf M (\boldsymbol \theta)$ given in Eq. \eqref{eq: switch_laplacesolution} is a solution of the functional equation. Now, under the assumption that Conjecture \ref{lem: switch_uniqueness} holds true, we get that Eq. \eqref{eq: switch_laplacesolution} gives the unique solution to the functional equation given by Eq. \eqref{eq: switch_functional_eq}. Now, the result follows by using Proposition \ref{prop: switch_mgf_convergence}. This completes the proof. \hfill $\blacksquare$
\endproof

\section{Proofs for \nsys}
\label{app: n_sys}

\subsection{Proof of State Space Collapse under  MaxWeight}
\label{app: n_sys_ssc}

\proof{\textit{Proof}.}
The proof is based on the arguments provided in 
\cite[Lemma 10]{wang2022heavy}.
Consider the Lyapunov function $V_1(\q) = (q_2-q_1)\mathbf{1}_{\{q_2>q_1\}}$.  Suppose $G$  is the infinitesimal generator matrix of the underlying CTMC, then we define the drift of the Lyapunov function $V_1(\mathbf q) $ by,
\begin{equation*}
    \Delta V_1(\mathbf q) \overset{\Delta}{=}  \sum_{\mathbf q'} G(\mathbf q,\mathbf q') \big[V_1(\mathbf q') - V_1(\mathbf q)\big].
\end{equation*}
For the $\mathcal{N}$-system presented in Section \ref{sec: n_sys_model}, under the condition $V_1(\q) \geq 2$, we have $q_2 \geq q_1 +2$ and possible transitions for the queue length vector:
\begin{align*}
    \q' =\begin{cases}
    (q_1+1, q_2) \quad \text{ with rate } \lambda_1, \\
    (q_1,q_2+1) \quad \text{ with rate } \lambda_2,\\
      (q_1,q_2-1) \quad \text{ with rate } \mu_1+\mu_2.
    \end{cases}
\end{align*} 
In the first case, the transition is due to an arrival to $Q_1$ which increases $q_1$ by $1$ and the rate of this transition is $\lambda_1$. In the second case,  the transition is due to an arrival to $Q_2$ which increases $q_2$ by $1$ and the rate of this transition is $\lambda_2$. Finally, the third case due to the service to the second queue. Note that as $q_2 >q_1$, only the second queue is served by both the servers. In this case, $q_2$ decreases by $1$ and the rate of this transition is $\mu_1 + \mu_2$.
Also, note that in all three cases we still satisfy the condition that second queue is greater than the first queue, that is, $q_2' > q_1'$. Thus, under the condition $V(\q) \geq 2$, $V(q') = (q_2'-q_1')$. 
Then, we have 
\begin{align*}
    \Delta V_1(\q) &= \lambda_1 V(q_1+1,q_2) + \lambda_2 V(q_1,q_2+1) + (\mu_1 + \mu_2) V(q_1,q_2-1) - (\lambda_1 + \lambda_1 - \mu_1 - \mu_2)V(q_1,q_2) \\
    &= -\lambda_1 +\lambda_2 - (\mu_1 + \mu_2) \\
    &\stackrel{(a)}{=} -2\mu_1 + \epsilon (2\mu_1 - \gamma \mu_1 -\gamma \mu_2) \\
    &\stackrel{(b)}{\leq} -\mu_1
\end{align*}
where (a) follows by using the Eq. \eqref{eq: n_sys_arrival_vector} and (b) easily follows whenever $2\mu_1 - \gamma \mu_1 -\gamma \mu_2 \leq 0$ or $2\mu_1 - \gamma \mu_1 -\gamma \mu_2> 0$ and  $\epsilon \leq \frac{1}{\mu_1} (2\mu_1 - \gamma \mu_1 -\gamma \mu_2)$. This fulfils the first requirement in \cite[Lemma 10]{wang2022heavy}. For the second condition, note that any transition will change the queue length of exactly one queue by exactly 1 (either increase or decrease). So, for any transition, $V(\q)$ can change by atmost 1. And finally, the third condition in \cite[Lemma 10]{wang2022heavy} is satisfied because all the transition rates are finite. Then, from the result of \cite[Lemma 10]{wang2022heavy}, we get that MaxWeight achieves state space collapse. \hfill $\blacksquare$
\endproof

\subsection{Required Lemma}
Recall that $\pi^\epsilon$ denotes the steady state distribution of the $\mathcal{N}$-system with the heavy traffic parameter $\epsilon$, and as the underlying Markov chain is positive recurrent, $\pi^\epsilon$ exists and is unique. Further, we use $\mathbb E_{\pi^\epsilon}[\cdot]$ to denote the expectation under the probability distribution $\pi_\epsilon$.
\begin{lemma}
\label{lem: nsys_mgf_equivalence}
Consider the \nsys as defined in Section \ref{sec: n_sys_model} operating under the MaxWeight scheduling. Suppose $\boldsymbol \theta \in \boldsymbol \Theta$ where $\boldsymbol \Theta = \{\boldsymbol \theta \in\mathbb{C}^2: Re(\theta_1) \leq 0, Re(\theta_1 + \theta_2 )\leq  0\}$ and $\boldsymbol \Theta^\epsilon = \boldsymbol \Theta \cap \{ \TT\in \mathbb C^2: \|\TT\| \leq \theta_0/4\epsilon \}$, where $\theta_0$ is given in Proposition \ref{prop: n_sys_ssc}. Let $\mathcal A \subseteq \mathcal X$ be a measurable set with respect to the probability measure $\mathbb \pi^\epsilon$. 
\begin{enumerate}[label=(\alph*), ref=\ref{lem: nsys_mgf_equivalence}.\alph*]
\item \label{lem: nsys_mgf_equivalence_a} For any $\epsilon< \epsilon_0$ and $\boldsymbol \theta \in \boldsymbol \Theta^\epsilon$, we have
    \begin{align*}
        \left| \EEpe \left[e^{\epsilon(\theta_1 q_1+\theta_2 q_2)} \mathbf 1_{\mathcal A} \right]\right| < \infty.
    \end{align*}
\item \label{lem: nsys_mgf_equivalence_b} For any $\TT \in \boldsymbol\Theta$, we have 
\begin{align*}
   \eplim \EEpe \left[  e^{\epsilon (\theta_1 q_1 +\theta_2 q_2)} \mathbf{1}_{\mathcal  A}  \right]  = \eplim \EEpe \left[e^{\epsilon(\theta_1 (q_1-q_2)\mathbf{1}_{\{q_1\geq q_2\}} + (\theta_1 + \theta_2)q_2 )}\mathbf{1}_{\mathcal A} \right],
\end{align*}
and
\begin{align*}
    \eplim \left| \EEpe \left[e^{\epsilon(\theta_1 q_1+\theta_2 q_2)} \mathbf 1_{\mathcal A} \right]\right| = \eplim \left| \EEpe \left[e^{\epsilon(\theta_1 (q_1-q_2)\mathbf{1}_{\{q_1\geq q_2\}} + (\theta_1 + \theta_2)q_2 )}\mathbf{1}_{\mathcal  A} \right] \right|  \leq \eplim \PPpe(\mathcal{A}).
\end{align*}
\item \label{lem: nsys_mgf_equivalence_c} Suppose $ \PPpe(\mathcal A) \leq c\epsilon $ for some constant $c$, which is independent of epsilon. Then, for any $\TT\in \boldsymbol \Theta$, we have 
\begin{align*}
     \eplim \frac{1}{\epsilon} \EEpe \left[  e^{\epsilon (\theta_1 q_1 +\theta_2 q_2)} \mathbf{1}_{\mathcal  A}  \right]  = \eplim \frac{1}{\epsilon} \EEpe \left[e^{\epsilon(\theta_1 (q_1-q_2)\mathbf{1}_{\{q_1\geq q_2\}} + (\theta_1 + \theta_2)q_2 )}\mathbf{1}_{\mathcal A}\right],
\end{align*}
and
\begin{align*}
    \eplim  \frac{1}{\epsilon} \left| \EEpe \left[e^{\epsilon(\theta_1 q_1+\theta_2 q_2)} \mathbf 1_{\mathcal A} \right]\right| = \eplim  \frac{1}{\epsilon} \left| \EEpe \left[e^{\epsilon(\theta_1 (q_1-q_2)\mathbf{1}_{\{q_1\geq q_2\}} + (\theta_1 + \theta_2)q_2 )}\mathbf{1}_{\mathcal  A} \right] \right|  \leq  1.
\end{align*}
\end{enumerate}
\end{lemma}
\proof{\textit{Proof of Lemma \ref{lem: nsys_mgf_equivalence}.}}
Recall that $q_{\perp} = (q_2-q_1)\mathbf{1}_{\{q_1 < q_2\}}$. Then, we have
\begin{align*}
    \theta_{1} q_1 + \theta_2 q_2 &= \theta_1 (q_1-q_2)\mathbf{1}_{\{q_1\geq q_2\}} + (\theta_1 + \theta_2)q_2 + \theta_1 (q_1-q_2)\mathbf{1}_{\{q_1 < q_2\}}\\
    & = \theta_1 (q_1-q_2)\mathbf{1}_{\{q_1\geq q_2\}} + (\theta_1 + \theta_2)q_2 - \theta_1 q_{\perp}.
\end{align*}
 This gives us that,
\begin{align}
\label{eq: n_sys_requiredlemma_sscbound}
   & \left| \EEpe \left[   e^{\epsilon (\theta_1 q_1 +\theta_2 q_2)} \mathbf{1}_{\mathcal A}  \right] -  \EEpe \left[e^{\epsilon(\theta_1 (q_1-q_2)\mathbf{1}_{\{q_1\geq q_2\}} + (\theta_1 + \theta_2)q_2 )}\mathbf{1}_{\mathcal A} \right] \right|\nonumber\\
   &\ \ \ \ \leq \EEpe \left[\left|e^{\epsilon(\theta_1 (q_1-q_2)\mathbf{1}_{\{q_1\geq q_2\}} + (\theta_1 + \theta_2)q_2 )}\mathbf{1}_{\mathcal A}\right| \left|  e^{-\epsilon\theta_1 q_{\perp}} -1 \right|  \right] \nonumber\\  
   &\ \ \ \ \stackrel{(a)}{\leq} \EEpe[\left|  e^{-\epsilon\theta_1 q_{\perp}} -1 \right| \mathbf{1}_{\mathcal A} ]\allowdisplaybreaks \nonumber\\
    &\ \ \ \ \stackrel{(b)}{\leq}\EEpe[\epsilon |\theta_1|q_{\perp} e^{\epsilon|\theta_1 |q_{\perp}}  \mathbf{1}_{\mathcal A} ]\allowdisplaybreaks \nonumber\\
    &\ \ \ \  \stackrel{(c)}{\leq}\epsilon |\theta_1| \left (\PPpe(\mathcal{A})\right)^{\frac{1}{2}} \EEpe[ q_{\perp}^4]^{\frac{1}{4}} \EEpe[e^{4\epsilon|\theta_1 |q_{\perp}}  ]^{\frac{1}{4}},
\end{align}
where (a) holds as $Re(\theta_1)\leq $ and $Re(\theta_1 + \theta_2) \leq 0$ as $\TT \in \boldsymbol \Theta^\epsilon \subset \boldsymbol \Theta$; (b) holds because $|e^x-1| \leq |x|e^{|x|}$ for any $x\in \mathbb{C}$ and $q_{\perp} \geq 0$ by definition; (c) holds by using Cauchy-Schwarz inequality twice. As the scheduling policy achieves SSC according to Proposition \ref{prop: n_sys_ssc}, for any $4\epsilon\|\TT\|\leq \theta_0$ (see Proposition \ref{prop: n_sys_ssc}), we have $\eplim \EEpe[ q_{\perp}^4]^{\frac{1}{4}} \EEpe[e^{4\epsilon|\theta_1 |q_{\perp}} ]^{\frac{1}{4}} < \infty$. By again using $Re(\theta_1)\leq $ and $Re(\theta_1 + \theta_2) \leq 0$, we have 
\begin{align*}
    \left|\EEpe \left[e^{\epsilon(\theta_1 (q_1-q_2)\mathbf{1}_{\{q_1\geq q_2\}} + (\theta_1 + \theta_2)q_2 )}\mathbf{1}_{\mathcal A} \right] \right| \leq \EEpe \left[ \left|e^{\epsilon(\theta_1 (q_1-q_2)\mathbf{1}_{\{q_1\geq q_2\}} + (\theta_1 + \theta_2)q_2 )}\mathbf{1}_{\mathcal A} \right| \right]  \leq 1.
\end{align*}
Combining this with Eq. \eqref{eq: n_sys_requiredlemma_sscbound}, we get Lemma \ref{lem: nsys_mgf_equivalence_a}.
By taking $\epsilon\rightarrow 0$ in Eq. \eqref{eq: n_sys_requiredlemma_sscbound}, as $\eplim \boldsymbol \Theta^\epsilon = \boldsymbol \Theta$, we have 
\begin{align*}
      \eplim \EEpe \left[  e^{\epsilon (\theta_1 q_1 +\theta_2 q_2)} \mathbf{1}_{\mathcal  A}  \right]  = \eplim \EEpe \left[e^{\epsilon(\theta_1 (q_1-q_2)\mathbf{1}_{\{q_1\geq q_2\}} + (\theta_1 + \theta_2)q_2 )}\mathbf{1}_{\mathcal A} \right],
\end{align*}
for all $\TT \in \boldsymbol \Theta$.
Note that, this in turn implies that
\begin{align*}
      \eplim \left| \EEpe \left[  e^{\epsilon (\theta_1 q_1 +\theta_2 q_2)} \mathbf{1}_{\mathcal  A}  \right] \right| &= \eplim \left|\EEpe \left[e^{\epsilon(\theta_1 (q_1-q_2)\mathbf{1}_{\{q_1\geq q_2\}} + (\theta_1 + \theta_2)q_2 )}\mathbf{1}_{\mathcal A} \right]\right| \stackrel{(a)}{\leq} \eplim \EEpe \left[\mathbf{1}_{\mathcal A} \right]= \eplim \PPpe(\mathcal{A}).
\end{align*}
This proves Lemma \ref{lem: nsys_mgf_equivalence_b}.
% where (a) holds as $Re(\theta_1)\leq $ and $Re(\theta_1 + \theta_2) \leq 0$. 
Further, from Eq. \eqref{eq: n_sys_requiredlemma_sscbound}, we also have that, 
\begin{align*}
   \frac{1}{\epsilon} & \left| \EEpe \left[   e^{\epsilon (\theta_1 q_1 +\theta_2 q_2)} \mathbf{1}_{\mathcal A}  \right] -  \EEpe \left[e^{\epsilon(\theta_1 (q_1-q_2)\mathbf{1}_{\{q_1\geq q_2\}} + (\theta_1 + \theta_2)q_2 )}\mathbf{1}_{\mathcal A} \right] \right|\\
   & \ \ \ \ \leq |\theta_1| \left (\PPpe(\mathcal{A})\right)^{\frac{1}{2}} \EEpe[ q_{\perp}^4]^{\frac{1}{4}} \EEpe[e^{4\epsilon|\theta_1 |q_{\perp}}  ]^{\frac{1}{4}}.
\end{align*}
Thus, under the condition $\PPpe(\mathcal A) \leq c\epsilon $, for any $\TT\in \boldsymbol\Theta$, we have
\begin{align*}
      \eplim \frac{1}{\epsilon} \EEpe \left[  e^{\epsilon (\theta_1 q_1 +\theta_2 q_2)} \mathbf{1}_{\mathcal  A}  \right]  = \eplim \frac{1}{\epsilon} \EEpe \left[e^{\epsilon(\theta_1 (q_1-q_2)\mathbf{1}_{\{q_1\geq q_2\}} + (\theta_1 + \theta_2)q_2 )}\mathbf{1}_{\mathcal A}\right],
\end{align*}
and 
\begin{align*}
      \eplim \frac{1}{\epsilon} \left| \EEpe \left[  e^{\epsilon (\theta_1 q_1 +\theta_2 q_2)} \mathbf{1}_{\mathcal  A}  \right] \right| &= \eplim \frac{1}{\epsilon} \left|\EEpe \left[e^{\epsilon(\theta_1 (q_1-q_2)\mathbf{1}_{\{q_1\geq q_2\}} + (\theta_1 + \theta_2)q_2 )}\mathbf{1}_{\mathcal A} \right]\right| \stackrel{(a)}{\leq} \eplim \frac{1}{\epsilon} \PPpe(\mathcal{A}) \leq c,
\end{align*}
where (a) holds as $Re(\theta_1)\leq $ and $Re(\theta_1 + \theta_2) \leq 0$.
\hfill $\blacksquare$
\endproof

\subsection{Proof of Theorem \ref{thm: n_sys_mgf_eq}}
\label{app: n_sys_mgf_eq}

\proof{\textit{Proof of Theorem \ref{thm: n_sys_mgf_eq}}.}
Consider the exponential Lyapunov function defined as 
\begin{equation*}
    V(\mathbf q) \overset{\Delta}{=} e^{\epsilon(\theta_1 q_1+\theta_2 q_2)}.
\end{equation*}

% \color{red}
% In order to prove the functional equation (i.e., Eq. \eqref{eq: n_sys_mgf_eq}) in Theorem \ref{thm: n_sys_mgf_eq}, we first need that the Laplace transform of the heavy traffic distribution exists, i.e., the absolute value of $L(\boldsymbol \theta), M_1(\boldsymbol \theta)$ and $M_2(\boldsymbol \theta)$ are finite. 
%  Using Lemma \ref{lem: nsys_mgf_equivalence}, we get that, $|L(\boldsymbol \theta)| < \infty$ for all $\boldsymbol \theta \in \Theta$. Similarly, $|M_1(\boldsymbol \theta)| < \infty$ and $|M_2(\boldsymbol \theta)| < \infty$ for all $\boldsymbol \theta \in \Theta$.
%  \color{black}
% This holds because $Re(\boldsymbol\theta) \leq 0$ as
% \begin{equation*}
%     |L(\boldsymbol \theta)| \leq \lim_{\epsilon\rightarrow 0} \mathbb{E}[|e^{\epsilon(\theta_1 q_1+\theta_2 q_2)}|] \leq 1.
% \end{equation*}
% Similarly, $|M_1(\boldsymbol \theta)|$ and $|M_2(\boldsymbol \theta)|$ are also finite. 
% Now, consider the exponential Lyapunov function 

Suppose $G$  is the infinitesimal generator matrix of the underlying CTMC, then we define the drift of the Lyapunov function $V(\mathbf q) $ by,
\begin{equation*}
    \Delta V(\mathbf q) \overset{\Delta}{=}  \sum_{\mathbf q'} G(\mathbf q,\mathbf q') \big[V(\mathbf q') - V(\mathbf q)\big].
\end{equation*}
For the $\mathcal{N}$-system presented in Section \ref{sec: n_sys_model}, there are four possible transition as shown below,
\begin{align*}
    \q' =\begin{cases}
    (q_1+1, q_2) \quad \text{ with rate } \lambda_1, \\
    (q_1,q_2+1) \quad \text{ with rate } \lambda_2,\\
      (q_1-1,q_2)  \quad \text{ with rate } \mu_1 \text{ if } q_1>q_2 \text{ and } q_1>0, \\
      (q_1,q_2-1) \quad \text{ with rate } \mu_2 \text{ if } q_1>q_2 \text{ and } q_2>0,\\
      (q_1,q_2-1) \quad \text{ with rate } \mu_1+\mu_2 \text{ if } q_1\leq q_2 \text{ and } q_2>0.
    \end{cases}
\end{align*}
In the above equation, the first two cases correspond to the drift due to arrivals in each of the queue; the third and the fourth case correspond to the service for each queue when $q_1>q_2$; and the last case correspond to the service when $q_2\geq q_1$ in which case the only $Q_2$ is served. Also, a queue can be served only when the queue length is greater than zero, and hence we have the condition, $q_1>0$ in the third case and $q_2>0$ in fourth and fifth case.
Using this definition, the drift $\Delta V(\mathbf q)$ is given by,
\begin{align}
\label{eq: n_sys_thm1_drift}
    \Delta V(\mathbf q) &= e^{\epsilon(\theta_1 q_1 + \theta_2 q_2)} [ \lambda_1 (e^{\epsilon\theta_1 }-1) +  \lambda_2 (e^{\epsilon\theta_2 }-1)+ \mu_1(e^{-\epsilon\theta_1 }-1) \mathbf{1}_{\{q_1 > q_2,q_1>0\}} \nonumber\\ 
    & \ \ \ \  +\mu_2(e^{-\epsilon\theta_2 }-1) \mathbf{1}_{\{q_1 > q_2,q_2>0\}}+ (\mu_2+\mu_1)(e^{-\epsilon\theta_2 }-1) \mathbf{1}_{\{q_1 \leq q_2,q_2 >0\}} \big] \nonumber\\
    & \nonumber\\
    & = e^{\epsilon(\theta_1 q_1 + \theta_2 q_2)} \big[ \lambda_1 (e^{\epsilon\theta_1 }-1) +  \lambda_2 (e^{\epsilon\theta_2 }-1)+ \mu_1(e^{-\epsilon\theta_1 }-1)  +\mu_2(e^{-\epsilon\theta_2 }-1) \nonumber\\
    & \ \ \ \ + \mu_1(e^{-\epsilon\theta_2 }-e^{-\epsilon\theta_1 })\mathbf{1}_{\{q_1 \leq q_2\}}\big] \nonumber\\ 
    & \ \ \ \  - \mu_2e^{\epsilon\theta_1 q_1 }(e^{-\epsilon\theta_2 }-1) \mathbf{1}_{\{q_2 =0\}} - \mu_1(e^{-\epsilon\theta_2 }-1) \mathbf{1}_{\{q_1=q_2 =0\}},
\end{align}
where we have used that $e^{\epsilon(\theta_1 q_1+\theta_2 q_2) }\mathbf{1}_{\{q_2 =0\}}= e^{\epsilon\theta_1 q_1 }\mathbf{1}_{\{q_2 =0\}}$ and $e^{\epsilon(\theta_1 q_1+\theta_2 q_2) }\mathbf{1}_{\{q_1=q_2 =0\}}= \mathbf{1}_{\{q_2 =0\}}$. 
As the underlying Markov chain is positive recurrent, we have that if $\q$ follows the steady state distribution $\pi^\epsilon$, then $\q'$ also follows the distribution $\pi^\epsilon$. 
% Further, for Lemma \ref{lem: nsys_mgf_equivalence}, we have $\eplim |\EEpe[V(\q)]| \leq 1 $. Hence, there exists $\epsilon_{\mathcal{N}}>0$ and $C_\mathcal{N}>1$ such that, for all $\epsilon < \epsilon_\mathcal{N}$, we have $|\EEpe[V(\q)]| \leq C_n$. 
Further, from Lemma \ref{lem: nsys_mgf_equivalence_a}, we have that $|\EEpe[V(\q)]| < \infty$ for any $\TT\in \boldsymbol \Theta^\epsilon$. So, $\EEpe[V(\q')] = \EEpe[V(\q)]$. This implies that $\EEpe[\Delta V(\q)] = \EEpe[V(\q')] - \EEpe[V(\q)] =0$.
This gives us,
\begin{align}
\label{eq: n_sys_funcprove_termall}
   0 =\EEpe[\Delta V(\q)] =\big[& \lambda_1 (e^{\epsilon\theta_1 }-1) +  \lambda_2 (e^{\epsilon\theta_2 }-1)+ \mu_1(e^{-\epsilon\theta_1 }-1)  +\mu_2(e^{-\epsilon\theta_2 }-1)\big] \EEpe[e^{\epsilon(\theta_1 q_1 + \theta_2 q_2)}]\nonumber \\
    & +\mu_1(e^{-\epsilon\theta_2 }-e^{-\epsilon\theta_1 })  \EEpe[e^{\epsilon(\theta_1 q_1 + \theta_2 q_2)}\mathbf 1_{\{q_1\leq q_2\}} ]\nonumber\\
    & - (e^{-\epsilon\theta_2 }-1)\EEpe[e^{\epsilon\theta_1 q_1 }\left(\mu_2\mathbf 1_{\{q_2 = 0\}}+\mu_1\mathbf 1_{\{q_2 =q_1=0\}} \right) ]. 
    % - \mu_1(e^{-\epsilon\theta_2 }-1) \mathbb{P}(q_1=q_2 =0).
\end{align}
Using these arguments, we present a step by step procedure to get the functional equation as follows.

\textbf{Step 1:} By putting $\theta_2 = 0$ in Eq. \eqref{eq: n_sys_funcprove_termall}, we get that, for any $\theta_1$ such that $Re(\theta_1)\leq 0$, we have 
% \begin{align*}
%     \Delta V(\q)  &= e^{\epsilon\theta_1 q_1} [ \lambda_1 (e^{\epsilon\theta_1 }-1) +  \mu_1(e^{-\epsilon\theta_1 }-1) \mathbf{1}_{\{q_1 > q_2\}} \big]\\
%     & = (e^{\epsilon\theta_1 }-1) e^{\epsilon\theta_1 q_1} [ \lambda_1  -  \mu_1e^{-\epsilon\theta_1 }\mathbf{1}_{\{q_1 > q_2\}} \big]
% \end{align*}
% By taking expectation on both sides with respect to the steady state distribution and using $\EEpe[\Delta V(\q)] = 0$ in steady state, we get
\begin{align*}
   (e^{\epsilon\theta_1 }-1) \Big[ \lambda_1 \EEpe[ e^{\epsilon\theta_1 q_1} ] - \mu_1 e^{-\epsilon\theta_1} \EEpe[ e^{\epsilon\theta_1 q_1}\mathbf{1}_{\{q_1 > q_2\}} ] \Big] =0.
\end{align*}
Thus for any $\theta_1$ such that $Re(\theta_1)<0$, we have
\begin{align*}
 \lambda_1 \EEpe[ e^{\epsilon\theta_1 q_1} ] - \mu_1 e^{-\epsilon\theta_1} \EEpe[ e^{\epsilon\theta_1 q_1}\mathbf{1}_{\{q_1 > q_2\}} ]  =0.
\end{align*}
Now, by taking $\theta_1\in (-\infty,0)$ and $\theta_1 \rightarrow 0$ in the above equation, we get that,
\begin{align}
\label{eq: n_sys_prob_surface1}
    \PPpe(q_1 \leq q_2) = 1-\frac{\lambda_1}{ \mu_1} =  \epsilon.
\end{align}
\textbf{Step 2:} 
% By taking expectation on both sides and equating $ \EEpe[\Delta V(\q)] =0$ in steady state,
Similar to that in Step 1, by putting $\theta_1 = 0$ in Eq. \eqref{eq: n_sys_funcprove_termall}, and dividing by the term $(e^{-\epsilon \theta_2}-1)$ for $\theta_2\neq 0$ on both sides, we have that for any $\theta_2$ such that $Re(\theta_2)<0$,
\begin{equation*}
     (\mu_2 - \lambda_2e^{\epsilon \theta_2})\EEpe[e^{\epsilon\theta_2 q_2}] + \mu_1 \EEpe[e^{\epsilon\theta_2 q_2}\mathbf{1}_{\{ q_1\leq q_2\}}]  - \mu_2 \PPpe(q_2=0) - \mu_1 \PPpe(q_1=q_2 =0) = 0.
\end{equation*}
Now, by taking $\theta_2 \in (-\infty,0)$ and $\theta_2 \rightarrow 0$ to get,
\begin{align*}
   & \mu_2 - \lambda_2 +\mu_1 \PPpe(q_1\leq q_2) - \mu_2 \PPpe(q_2=0) - \mu_1 \PPpe(q_1=q_2 =0) = 0.
\end{align*}
By combining the above equation with $\PPpe(q_1\leq q_2) = 1- \frac{\lambda_1}{\mu_1}$ from Eq. \eqref{eq: n_sys_prob_surface1}, we get 
\begin{align}
\label{eq: n_sys_prob_surface2}
    \mu_2 \PPpe(q_2=0) + \mu_1 \PPpe(q_1=q_2 =0) = \mu_1  + \mu_2  - \lambda_1- \lambda_2  = \gamma \epsilon(\mu_1  + \mu_2).
\end{align}
\textbf{Step 3:} Now, we do the heavy traffic approximation, where we use the second order Taylor expansion of complex exponential function. Essentially, as $\EEpe[\Delta V(\q)] =0$ for all value of $\epsilon$, we have
\begin{align}
\label{eq: n_sys_funcprove_secondorder}
    \eplim \frac{1}{\epsilon^2} \EEpe [\Delta V(\q)] =0.
\end{align}
For the first term, we have 
\begin{align*}
   \lambda_1 &(e^{\epsilon\theta_1 }-1)  +  \lambda_2 (e^{\epsilon\theta_2 }-1)+ \mu_1(e^{-\epsilon\theta_1 }-1)  +\mu_2(e^{-\epsilon\theta_2 }-1) \\
   & \stackrel{(a)}{=} \mu_1(1-\epsilon) (e^{\epsilon\theta_1 }-1)  + \big((1-\gamma\epsilon) \mu_2 + \epsilon \mu_1 (1-\gamma)\big) (e^{\epsilon\theta_2 }-1)+\mu_1(e^{-\epsilon\theta_1 }-1)  +\mu_2(e^{-\epsilon\theta_2 }-1) \\ 
   & = \mu_1 \left(e^{\epsilon\theta_1 }+e^{-\epsilon\theta_1 } -2\right) + \mu_2 \left(e^{\epsilon\theta_2 }+e^{-\epsilon\theta_2 }-2\right) - \epsilon\left( \mu_1 (e^{\epsilon\theta_1 }-1) + (\gamma \mu_2+ (1-\gamma)\mu_1)(e^{\epsilon\theta_2 }-1)  \right). 
\end{align*}
where (a) follows by using the value of $\lambda_1$ and $\lambda_2$ as in Eq. \eqref{eq: n_sys_arrival_vector}. 
Now, by using 
\begin{align*}
    \eplim \frac{1}{\epsilon^2} \left(e^{\epsilon\theta }+e^{-\epsilon\theta } -2\right) = \theta^2, && \eplim \frac{1}{\epsilon} \left(e^{\epsilon\theta } -1\right) = \theta,
\end{align*}
for any $\theta \in \mathbb C$,  we have 
\begin{align*}
   \eplim \frac{1}{\epsilon^2} & \left( \lambda_1 (e^{\epsilon\theta_1 }-1)  +  \lambda_2 (e^{\epsilon\theta_2 }-1)+ \mu_1(e^{-\epsilon\theta_1 }-1)  +\mu_2(e^{-\epsilon\theta_2 }-1)\right) \\
   & = \mu_1 \theta_1^2 + \mu_2\theta_2^2 - \mu_1\theta_1 - (\gamma \mu_2+ (1-\gamma)\mu_1)\theta_2\\
   & = \mu_1 (-\theta_1+\theta_1^2) + \mu_2 (-\gamma\theta_2+\theta_2^2)+ \theta_2\mu_1(1-\gamma).
\end{align*}
Thus,
\begin{align}
\label{eq: n_sys_funcprove_term1}
\eplim \frac{1}{\epsilon^2} \lambda_1 &(e^{\epsilon\theta_1 }-1)  +  \lambda_2 (e^{\epsilon\theta_2 }-1)+ \mu_1(e^{-\epsilon\theta_1 }-1)  +\mu_2(e^{-\epsilon\theta_2 }-1) \EEpe[e^{\epsilon(\theta_1 q_1 + \theta_2 q_2)}] \nonumber\\
&= \mu_1 (-\theta_1+\theta_1^2) + \mu_2 (-\gamma\theta_2+\theta_2^2)+ \theta_2\mu_1(1-\gamma) \eplim\EEpe[e^{\epsilon(\theta_1 q_1 + \theta_2 q_2)}].
\end{align}
% \begin{align}
%    &\eplim \frac{1}{\epsilon^2} \left( \lambda_1 (e^{\epsilon\theta_1 }-1)  +  \lambda_2 (e^{\epsilon\theta_2 }-1)+ \mu_1(e^{-\epsilon\theta_1 }-1)  +\mu_2(e^{-\epsilon\theta_2 }-1)\right)\nonumber\\
%     & \stackrel{(a)}{=} \mu_1(1-\epsilon) \big(\epsilon\theta_1+ \epsilon^2 \frac{\theta_1^2}{2}\big) + \big((1-\gamma\epsilon) \mu_2 + \epsilon \mu_1 (1-\gamma)\big)\big(\epsilon\theta_2+ \epsilon^2 \frac{\theta_2^2}{2}\big) \nonumber\\
%     & \quad \quad + \mu_1 \big(-\epsilon\theta_1+ \epsilon^2 \frac{\theta_1^2}{2}\big) + \mu_2 \big(-\epsilon\theta_2+ \epsilon^2 \frac{\theta_2^2}{2}\big) + o(\epsilon^2)\nonumber\\
%     & = \epsilon^2 \big[\mu_1 (-\theta_1+\theta_1^2) + \mu_2 (-\gamma\theta_2+\theta_2^2)+ \theta_2\mu_1(1-\gamma) \big] +o(\epsilon^2),
% \end{align}
For the second term, 
% by combining Eq. \eqref{eq: n_sys_prob_surface1} and Lemma \ref{lem: nsys_mgf_equivalence}, we have 
% \begin{align*}
%     \eplim \frac{1}{\epsilon} \EEpe \left[ e^{\epsilon (\theta_1 q_1 +\theta_2 q_2)} \mathbf{1}_{\{q_1\leq q_2\}}  \right] 
%     &= \eplim \frac{1}{\epsilon} \EEpe \left[e^{\epsilon(\theta_1 (q_1-q_2)\mathbf{1}_{\{q_1\geq q_2\}} + (\theta_1 + \theta_2)q_2 )}\mathbf{1}_{\{q_1\leq q_2\}}\right]\\
%     & = \eplim \frac{1}{\epsilon} \EEpe \left[e^{\epsilon (\theta_1 + \theta_2)q_2 }\mathbf{1}_{\{q_1\leq q_2\}}\right].
% \end{align*}
% Also, 
\begin{align*}
    \eplim \frac{1}{\epsilon} \mu_1 (e^{-\epsilon\theta_2} -e^{-\epsilon\theta_1}) = \mu_1 ( \theta_1 - \theta_2).
\end{align*}
Thus,
\begin{align}
\label{eq: n_sys_funcprove_term2}
    \eplim \frac{1}{\epsilon^2 } \mu_1 (e^{-\epsilon\theta_2} -e^{-\epsilon\theta_1}) \EEpe \left[ e^{\epsilon (\theta_1 q_1 +\theta_2 q_2)} \mathbf{1}_{\{q_1\leq q_2\}}  \right] &=  \mu_1 ( \theta_1 - \theta_2) \eplim \frac{1}{\epsilon} \EEpe \left[e^{\epsilon (\theta_1q_1 + \theta_2q_2) }\mathbf{1}_{\{q_1\leq q_2\}}\right].
\end{align}
Finally, for the third term,
% by combining Eq. \eqref{eq: n_sys_prob_surface2}, Lemma \ref{lem: nsys_mgf_equivalence} and $\eplim \frac{1}{\epsilon} (e^{-\epsilon\theta_2 }-1) = -\theta_2$, we have 
\begin{align}
\label{eq: n_sys_funcprove_term3}
  \eplim & \frac{1}{\epsilon^2} (e^{-\epsilon\theta_2 }-1)\EEpe[e^{\epsilon(\theta_1 q_1 +\theta_2 q_2)}\left(\mu_2\mathbf 1_{\{q_2 = 0\}}+\mu_1\mathbf 1_{\{q_2 =q_1=0\}} \right) ] \nonumber\\
  % & = -\theta_2 \eplim \frac{1}{\epsilon} \EEpe[e^{\epsilon(\theta_1 (q_1-q_2)\mathbf{1}_{\{q_1\geq q_2\}} + (\theta_1 + \theta_2)q_2 )}\left(\mu_2\mathbf 1_{\{q_2 = 0\}}+\mu_1\mathbf 1_{\{q_2 =q_1=0\}} \right) ] \nonumber\\
  &= -\theta_2 \eplim \frac{1}{\epsilon} \EEpe[e^{\epsilon (\theta_1 q_1 +\theta_2 q_2)}\left(\mu_2\mathbf 1_{\{q_2 = 0\}}+\mu_1\mathbf 1_{\{q_2 =q_1=0\}} \right) ].
\end{align}
Now by substituting Eq. \eqref{eq: n_sys_funcprove_term1}, \eqref{eq: n_sys_funcprove_term2} and \eqref{eq: n_sys_funcprove_term3} in the Eq. \eqref{eq: n_sys_funcprove_secondorder}, and using the notation,
\begin{align*}
   L(\boldsymbol \theta) =  \eplim\EEpe \left[e^{\epsilon(\theta_1 q_1 + \theta_2 q_2)}\right]
   % =  \eplim \EEpe \left[e^{\epsilon(\theta_1 (q_1-q_2)\mathbf{1}_{\{q_1\geq q_2\}} + (\theta_1 + \theta_2)q_2 )}\mathbf{1}_{\mathcal A} \right],
\end{align*}
and 
\begin{align*}
   M_1(\boldsymbol \theta) = \eplim \frac{1}{\epsilon} \EEpe \left[e^{\epsilon (\theta_1q_1 + \theta_2q_2) }\mathbf{1}_{\{q_1\leq q_2\}}\right], && M_2(\boldsymbol \theta) = \eplim \frac{1}{\epsilon}  \EEpe[e^{\epsilon (\theta_1 q_1 + \theta_2 q_2) }\left(\mu_2\mathbf 1_{\{q_2 = 0\}}+\mu_1\mathbf 1_{\{q_2 =q_1=0\}} \right) ],
\end{align*}
we have that for all $\boldsymbol \theta \in \capTT$,
\begin{equation}
\label{eq: nsys_mgf_withp2}
    L(\boldsymbol \theta) \big[ \mu_2 (-\gamma\theta_2  +\theta_2^2)+\theta_2\mu_1(1-\gamma)+ \mu_1 (-\theta_1+\theta_1^2)\big] +\mu_1(\theta_1-\theta_2) M_1(\boldsymbol \theta)  +\theta_2 M_2(\boldsymbol \theta)   =0.
\end{equation}
\endproof

\subsection{Proof of Lemma \ref{lem: n_sys_uniqueness}}
\label{app: n_sys_uniqueness}

\proof{\textit{Proof of Lemma \ref{lem: switch_analytic}.}}
In order to prove Lemma \ref{lem: switch_analytic}, we first show that all the moments of the scaled total queue length, i.e., $\epsilon (q_1+ q_2)$ exists. We consider the Lyapunov function $V_1(\q):= e^{\epsilon \theta q_1} +  e^{\epsilon \theta q_2}$, where $\theta\in\mathbb R$. 
The drift of the of the function $V_1(\q)$ can be obtained by using similar calculation as in Eq. \eqref{eq: n_sys_funcprove_termall}. We first substitute $\theta_1 = \theta$ and $ \theta_2=0$, to get 
\begin{align*}
    \frac{1}{1-e^{-\epsilon \theta}} \EE[\Delta e^{\epsilon \theta q_1}| \q] = \big[ \lambda_1 e^{\epsilon \theta}-  \mu_1\big] e^{\epsilon\theta q_1}+  \mu_1 e^{\epsilon\theta q_1} \mathbf{1}_{\{q_1\leq q_2\}}.
\end{align*}
Next, by substituting $\theta_1 = 0$ and $ \theta_2=\theta$, we get
\begin{align*}
    \frac{1}{1-e^{-\epsilon \theta}} \EE[\Delta e^{\epsilon \theta q_2}| \q] &= \big[ \lambda_2 e^{\epsilon \theta}-  \mu_2\big] e^{\epsilon\theta q_2}- \mu_1 e^{\epsilon\theta q_2}\mathbf 1_{\{q_1\leq q_2\}}  + \mu_2\mathbf 1_{\{q_2 = 0\}}+\mu_1\mathbf 1_{\{q_2 =q_1=0\}}. 
\end{align*}
This gives us that
\begin{align*}
    \frac{1}{1-e^{-\epsilon \theta}} \EE[\Delta V_1(\q)| \q] & = e^{\epsilon \theta} \big(\lambda_1 e^{\epsilon\theta q_1} + \lambda_2 e^{\epsilon\theta q_2} \big) - \mu_2 e^{\epsilon\theta q_2} - \mu_1 \max\{e^{\epsilon\theta q_1} ,e^{\epsilon\theta q_2}   \}  \nonumber \\
    & \quad + \mu_2\mathbf 1_{\{q_2 = 0\}}+\mu_1\mathbf 1_{\{q_2 =q_1=0\}}.
\end{align*}
Next, we use $\max\{e^{\epsilon\theta q_1} ,e^{\epsilon\theta q_2}   \} \geq \epsilon(1-\gamma) e^{\epsilon\theta q_2} + (1-\epsilon(1-\gamma)) e^{\epsilon\theta q_1}$. Now by using,
\begin{align*}
    \lambda_1 = (1-\epsilon)\mu_1, && \lambda_1 + \lambda_2 =(1-\gamma \epsilon)( \mu_1 + \mu_2),
\end{align*}
we have,
\begin{align*}
    \frac{1}{1-e^{-\epsilon \theta}} \EE[\Delta V_1(\q)| \q] & \leq  \mu_1 \Big(e^{\epsilon \theta} (1-\epsilon) -(1-\epsilon(1-\gamma))\Big) e^{\epsilon\theta q_1} \\
    &\quad + \Big( e^{\epsilon \theta} \big( (1-\gamma \epsilon)\mu_2 + \epsilon(1-\gamma)\mu_1\big) -\mu_2 - \epsilon(1-\gamma) \mu_1\Big)  e^{\epsilon\theta q_2} \nonumber \\
    & \quad + \mu_2\mathbf 1_{\{q_2 = 0\}}+\mu_1\mathbf 1_{\{q_2 =q_1=0\}}\\
    &\leq \mu_1\epsilon( \theta- \gamma) e^{\epsilon\theta q_1} + (\epsilon \theta \mu_2 - \gamma \epsilon \mu_2 + \epsilon^2(1-\gamma)\theta\mu_1)\\
    & \quad + \mu_2\mathbf 1_{\{q_2 = 0\}}+\mu_1\mathbf 1_{\{q_2 =q_1=0\}}\\
    &\leq -\frac{1}{2} \gamma\epsilon \mu_1 e^{\epsilon\theta q_1} -\frac{1}{4}\gamma\epsilon \mu_2 e^{\epsilon\theta q_2}+ \mu_2\mathbf 1_{\{q_2 = 0\}}+\mu_1\mathbf 1_{\{q_2 =q_1=0\}},
\end{align*}
where the last inequality holds for any $\theta<\gamma/2$ and $\epsilon< \frac{\gamma\mu_1}{4(1-\gamma)\mu_2}$. Thus, by equating the drift to zero in steady state, for all $\theta<\gamma/2$ and $\epsilon$ small enough, we have 
\begin{align*}
    \EEpe \big[\mu_1e^{\epsilon\theta q_1}+\mu_2 e^{\epsilon\theta q_1} ] \leq \frac{4}{\epsilon \gamma }\EEpe[\mu_2\mathbf 1_{\{q_2 = 0\}}+ \mu_1\mathbf 1_{\{q_2 =q_1=0\}}].
\end{align*}
Next, by using Eq. \eqref{eq: n_sys_prob_surface2}, we get
\begin{align*}
    \EEpe \big[\mu_1e^{\epsilon\theta q_1}+\mu_2 e^{\epsilon\theta q_1} ] \leq 4(\mu_1+\mu_2).
\end{align*}
As the above equation is true for all $\theta<\gamma/2$ and $\epsilon< \frac{\gamma\mu_1}{4(1-\gamma)\mu_2}$, we have that all the moments of $\epsilon q_1$ and $\epsilon q_2$ exists for $\epsilon< \frac{\gamma\mu_1}{4(1-\gamma)\mu_2}$. This implies that the function $L(\boldsymbol \theta)$ can be represented as a convergent power series for all $\boldsymbol\theta\in \capTT$. This shows that $L(\TT)$ is holomorphic and continuous over $\capTT$. Further, $L(\TT)$ is non zero simply because $L(\mathbf 0) =1$, and $L(\TT)$ is bounded over $\capTT$ by using Lemma \ref{lem: nsys_mgf_equivalence}. As such, we have that $L(\boldsymbol \theta) \in \mathcal{L} (\capTT)$. 

Next, for $M_1(\TT)$ and $M_2(\TT)$
Similarly, $M_1(\TT)$ and $M_2(\TT)$ lies in the set $\mathcal{L} (\capTT)$ by using similar arguments in Lemma \ref{lem: nsys_mgf_equivalence_c} along the existence of moments of $\epsilon q_1$ and $\epsilon q_2$.
\endproof

\proof{\textit{Proof of Lemma \ref{lem: n_sys_uniqueness}}.}
In order to prove Lemma \ref{lem: n_sys_uniqueness}, we are going to Lemma \ref{lem: functional_uniqueness}. Using Lemma \ref{lem: nsys_mgf_equivalence}, we know that
\begin{align*}
   L(\boldsymbol \theta) =  \eplim\EEpe \left[e^{\epsilon(\theta_1 q_1 + \theta_2 q_2)}\right] =  \eplim \EEpe \left[e^{\epsilon(\theta_1 (q_1-q_2)\mathbf{1}_{\{q_1\geq q_2\}} + (\theta_1 + \theta_2)q_2 )} \right],
\end{align*}
and 
\begin{align*}
   M_1(\boldsymbol \theta)& =\eplim \frac{1}{\epsilon} \EEpe \left[e^{\epsilon(\theta_1 (q_1-q_2)\mathbf{1}_{\{q_1\geq q_2\}} + (\theta_1 + \theta_2)q_2 )} \mathbf{1}_{\{q_1\leq q_2\}}\right] = \eplim \frac{1}{\epsilon} \EEpe \left[e^{\epsilon (\theta_1 + \theta_2)q_2 }\mathbf{1}_{\{q_1\leq q_2\}}\right],\\
    M_2(\boldsymbol \theta) 
    &= \eplim \frac{1}{\epsilon}  \EEpe[e^{\epsilon( \theta_1 q_1 +\theta_2 q_2) }\left(\mu_2\mathbf 1_{\{q_2 = 0\}}+\mu_1\mathbf 1_{\{q_2 =q_1=0\}} \right) ]\\
    &= \eplim \frac{1}{\epsilon}  \EEpe[e^{\epsilon \theta_1 q_1 }\left(\mu_2\mathbf 1_{\{q_2 = 0\}}+\mu_1\mathbf 1_{\{q_2 =q_1=0\}} \right) ]
\end{align*}
% We already know that the heavy traffic distribution of the scaled queue length vector i.e., the distribution of $\lim_{\epsilon \rightarrow 0} \epsilon \q$ exists. 
Next, we do a linear transform of the variable $\boldsymbol \theta$ so that the Laplace transform $M_1(\cdot)$ and $M_2(\cdot)$ depends only on one variable. We can pick $\psi_1 = \theta_1$ and $\psi_2 = \theta_1 +\theta_2$, i.e., $\boldsymbol \psi = (\psi_1,\psi_2) = (\theta_1,\theta_1 +\theta_2)$. Note that the mapping $(\psi_1,\psi_2) = (\theta_1,\theta_1 +\theta_2)$ is bijective, as we have that $\boldsymbol \theta = (\psi_1,\psi_2-\psi_1) $. Thus, we can transform the function equation in Theorem \ref{thm: n_sys_mgf_eq} by using $\boldsymbol \theta = (\psi_1,\psi_2-\psi_1) $, i.e., we replace $L(\boldsymbol \theta)$, $M_1(\boldsymbol \theta) $ and $M_2(\boldsymbol \theta)$ with $\tilde L(\boldsymbol \psi)$, $\tilde M_1(\psi_2)$ and  $\tilde M_2(\psi_1)$ respectively, such that 
\begin{align*}
   \tilde L(\boldsymbol \psi) =  \eplim \EEpe \left[e^{\epsilon(\psi_1 (q_1-q_2)\mathbf{1}_{\{q_1\geq q_2\}} + \psi_2 q_2 )}\mathbf{1}_{\mathcal A} \right],
\end{align*}
and
\begin{align*}
   \tilde  M_1( \psi_2) = \eplim \frac{1}{\epsilon} \EEpe \left[e^{\epsilon \psi_2 q_2 }\mathbf{1}_{\{q_1\leq q_2\}}\right], && \tilde  M_2(\psi_1) = \eplim \frac{1}{\epsilon}  \EEpe[e^{\epsilon \psi_1 q_1 }\left(\mu_2\mathbf 1_{\{q_2 = 0\}}+\mu_1\mathbf 1_{\{q_2 =q_1=0\}} \right) ].
\end{align*}
and for $\boldsymbol \psi \in \boldsymbol \Psi $, where
\[\boldsymbol\Psi =\mathbf B^T \capTT = \{ \boldsymbol \psi \in \mathbb C^2: Re(\boldsymbol \psi ) \leq \mathbf 0_2\}, \]
we have 
\begin{align}
\label{eq: nsys_rewritten_func_eq}
    \big((\mu_1+\mu_2) \psi_1^2  +\mu_2 \psi_2^2 -2\mu_2 \psi_1\psi_2 +& \psi_1(\gamma(\mu_1+\mu_2) - 2\mu_1) + \psi_2(\mu_1 - \gamma(\mu_1+\mu_2)) \big) \tilde L(\boldsymbol{\psi}) \nonumber\\
    &+ \mu_1(2\psi_1 -\psi_2)\tilde  M_1(\psi_2) +(\psi_2 - \psi_1) \tilde  M_2(\psi_1)  = 0.
\end{align}
From this, by substituting $\psi_1 =0$, we have 
\begin{align*}
 \mu_1 \tilde M_1(\psi_2) &= \tilde M_2 (0)+ \big((\mu_1+\mu_2) \psi_1  + (\gamma(\mu_1+\mu_2) - 2\mu_1) \big) \tilde L(0,\psi_2)\\
 &=\gamma(\mu_1+\mu_2) + \big((\mu_1+\mu_2) \psi_1  + (\gamma(\mu_1+\mu_2) - 2\mu_1) \big) \tilde L(0,\psi_2).
\end{align*}
From the above relation, as $\tilde L(0,\psi_2)$ is infinitely differentiable with respect to $\psi_2$, it follows that $\tilde M_1(\psi_2)$ is also infinitely differentiable with respect to $\psi_2$. And so, $\tilde M_1(\psi_2)$ is holomorphic and continuous. Also, $\tilde M_1(0)=1$, and so $\tilde M_1(\psi_2)$ is non-zero. Finally, $\tilde M_1(\psi_2)$ is bounded simply by using Lemma \ref{lem: nsys_mgf_equivalence}. As such, $\tilde M_1(\cdot)\in \mathcal{L}(\boldsymbol{\psi})$. Similarly, $\tilde M_2(\cdot)\in \mathcal{L}(\boldsymbol{\psi})$.

% Further, by using Lemma \ref{lem: nsys_mgf_equivalence} again, the absolute values of $L(\boldsymbol \psi)$, $M_1(\psi_2)$ and  $M_2(\psi_1)$ are bounded for all 
% The functional equation can be rewritten as,

As the functional equation in Eq. \eqref{eq: n_sys_mgf_eq} holds for any $\boldsymbol \theta \in \capTT$, so the rewritten functional equation in Eq. \eqref{eq: nsys_rewritten_func_eq} holds for any $\boldsymbol \psi \in  \boldsymbol \Psi $. Further, one can easily verify that the conditions in \eqref{eq: funceqconditions} are satisfied for the functional equation in \eqref{eq: nsys_rewritten_func_eq}.
Thus, the functional equation in Eq. \eqref{eq: nsys_rewritten_func_eq} satisfies the conditions mentioned in Lemma \ref{lem: functional_uniqueness} and so, the functional equation in Eq. \eqref{eq: nsys_rewritten_func_eq} has a unique solution. This in turn implies that there is a unique $L(\TT)$ that satisfies the functional equation in \eqref{eq: n_sys_mgf_eq} for all $\TT\in \capTT$. \hfill $\blacksquare$

\endproof

\subsection{Proof of Theorem \ref{thm: n_sys_distribution_c}}
\label{app: n_sys_distribution}

\proof{\textit{Proof of Theorem \ref{thm: n_sys_distribution_c}}.}
We know that the Laplace transform of distribution uniquely defines the distribution. 
% And also,  
% \begin{equation*}
%     \left|\lim_{\epsilon\rightarrow 0} \mathbb{E}[e^{\epsilon(\theta_1 q_1+\theta_2 q_2)}] \right| \leq \lim_{\epsilon\rightarrow 0} \mathbb{E}[|e^{\epsilon(\theta_1 q_1+\theta_2 q_2)}|] \leq 1.
% \end{equation*}
For the considered distribution, $\epsilon q_1 \rightarrow \Upsilon_1+ \Upsilon_2$ and $\epsilon q_2 \rightarrow \Upsilon_2$ as $\epsilon \rightarrow 0$ is equivalent to saying that for all $\boldsymbol \theta \in \capTT$, 
\begin{equation*}
    \lim_{\epsilon\rightarrow 0} \mathbb{E}[e^{\epsilon(\theta_1 q_1+\theta_2 q_2)}] = \mathbb{E}[e^{\theta_1 \Upsilon_1+(\theta_1+\theta_2) \Upsilon_2}] = \frac{1}{(1-\theta_1)\big(1- \frac{\theta_1+\theta_2}{2\gamma}\big)}.
\end{equation*}
Under the condition $\mu_1 = \mu_2$, the functional equation given in Eq. \eqref{eq: n_sys_mgf_eq} is given by 
 \begin{align}
    L(\boldsymbol \theta) \big[ (1-2\gamma)\theta_2 & +\theta_2^2 -\theta_1+\theta_1^2\big] +(\theta_1-\theta_2) M_1(\boldsymbol \theta) +\frac{1}{\mu_1}\theta_2 M_2(\boldsymbol \theta)   =0,
\end{align}
Now, if we choose,
\begin{align*}
  L(\boldsymbol \theta) = \frac{1}{(1-\theta_1)\big(1- \frac{\theta_1+\theta_2}{2\gamma}\big)},&&
  M_1(\boldsymbol \theta) = \frac{1}{1- \frac{\theta_1+\theta_2}{2\gamma}}, &&
  M_2(\boldsymbol \theta) = \frac{2\gamma\mu_1}{1-\theta_1},
\end{align*}
% Note that $L(\boldsymbol \theta)$ chosen above is the Laplace transform of the joint distribution of $\Tilde{q}_1$ and $\Tilde{q}_2$ in Theorem \ref{thm: n_sys_distribution}. 
then,
\begin{align}
    L(\boldsymbol \theta) &\big[ (1-2\gamma)\theta_2  +\theta_2^2 -\theta_1+\theta_1^2\big] +(\theta_1-\theta_2) M_1(\boldsymbol \theta) +\frac{1}{\mu_1}\theta_2 M_2(\boldsymbol \theta) \nonumber\\
    &= \frac{(1-2\gamma)\theta_2  +\theta_2^2 -\theta_1+\theta_1^2}{(1-\theta_1)\big(1- \frac{\theta_1+\theta_2}{2\gamma}\big)}+\frac{\theta_1-\theta_2}{1- \frac{\theta_1+\theta_2}{2\gamma}} + \frac{2\gamma\theta_2}{1-\theta_1}\nonumber\allowdisplaybreaks\\
    & = \frac{1}{(1-\theta_1)\big(1- \frac{\theta_1+\theta_2}{2\gamma}\big)} \left[  (1-2\gamma)\theta_2  +\theta_2^2 -\theta_1+\theta_1^2 + (1-\theta_1)(\theta_1-\theta_2) + 2\gamma\theta_2\left(1- \frac{\theta_1+\theta_2}{2\gamma}\right)\right]\nonumber\allowdisplaybreaks\\
    & = 0.
\end{align}
Thus, for the chosen solution $(L(\boldsymbol \theta) , M_1(\boldsymbol \theta), M_2(\boldsymbol \theta))$, the functional equation in Eq. \eqref{eq: n_sys_mgf_eq} is satisfied. And from Lemma \ref{lem: n_sys_uniqueness}, we know that there is a unique solution to the Eq. \eqref{eq: n_sys_mgf_eq}. Now, the result follows by using Lemma \ref{lem: laplace_convergence}.  \hfill $\blacksquare$
% \begin{align*}
%   L(\boldsymbol \theta) &= \lim_{\epsilon\rightarrow 0} \mathbb{E}[e^{\epsilon(\theta_1 q_1+\theta_2 q_2)}] = \lim_{\epsilon\rightarrow 0} \mathbb{E}[e^{\epsilon(\theta_1 (q_1-q_2)+(\theta_1+\theta_2) q_2)}] = \frac{1}{(1-\theta_1)\big(1- \frac{\theta_1+\theta_2}{2\gamma}\big)},\\
%   M_1(\boldsymbol \theta) &= \lim_{\epsilon\rightarrow 0} \mathbb{E}[e^{\epsilon(\theta_1 q_1+\theta_2 q_2)}| q_1 \leq q_2] = \lim_{\epsilon\rightarrow 0} \mathbb{E}[e^{\epsilon(\theta_1 (q_1-q_2)+(\theta_1+\theta_2) q_2)}|q_1 \leq q_2] =\frac{1}{1- \frac{\theta_1+\theta_2}{2\gamma}}, \\
%   M_2(\boldsymbol \theta) &= \mathbb{E}[e^{\epsilon \theta_1 q_1}|q_2 =0] =\mathbb{E}[e^{\epsilon \theta_1 (q_1-q_2)}|q_2 =0]= \frac{1}{1-\_1}
% \end{align*}
\endproof

\section{Proof of Lemma \ref{lem: functional_uniqueness}}
\label{app: functional_uniqueness}

\proof{\textit{Proof of Lemma \ref{lem: functional_uniqueness}}.}
Recall that the functional equation is given by 
\begin{equation*}
    \gamma(\boldsymbol \psi) \Phi (\boldsymbol \psi) + \gamma_1(\boldsymbol \psi) \Phi_1 (\psi_2) +\gamma_2(\boldsymbol \psi) \Phi_2 (\psi_1) = 0,
\end{equation*}
for all $\boldsymbol \psi \in \boldsymbol\Psi = \{\mathbf x \in \mathbb{C}^2:  Re( \mathbf x ) \leq \mathbf 0_2\}$, where $\Phi (\boldsymbol \psi), \Phi_1 (\psi_2)$ and $\Phi_2 (\psi_1)$ are analytic functions over the domain $\boldsymbol\Psi$, and $\gamma(\boldsymbol \psi), \gamma_1(\boldsymbol \psi)$ and $\gamma_2(\boldsymbol \psi)$ are given by
\begin{align*}
    \gamma(\boldsymbol \psi) &= \alpha_1 \psi_1 + \alpha_2 \psi_2 + \frac{1}{2} (\sigma_{11} \psi_1^2 + 2\sigma_{12} \psi_1\psi_2 + \sigma_{22}\psi_2^2),\\
    \gamma_1 (\boldsymbol \psi) &= r_{11} \psi_1 + r_{21}\psi_2,\\
     \gamma_2 (\boldsymbol \psi)  &= r_{12} \psi_1 + r_{22}\psi_2.
\end{align*}

Using this terminology, we define the algebraic functions $\Psi_1^{\pm}(\psi_2)$ and $\Psi_2^{\pm} (\psi_1)$ such that  
\begin{align*}
    \gamma(\psi_1,\Psi_2^{\pm} (\phi_1)) = \gamma(\Psi_1^{\pm} (\psi_2),\psi_2) = 0.
\end{align*}
By solving these equations, we get,
\begin{align*}
    \Psi_1^{\pm}(\psi_2) &= \frac{-(\sigma_{12}\psi_2 + \alpha_1) \pm \sqrt{\psi_2^2(\sigma_{12}^2 -\sigma_{11}\sigma_{22}) + 2\psi_2(\alpha_1\sigma_{12} -\alpha_2\sigma_{11}) +\alpha_1^2}}{ \sigma_{11}}, \\
    \Psi_2^{\pm}(\psi_1) &= \frac{-(\sigma_{12}\psi_1 + \alpha_2) \pm \sqrt{\psi_1^2(\sigma_{12}^2 -\sigma_{11}\sigma_{22}) + 2\psi_1(\alpha_2\sigma_{12} -\alpha_1\sigma_{22}) +\alpha_2^2}}{ \sigma_{22}}.
\end{align*}
Under the conditions given in Eq. \eqref{eq: funceqconditions}, the polynomials under the square root in the above equations have two zeros, which are real and of opposite signs, given by

\begin{align*}
    \psi_2^{\pm} &= \frac{(\alpha_1\sigma_{12} -\alpha_2\sigma_{11}) \pm \sqrt{(\alpha_1\sigma_{12} -\alpha_2\sigma_{11})^2 + \alpha_1^2 (\sigma_{11}\sigma_{22} - \sigma_{12}^2)}}{\sigma_{11}\sigma_{22}-\sigma_{12}^2   },\\
    \psi_1^{\pm} &= \frac{(\alpha_2\sigma_{12} -\alpha_1\sigma_{22}) \pm \sqrt{(\alpha_2\sigma_{12} -\alpha_1\sigma_{22})^2 + \alpha_2^2 (\sigma_{11}\sigma_{22} - \sigma_{12}^2)}}{\sigma_{11}\sigma_{22}-\sigma_{12}^2}.
\end{align*}
Note that $\psi_1^-<0$ and $\psi_2^-<0$, and similarly, $\psi_1^+>0$ and $\psi_2^+>0$.
Using these notations, we define the curve $\mathcal{R}$ to be
\begin{align*}
    \mathcal{R} = \{\psi_2\in\mathbb C: \gamma(\psi_1,\psi_2) = 0 \text{ and } \psi_1\in (-\infty,\psi_1^-)\}=\Psi_2^{\pm}((-\infty,\psi_1^-)).
\end{align*}
We use $\mathcal{G}_{\mathcal{R}}$ to denote the open domain in $\mathbb C$ containing $0$ and bounded by the curve $\mathcal R$.
Further, $\overline{\mathcal{G}}_{\mathcal{R}} = \mathcal{G}_{\mathcal{R}} \cup \mathcal{R}$ is the closure of the set $\mathcal{G}_{\mathcal{R}}$.

Finally, let $G$ be the function give by
\begin{align*}
    G(\psi_2) = \frac{\gamma_1}{\gamma_2}\big((\Psi_1^-(\psi_2),\psi_2)\big)\frac{\gamma_2}{\gamma_1}\big((\Psi_1^-(\overline \psi_2),\overline \psi_2)\big),
\end{align*}
where $\overline{\psi_2}$ is the complex conjugate of $\psi_2$.

\begin{lemma} \label{lem: meroextension}
The Laplace transform $\Phi_1(\psi_2)$ can be extended meromorphically to the open and simply connected set
\begin{align} \label{eq: continuation_domain}
    \big\{ \psi_2\in \mathbb C \backslash (\psi_2^-,\infty): Re(\psi_2)\leq 0 \text{ or } Re(\Psi_1^-(\psi_2))<0 \big\},
 \end{align}
by mean of the formula
\begin{align} \label{eq: continuation_formula}
    \Phi_1(\psi_2) = -\frac{\gamma_2}{\gamma_1}\big((\Psi_1^-(\psi_2),\psi_2)\big) \Phi_2(\Psi_1^-(\psi_2)).
\end{align}
\end{lemma}

Lemma \ref{lem: meroextension} is provided in \cite[Lemma 3]{franceschi2019integral}. Further, as given in \cite[Lemma 5]{franceschi2019integral}, the set \[\big\{ \psi_2\in \mathbb C \backslash (\psi_2^-,\infty): Re(\psi_2)\leq 0 \text{ or } Re(\Psi_1^-(\psi_2))<0 \big\}\] strictly contains the set $\overline{\mathcal{G}}_{\mathcal{R}}$. Thus, Lemma \ref{lem: meroextension} allows us to meromorphically extend the function $\Phi_1(\psi_2)$ to the set $\overline{\mathcal{G}}_{\mathcal{R}}$. 

By the continuation formula in Eq. \eqref{eq: continuation_formula}, the only pole for $\Psi_1$ in the domain given in Eq. \eqref{eq: continuation_domain} comes when the denominator $\gamma_1$ is zero. Suppose $p$ be the non-zero point such that 
\begin{align*}
    \gamma_1\big((\Psi_1^-(p),p)\big) =0.
\end{align*}
Simply solving the above equation gives us
\begin{align*}
    p = \frac{2r_{11}(\alpha_1r_{21}-\alpha_2r_{11})}{r_{11}^2 \sigma_{22} - 2r_{11}r_{21}\sigma_{12} + r_{21}^2\sigma_{11}}.
\end{align*}
Also, the pole $p$ of $\Psi_1$ is simple by the definition of $\gamma_1$. 

Now, we are equipped to write the Carleman Boundary Value Problem (BVP) for the function $\Phi_1$.

\begin{proposition}[Carleman BVP with shift]
\label{prop: carleman_bvp}
The function $\Phi_1$ is 
\begin{itemize}
    \item is meromorphic on $\mathcal{G}_{\mathcal{R}}$ with 
    \begin{itemize}
        \item without pole on $\mathcal{G}_{\mathcal{R}}$ if $\gamma_1(\psi_1^-, \Psi_2^-(\psi_1^-))<0$,
        \item with a single pole on $\mathcal{G}_{\mathcal{R}}$ at $p$ of order one if $\gamma_1(\psi_1^-, \Psi_2^-(\psi_1^-))>0$,
        \item without pole on $\mathcal{G}_{\mathcal{R}}$ and with a single pole of order one on the boundary $\mathcal{R}$ of $\mathcal{G}_{\mathcal{R}}$, at $p = \Psi_2^-(\psi_1^-)$ if $\gamma_1(\psi_1^-, \Psi_2^-(\psi_1^-))=0$,
    \end{itemize}
    \item is continuous on $\overline{\mathcal{G}}_{\mathcal{R}}\backslash\{p\}$ and bounded on infinity,
    \item satisfies the boundary condition
    \begin{align*}
        \Phi_1(\overline\psi_2) = G(\psi_2) \Phi_1(\psi_2), \ \ \ \ \forall \psi_2 \in \mathcal{R}.
    \end{align*}
\end{itemize}
\end{proposition}

The BVP presented in Proposition \ref{prop: carleman_bvp} is solvable and has a unique solution. A method of solving the BVP in Proposition \ref{prop: carleman_bvp} is presented in \cite{franceschi2019integral}. This implies that there is unique solution $\Phi_1(\psi_2)$ for the functional equation provided in Lemma \ref{lem: functional_uniqueness} for all $\psi_2$ such that $Re(\psi_2)\leq 0$. Similarly, the function $\Phi_2$ is also unique for all $\psi_1$ such that $Re(\psi_1)\leq 0$. This in turn implies that there is a unique $\Phi(\boldsymbol \psi)$ that satisfies the functional equation given in Lemma \ref{lem: functional_uniqueness} for all $\boldsymbol \psi \in \boldsymbol \Psi$. \hfill $\blacksquare$
\endproof

In order to provide more understanding of the proof of Lemma \ref{lem: functional_uniqueness}, we consider a simpler case, and reiterate the proof and the solution of the corresponding Carleman BVP.

\subsection{Simpler case: Orthogonal reflection and diagonal variance matrix}

By orthogonal reflection, we mean that $r_{11}=r_{22}=1$ and $r_{21}=r_{12} =0$; and diagonal variance matrix means $\sigma_{12}=0$. Under these conditions, the argument becomes much simpler as follows. 

Using $r_{11}=r_{22}=1$ and $r_{21}=r_{12} =0$, we have that $\gamma_1(\boldsymbol\psi) = \psi_1$ and $\gamma_2(\boldsymbol\psi) = \psi_2$. We set $g_{1}(\psi_2) =\frac{1}{\psi_2} \Phi_1(\psi_2)$ and $g_{2}(\psi_1) =\frac{1}{\psi_1} \Phi_2(\psi_1)$.
Then, Proposition \ref{prop: carleman_bvp} simplifies to following.
\begin{corollary}
\label{cor: simpler_bvp}
The function $g_1$ satisfies the following BVP:
\begin{enumerate}
    \item $g_1$ is meromorphic on $\mathcal{G}_{\mathcal{R}}$, with a single pole at $0$, or order 1 and residue $\Phi_1(0)$,
    \item $g_1$ is continuous on $\overline{\mathcal{G}}_{\mathcal{R}}\backslash\{0\}$ and vanishes on infinity,
    \item $g_1$ satisfies the boundary condition
    \begin{align*}
        g_1(\overline\psi_2) =  g_1(\psi_2), \ \ \ \ \forall \psi_2 \in \mathcal{R}.
    \end{align*} 
\end{enumerate}
\end{corollary}

The proof of Corollary \ref{cor: simpler_bvp} is provided in \cite[Proposition 3.2]{franceschi2017tutte}. We have also provided the arguments below for clarification.

\proof{\textit{Proof of Corollary \ref{cor: simpler_bvp}}.} 
When the element $\psi_2$ satisfies $Re(\psi_2)\leq0$, the argument in Part 1 and Part 2 follows by using the analytic properties of $\Phi_1(\psi_2)$ when $Re(\psi_2)<0$. In the domain of $\mathcal{G}_{\mathcal{R}}$ when $Re(\psi_2)\geq 0$, we use the continuation formula provided in Eq. \eqref{eq: continuation_formula} and the result follows, as $\mathcal{G}_{\mathcal{R}}$ is contained in the domain given in Eq. \eqref{eq: continuation_domain}. 

Further, for any $\psi_1 \in (-\infty,\psi_1^-)$, we have that 
\begin{equation*}
    \psi_1 \Phi_1 (\Psi_2^{\pm}(\psi_1)) + \Psi_2^{\pm}(\psi_1) \Phi_2(\psi_1) =0.
\end{equation*}
This gives us that,
\begin{align*}
    g_1 (\Psi_2^{-}(\psi_1)) + g_2(\psi_1)=g_1 (\Psi_2^{+}(\psi_1)) + g_2(\psi_1)  =0.
\end{align*}
Now, as $\Psi_2^{-}(\psi_1)$ and $\Psi_2^{+}(\psi_1)$ are complex conjugates of each other, we get the condition in Part 3. \hfill $\blacksquare$
\endproof

\begin{lemma}[Invariant lemma]
\label{lem: inv_lem}
The problem of finding functions $f$ such that 
\begin{enumerate}
    \item $f$ is analytic in $\mathcal{G}_{\mathcal{R}}$ and continuous in $\mathcal{G}_{\mathcal{R}}$,
    \item $f$ satisfies the boundary condition $f(\overline\psi_2) =  f(\psi_2), \ \forall \psi_2 \in \mathcal{R}$,
\end{enumerate}
does not have non-trivial solutions in the class of functions f vanishing at infinity.
\end{lemma}

\begin{lemma}[Conformal gluing function]
\label{lem: gluing_func}
Suppose $w(\psi)$ is given by 
\begin{equation*}
    w(\psi) = 2 \Big( \frac{2\psi - (\psi_2^++ \psi_2^-)}{\psi_2^+- \psi_2^-} \Big)^2 -1.
\end{equation*}
Then, $w$ satisfies
\begin{enumerate}
    \item $w$ is analytic in $\mathcal{G}_{\mathcal{R}}$, continuous in $\mathcal{G}_{\mathcal{R}}$ and unbounded at infinity,
    \item w is injective in $\mathcal{G}_{\mathcal{R}}$,
    \item $w$ satisfies the boundary condition $w(\overline\psi_2) =  w(\psi_2), \ \forall \psi_2 \in \mathcal{R}$.
\end{enumerate}

\end{lemma}
Lemma \ref{lem: gluing_func} is presented in \cite[Lemma 3.4]{franceschi2017tutte}. The function $w$ is called the conformal gluing function. In this section, we have considered a simple class of functional equations for which the corresponding conformal gluing function takes a simple form as given in Lemma \ref{lem: gluing_func}. For a more general class of functional equations, the conformal gluing functions take a more complicated form, which is beyond the scope of this paper.

Next, we introduce the function $f$ to be
\begin{align*}
    f(\psi_2) = g_1(\psi_2) - \Phi_1(0) \frac{w'(0)}{ w(\psi_2) - w(0)}
\end{align*}
The above function $f$ satisfies the conditions in Lemma \ref{lem: inv_lem}. The only possible pole for $f$ is at $0$ since $g_1$ has a unique pole at $0$ and $w$ is injective, so $w(\psi_2)-w(0)=0$ has only one solution. 
However, by construction of $w$, the residue of the function $f $ at $0$ is $0$, which implies that $0$ is a removable singularity. This shows that $f$ satisfies Part 1 of Lemma \ref{lem: inv_lem}. 
The function $f$ also satisfies Part 2 of  Lemma \ref{lem: inv_lem} as both $g_1$ and $w$ satisfies the boundary condition in Part 2. Furthermore, $f$ vanishes at infinity as $g_1$ vanishes at infinity, and $w$ is unbounded at infinity. Thus, $f$ satisfies all the conditions in Lemma \ref{lem: inv_lem}. As a conclusion of Lemma \ref{lem: inv_lem}, we get that $f =0$. 
Thus, we get the unique solution,
\begin{align*}
    \Phi_1(\psi_2) = \psi_2 g_1(\psi_2) = \psi_2 \Phi_1(0) \frac{w'(0)}{ w(\psi_2) - w(0)}.
\end{align*}
By using the definition of the function $w$, we get 
\begin{align*}
    \Phi_1(\psi_2) = -\frac{2\alpha_1\alpha_2}{\sigma_{22}} \frac{1}{\psi_2 + \frac{2\alpha_2}{\sigma_{22}}}.
\end{align*}
Similarly, one can solve for $\Phi_2(\psi_1) $ and then $\Phi(\boldsymbol \psi)$.

\end{APPENDICES}

\end{document}